\documentclass[11pt, a4paper, reqno]{amsart}

\usepackage[utf8]{inputenc}
\usepackage[T1]{fontenc}
\usepackage{amsmath, mathtools, amsthm}
\usepackage{amssymb,url,xspace}
\usepackage[mathscr]{eucal}
\usepackage{graphicx}
\usepackage[outercaption]{sidecap} 
\usepackage{subfigure}
\usepackage{bm}
\usepackage[dvipsnames]{xcolor}
\usepackage{cases}
\usepackage[hidelinks]{hyperref}
\usepackage{cleveref}
\usepackage{tikz}
\usepackage{verbatim}
\usetikzlibrary{matrix}
\usepackage{xfrac}

\usepackage{enumitem}
\usepackage{multirow}
\usepackage[ruled,vlined]{algorithm2e}
\usetikzlibrary{positioning,decorations.pathreplacing}
\usepackage[square]{natbib} 

\definecolor{dukeblue}{rgb}{0.0, 0.0, 0.61}
\definecolor{darkcandyapplered}{rgb}{0.64, 0.0, 0.0}

\hypersetup{
	colorlinks,
	citecolor=dukeblue,
	filecolor=black,
	linkcolor=dukeblue,
	urlcolor=black
}
\usepackage{empheq}

\usepackage[svgnames,pdf]{pstricks}
\usepackage{anysize}
\marginsize{3.25cm}{3.25cm}{3.25cm}{3.25cm}

\newcommand{\one}{\mathbf{1}}
\newcommand{\diff}{\, \mathrm{d}}

\newcommand{\del}{\partial}
\newcommand{\N}{\mathbb{N}}

\newcommand{\R}{\mathbb{R}}

\def\BV{{\mathrm{BV}}}

\def\aug{\mathrm{aug}}
\def\hid{\mathrm{hid}}

\def\img{\mathrm{img}}
\def\aux{\mathrm{aux}}
\def\tr{\top}
\def\ini{\mathrm{in}}
\def\Lip{\mathrm{Lip}}
\def\loss{\mathrm{loss}}
\DeclareMathOperator*{\argmin}{arg\,min}

\def\*#1{\mathbf{#1}}

\theoremstyle{plain}
\numberwithin{equation}{section}
\newtheorem{remark}{Remark}
\newtheorem{definition}{Definition}[section]
\newtheorem{proposition}{Proposition}[section]
\newtheorem{theorem}{Theorem}[section]
\newtheorem{lemma}{Lemma}[section]

\newtheorem{example}{Example}[section]
\newtheorem{assumption}{Assumption}

	\title 
	[Large-time asymptotics in Deep Learning]{Large-time asymptotics in Deep Learning}
		
		\author{Carlos Esteve-Yag\"ue}

	\author{Borjan Geshkovski}
\author{Dario Pighin}
	
	\address {\textbf{\textup{Carlos Esteve-Yag\"ue, Borjan Geshkovski, Dario Pighin}}
	\newline \indent
	{\textup{Departamento de Matem\'aticas}
	\newline \indent
	\textup{Universidad Aut\'onoma de Madrid}}
	\newline \indent
	\textup{28049 Madrid, Spain} 
	\newline \indent \hspace{2.5cm} \textit{and}	
	\newline \indent 
	\textup{Chair of Computational Mathematics} \hspace{1.28cm}
	\newline \indent
	\textup{Fundaci\'on Deusto}
	\newline \indent
	\textup{Av. de las Universidades, 24}
	\newline \indent
	\textup{48007 Bilbao, Basque Country, Spain} 
	}
	\email{\href{mailto:borjan.geshkovski@uam.es}{\textcolor{dukeblue}{\texttt{\{carlos.esteve, borjan.geshkovski, dario.pighin\}@uam.es}} }}
	
	\author{Enrique Zuazua}
	\address{\textbf{\textup{Enrique Zuazua}}
		\newline \indent
		\textup{Chair in Applied Analysis, Alexander von Humboldt-Professorship}
		\newline \indent
		\textup{Department of Mathematics} \newline \indent
		\textup{Friedrich-Alexander-Universit\"at Erlangen-N\"urnberg}
		\newline \indent
		\textup{91058 Erlangen, Germany}
		\newline \indent \hspace{2.5cm} \textit{and} \newline \indent
	\textup{Chair of Computational Mathematics} 
	\newline \indent
	\textup{Fundaci\'on Deusto}
	\newline \indent
	\textup{Av. de las Universidades, 24}
	\newline \indent
	\textup{48007 Bilbao, Basque Country, Spain}
	\newline \indent \hspace{2.5cm} \textit{and} \newline \indent
	\textup{Departamento de Matemáticas} \newline \indent
		\textup{Universidad Autónoma de Madrid}
	\newline \indent
	\textup{28049 Madrid, Spain}
	}
	\email{\href{mailto:enrique.zuazua@fau.de}{\textcolor{dukeblue}{\texttt{enrique.zuazua@fau.de}}}}

\date{\today}

\begin{document}
	
		\begin{abstract}
		We consider the neural ODE perspective of supervised learning and study the impact of the final time $T$ (which may indicate the depth of a corresponding ResNet) in training. 
		For the classical $L^2$--regularized empirical risk minimization problem, 
		whenever the neural ODE dynamics are homogeneous with respect to the parameters, we show that the training error is at most of the order $\mathcal{O}\left(\frac{1}{T}\right)$. 
		Furthermore, if the loss inducing the empirical risk attains its minimum, the optimal parameters converge to minimal $L^2$--norm parameters which interpolate the dataset.
		By a natural scaling between $T$ and the regularization hyperparameter $\lambda$ we obtain the same results when $\lambda\searrow0$ and $T$ is fixed. This allows us to stipulate generalization properties in the overparametrized regime, now seen from the large depth, neural ODE perspective.
		To enhance the polynomial decay, inspired by turnpike theory in optimal control, we propose a learning problem with an additional integral regularization term of the neural ODE trajectory over $[0,T]$. 
	In the setting of $\ell^p$--distance losses, we prove that both the training error and the optimal parameters are at most of the order $\mathcal{O}\left(e^{-\mu t}\right)$ in any $t\in[0,T]$.
		   The aforementioned stability estimates are also shown for continuous space-time neural networks, taking the form of nonlinear integro-differential equations. By using a time-dependent moving grid for discretizing the spatial variable, we demonstrate that these equations provide a framework for addressing ResNets with variable widths. 
		\end{abstract}
			
	\maketitle	
	
	\setcounter{tocdepth}{1}
	
	\tableofcontents
	
	
	{\small{\href{https://mathscinet.ams.org/msc/msc2010.html}{{\bf 	\color{dukeblue}{AMS Subject Classification}}}}. 49J15; 49M15; 49J20; 49K20; 93C20; 49N05.}

\section{Introduction}

	Modern supervised learning addresses the problem of predicting from data, which roughly consists in approximating an unknown function $f:\mathcal{X}\to\mathcal{Y}$ from $N$ known but possibly noisy samples $\left\{\vec{x}_i, \vec{y}_i\right\}_{i=1}^N\subset\mathcal{X}\times\mathcal{Y}$. 
	Depending on the nature of the labels $\vec{y}_i$, one distinguishes two types of supervised learning tasks, namely that of \emph{classification} (labels take values in a finite set of $m$ classes, e.g. $\mathcal{Y}:=\{1,\ldots, m\}$) and \emph{regression} (labels take continuous values in $\mathcal{Y}\subset\R^m$). 
	In many applications, the dimension $d$ of each sample $\vec{x}_i\in\mathcal{X}\subset\R^d$ may be big compared to the number/dimension $m$ of the labels -- in image classification for instance, a sample of the ImageNet dataset \citep{krizhevsky2012imagenet},  which has $m=1000$ classes, is an element of $\R^{65536}$.
	
	A plethora of methods for finding $f(\cdot)$ efficiently with theoretical and empirical guarantees have been developed and investigated in the machine learning literature in recent decades. 
	Prominent examples, to name a few, include linear parametric methods (e.g. linear or logistic regression), kernel-based methods (e.g. support vector machines), tree-based methods (e.g. decision trees) and so on. 
	We refer to \citep{goodfellow2016deep} for a comprehensive presentation of these topics. 
	\smallskip
	
	Deep neural networks are parametrized computational architectures which propagate each individual sample of the input data $\{\vec{x}_i\}_{i=1}^N\subset\mathcal{X}\subset\R^{d}$ across a sequence of linear parametric operators and simple nonlinearities. 
	The so-called \emph{residual neural networks} (ResNets, \citep{he2016deep}) may, in the simplest case, be cast as schemes of the mould
	\begin{equation} \label{eq: 1.1}
	\begin{dcases}
	\*x_i^{k+1} = \*x_i^k + \sigma\left(w^k \*x_i^k + b^k\right) &\text{ for } k \in \{0, \ldots, N_{\text{layers}}-1\}\\
	\*x_i^0 = \vec{x}_i \in \R^d
	\end{dcases}
	\end{equation}
	for all $i \in [N]$, where we set $[N]:= \{1, \ldots, N\}$.  
	The unknowns are the states $\*x_i^k\in \R^d$ for any $i\in[N]$,  $\sigma$ is an explicit scalar, Lipschitz continuous nonlinear function defined component-wise in \eqref{eq: 1.1}, $\left\{w^k, b^k\right\}_{k=0}^{N_{\text{layers}}-1}$ are optimizable parameters (controls) with $w^k \in \R^{d\times d}$ -- called weights, and $b^k \in \R^d$ -- called biases, and $N_{\text{layers}}\geqslant 1$ designates the number of layers referred to as the depth. 
	The training process consists in finding optimal parameters steering all of the network outputs $\*x_i^{N_{\text{layers}}}$ as close as possible to the corresponding labels $\vec{y}_i$ by solving
	\begin{equation*}
	\min_{\left\{w^k, b^k\right\}_{k=0}^{N_{\text{layers}}-1}} \frac{1}{N} \sum_{i=1}^N \loss\left(P\*x_i^{N_{\text{layers}}}, \vec{y}_i\right),
	\end{equation*}
	whilst guaranteeing reliable performance on unseen data (ensuring \emph{generalization}). 
	Here $\loss(\cdot,\cdot)$ is a given continuous and nonnegative function which  differs depending on the task in hand -- for instance $\loss(x,y) := \|x-y\|^p_{\ell^p}$ for $p\in\{1,2\}$ is commonly used for regression tasks, while $\loss(x,y) = \log(1+\exp(-yx))$ may be used for binary classification, namely when $\vec{y}_i\in\{-1,1\}$ (we refer to \Cref{eq: softmax} for more general settings). 
	On the other hand, $P:\R^d\to\R^m$ is an affine map whose coefficients in practice are part of the optimizable parameters. 
	In our work, we shall assume that $P$ is given and specified on a case-by-case basis.
	\smallskip
	
	Due to the inherent dynamical systems nature of ResNets, several recent works have aimed at studying an associated continuous-time formulation in some detail, a trend started with the works \citep{weinan2017proposal, haber2017stable} (see also \citep{lu2018beyond, ruthotto2019deep}). 
	This perspective is motivated by the simple observation that for any $i \in [N]$ and for $T>0$, \eqref{eq: 1.1} is roughly the forward Euler scheme for the neural ordinary differential equation (neural ODE)
	\begin{equation} \label{eq: 1.2}
	\begin{dcases}
	\dot{\*x}_i(t) = \sigma(w(t)\*x_i(t)+b(t)) & \text{ for } t \in (0, T) \\
	\*x_i(0) = \vec{x}_i \in \R^d.
	\end{dcases}
	\end{equation} 
	The continuous-time, neural ODE formalism of deep learning has been used to great effect in applications -- for instance, by using adaptive ODE solvers \citep{chen2018neural, dupont2019augmented, queiruga2020continuous}, symplectic  and multigrid methods \citep{celledoni2020structure, gunther2020layer}, or indirect training algorithms based on the Pontryagin Maximum Principle \citep{li2017maximum, benning2019deep} --, and also for generative modeling through normalizing flows \citep{grathwohl2018ffjord, chen2019residual}.
	We emphasize that the origins of continuous-time supervised learning go back to the 1980s -- the neural network model proposed in \citep{hopfield1982neural} is a differential equation, whereas in \citep{lecun1988theoretical} back-propagation is connected to the adjoint method arising in optimal control.
	Related works include studies on identification of the weights from data \citep{albertini1993neural, albertini1993uniqueness} and controllability of continuous-time recurrent networks \citep{sontag1997complete, sontag1999further}.
	\smallskip
	
	The role of the final time horizon $T>0$, which plays a key role in the control of dynamical systems, has not been analyzed in the context of supervised learning problems via models such as \eqref{eq: 1.2}. 
	Now note once again that in the ResNet \eqref{eq: 1.1} the time-step $\Delta t=\sfrac{T}{N_\text{layers}}$ is fixed (equal to $1$), and each time instance of a forward Euler discretization to \eqref{eq: 1.2} would represent a different layer of \eqref{eq: 1.1}. Hence, whenever the time-step $\Delta t=\sfrac{T}{N_{\text{layers}}}$ is fixed (or goes to zero when $T$ is increased), the time horizon $T$ in \eqref{eq: 1.2} serves as an indicator of the number of layers $N_{\text{layers}}$ in the ResNet \eqref{eq: 1.1}. 
	Thus, a good knowledge of the behavior of the learning problem and the neural ODE flow over longer time horizons is desirable in view of understanding approximation and generalization properties. 
	In this work, we aim to bridge this gap by proposing several insights stemming from an analysis of the role of the time horizon $T$.
	
	\subsection{Our contributions.} We shall focus this presentation on the neural ODE \eqref{eq: 1.2}, but our results also hold for other systems, as seen in the respective statements.
		
	\begin{enumerate}
	\item[\textbf{1.}] 
	We first consider the classical regularized empirical risk minimization problem
	\begin{align} \label{functional no penalized intro}
	&\inf_{\substack{[w,b]\in \mathscr{H}(0,T; \R^{d_u}) \\ \*x_i(\cdot) \text{ solves } \eqref{eq: 1.2}}} \underbrace{\frac{1}{N} \sum_{i=1}^N \loss\big(P\*x_i(T), \, \vec{y}_i\big)}_{:= \mathscr{E}(\*x(T))}
	 + \underbrace{\lambda\Big\|[w,b] \Big\|^2_{\mathscr{H}(0,T; \R^{d_u})}}_{\text{regularization}}
	\end{align}
	where\footnote{Here $H^1(0,T;\R^{d_u})$ denotes the Sobolev space of square integrable functions from $(0,T)$ to $\R^{d_u}$ with square integrable weak derivatives (see \Cref{sec: notation}). We consider $H^1$--regularization in the setting of \eqref{eq: 1.2} to ensure the existence of minimizers, see \Cref{rem: sobolev.regularization}.}  $\mathscr{H}$ is either $L^2$ or $H^1$.
	In \Cref{thm: no.running}, we show that when $\loss(\cdot,\cdot)$ and the affine map $P:\R^d\to\R^m$ are such that the minimum of $\mathscr{E}$ (equal to zero) is attained and when the activation function $\sigma$ in \eqref{eq: 1.2} is $1$--homogeneous, the training error $\mathscr{E}(\*x_T(T))$ of the collection $\*x_T =\{\*x_{T,i}\}_{i\in[N]}$ of solutions to \eqref{eq: 1.2} corresponding to any solution $\left[w_T, b_T\right]$ to the minimization problem \eqref{functional no penalized intro}, is at most of the order $\mathcal{O}\left(\frac{1}{T}\right)$, whilst the optimal parameters $\left[w_T, b_T\right]$ converge, on a suitable time-scale, to a solution $\left[w^\ast, b^\ast\right]$ of
	\begin{align} \label{eq: interpolation.intro}
	&\inf_{\substack{[w,b]\in \mathscr{H}(0,1;\R^{d_u}) \\ \*x_i(\cdot) \text{ solves } \eqref{eq: 1.2} \text{ in } [0,1] \\ \text{ and } \\ \scalebox{0.7}{$ \displaystyle \mathscr{E}(\*x(1))=0$}  }} \Big\|[w,b] \Big\|^2_{\mathscr{H}(0,1;\R^{d_u})} \hspace{0.25cm}
	\end{align}
	when $T\longrightarrow\infty$. 
	
	Let us put the above result into context.
	For neural ODEs for which $L^2$--regularization suffices, we remark that $T\longrightarrow\infty$ is equivalent to $\lambda\searrow0$. 
	The latter is the regularization path limit, studied in the literature for linear models and multi-layer perceptrons (but not for more compound neural ODE models), where the solutions obtained in the limit can be shown to satisfy desirable generalization properties (see \Cref{sec: related.work}).
		
	Using similar arguments as when $T\longrightarrow\infty$, in \Cref{thm: no.running.lambda} we obtain the same conclusions when $\lambda\searrow0$ and $T$ is fixed. 
	Consequently, \Cref{thm: no.running} stipulates generalization properties -- namely, optimizing with $T\gg1$, which may be interpreted as a larger depth for ResNets, has the practically desirable effect of making the training error close to zero, but by means of parameters with the smallest amplitude and thus trajectories with the least oscillations. 
	\smallskip	
	
	\item[\textbf{2.}] In the setting of losses for which $\mathscr{E}$ does not attain its minimum equal to zero, occurring in many classification contexts (for instance $\loss(Px,y)=\log\left(1+e^{-yPx}\right)$ when $\vec{y}_i \in \mathcal{Y}=\{-1,1\}$ or multi-label tasks where $\vec{y}_i \in[m]$ for $m\geqslant 2$ via cross-entropy loss), we show that the training error $\mathscr{E}(\*x_T(T))$ is at most of the order $\mathcal{O}\left(\frac{1}{T^{\alpha}}\right)$
	for all $\alpha\in\left(0,1\right)$ (see \Cref{thm: thm.classification.lambda}). Analog results hold when $T>0$ is fixed and $\lambda\searrow0$ (see \Cref{thm: thm.classification.lambda.reg}).
	 \smallskip

	\item[\textbf{3.}] 
	To enhance the polynomial stability of the training error $\mathscr{E}(\*x_T(T))$ with respect to $T$, we also consider the augmented empirical risk minimization problem\footnote{As for \eqref{functional no penalized intro}, we consider $\mathscr{H}=L^2$ when existence of minimizers can be ensured; however, unlike for \eqref{functional no penalized intro}, in the setting of \eqref{eq: 1.2}, we will sometimes consider a $\BV$ regularization (see \Cref{sec: notation} for a definition) rather than the stronger $H^1$ regularization for technical reasons (see \Cref{rem: BV.need}).}
	\begin{equation} \label{eq: turnpike.functional.intro}
	\inf_{\substack{[w,b]\in \mathscr{H}(0,T;\R^{d_u}) \\ \*x_i(\cdot) \text{ solves } \eqref{eq: 1.2}}} \mathscr{E}(\*x(T)) + \frac{1}{N}\sum_{i=1}^N\int_0^T \left\|\*x_i(t)-\overline{\*x}_i\right\|^2 \diff t + \lambda \Big\|[w,b] \Big\|_{\mathscr{H}(0,T;\R^{d_u})}^2,
	\end{equation}
	where $P:\R^d\to\R^m$ appearing in $\mathscr{E}$ is Lipschitz and surjective, $\loss$ is an $\ell^p$--distance, while the targets $\overline{\*x}_i \in P^{-1}(\{\vec{y}_i\})$ for all $i \in [N]$ are arbitrary but given. 
	Under a specific controllability assumption but without any differentiability assumptions on the activation function $\sigma$ or smallness assumptions on the dataset, in \Cref{thm: turnpike.P} (see also \Cref{thm: turnpike.BV}) we show that optimal parameters $[w_T(t),b_T(t)]$ and the training error $\mathscr{E}(\*x_T(t))$ of the corresponding vector $\*x_T(t)$ of solutions to \eqref{eq: 1.1} are at most of the order $\mathcal{O}(e^{-\mu t})$ for any $t\in[0,T]$ and some $\mu>0$ independent of $T$.
	
	This result is in line with \Cref{thm: no.running} and \Cref{thm: thm.classification.lambda}, but with an improved estimate of the time horizon needed to be $\varepsilon$--close to the interpolation regime for any given $\varepsilon>0$. 
	Due to the exponentially small global minimizers, numerical experiments show that the learned flow is simple, which could stipulate possible generalization properties.
	\Cref{thm: turnpike.P} is a manifestation of the so-called \emph{turnpike property}\footnote{The turnpike property is a staple of optimal control problems such as the linear quadratic regulator  (LQR) over large but finite time horizons; it refers to the fact that optimal controls (parameters) and trajectories, in long time intervals, remain close to the optimal stationary controls and trajectories. In the setting of \eqref{eq: turnpike.functional.intro}, due to the structure of the neural ODE, the optimal stationary parameters are zero, and thus the optimal stationary state is the collection $\{\overline{\*x}_i\}_{i\in[N]}$, as reflected in our result.}, well-known in optimal control theory (\citep{porretta2013long, trelat2015turnpike, faulwasser2020turnpike}).
	
	Problem \eqref{eq: turnpike.functional.intro} is motivated by the natural training problem which aims at minimizing the error over every time/layer:
	\begin{equation} \label{eq: turnpike.functional.intro.wanted}
	\inf_{\substack{[w,b]\in \mathscr{H}(0,T;\R^{d_u}) \\ \*x_i(\cdot) \text{ solves } \eqref{eq: 1.2}}} \int_0^T \mathscr{E}(\*x(t)) \diff t + \lambda \Big\|[w,b] \Big\|_{\mathscr{H}(0,T;\R^{d_u})}^2,
	\end{equation}
	where $\loss(\cdot,\cdot)$ appearing in $\mathscr{E}$ is continuous and nonnegative, but otherwise arbitrary.  
	Whilst left without proof, numerical experiments stipulate that a similar decay for the training error, and, combined with \Cref{thm: turnpike.P}, motivate the usage of \eqref{eq: turnpike.functional.intro.wanted} in practice (see \Cref{sec: motivating.problem}).
	\smallskip
	
	\item[\textbf{4.}]
	In \Cref{prop_Nleqd+1} we show that the controllability assumption needed for \Cref{thm: turnpike.P} is satisfied for a subclass of neural ODEs with $C^1$--regular activation functions $\sigma$. We also illustrate the relation between the amplitude of the weights needed to reach the interpolation regime, and the dispersion of the input data (see \Cref{prop_lower_bound}).
	\smallskip
		
	\item[\textbf{5.}] 
	To address variable width architectures motivated by multi-layer perceptrons and convolutional neural networks, in \Cref{sec: resnet.variable} we study a continuous space-time neural network formulation taking the form of a scalar integro-differential equation proposed in \citep{liu2019selection}.
	We show that, by making use of a time-dependent moving grid for discretizing the spatial variable, these equations provide a framework for addressing ResNets with variable widths. Furthermore, in \Cref{thm: non.local.scaling} (resp. \Cref{thm: sigma.inside.penalised.nonlocal}), we show that some of our finite-dimensional conclusions from \Cref{thm: no.running} (resp.  \Cref{thm: turnpike.P}) transfer, under similar assumptions, to the continuous space-time neural networks.
	\end{enumerate}
	
	\subsection{Related work.} \label{sec: related.work}
	
	We note that the impact of the time horizon for the regularized empirical risk minimization problem has been studied from a computational perspective in \citep{effland2020variational}. 
	We also refer the reader to the recent work \citep{faulwasser2021turnpike} for a stability analysis for the augmented supervised learning problem \eqref{eq: turnpike.functional.intro} in the ResNet setting, and to \citep{yague2021sparse} for a rigorous analysis of the problem \eqref{eq: turnpike.functional.intro.wanted} with an $L^1$--regularization of the parameters, where sparsity patterns of the parameters are shown.
	
	In \citep{thorpe2018deep}  (see also \citep{avelin2019neural}), the authors show, via $\Gamma$-convergence arguments, that the optimal control parameters in the discrete-time context converge to those of the continuous-time context when the step-size $\Delta t$ goes to $0$. 
	The latter is interpreted as an infinite layer limit when the final time horizon $T$ in the continuous-time context is fixed (equal to $1$).
	Our results are of different nature. Rather than proving that the discrete-time controls converge to the continuous-time ones, we exhibit the continuous-time, neural ODE representation, for which the final time horizon indicates the number of layers for the associated time-discretization when the time-step $\Delta t$ is fixed, and aim to characterize the impact of $T$ on the optimal parameters and on the training error.
	
	Our results are also related to several questions studied in existing literature.
	
	\smallskip
	\noindent
	\textbf{Universal approximation.}
	On a first note, the asymptotic results presented herein may (heuristically) be interpreted as approximation results in the sense of the universal approximation theory. 
	These are density results for neural networks, and in the simplest cases can be interpreted in terms of the elementary building blocks of measure theory such as the density of simple functions in Lebesgue spaces. 
	The first result in this direction is the seminal work \citep{cybenko}, which indicates that shallow neural networks with increasing width,  i.e. a superposition of sufficiently many dilated and translated sigmoids, may approximate any continuous  function on compact sets. 
	We also refer to \citep{hornik1989, pinkus1999approximation} for an extension to multi-layer neural networks. 
	Our results are somewhat dual to \citep{cybenko} -- therein, to increase the approximation accuracy, the width is allowed to grow, whilst we fix the width and allow the depth to increase. 
	We do note however that we prove approximation properties for the optimal parameters, and for a fixed dataset, unlike what is commonly done in universal approximation theorem, where the parameters are not known explicitly.
	We refer to the thesis \citep{muller2019universal} for results and a comprehensive review of universal approximation results for ResNets, and to the recent works \citep{li2019deep, teshima2020universal, zhang2020approximation} and \citep{ruiz2020universal}, for universal approximation results for neural ODE and for an analysis on the latter's working mechanisms.
	
	\smallskip
	\noindent
	\textbf{Regularization path limit: $\lambda\searrow0$.}
	The regularization path limit 
	$\lambda\searrow0$ for minimizers $[w_\lambda,b_\lambda]$ has also been addressed in the machine learning literature, albeit for more classical models.
	For instance in \citep{rosset2004boosting, rosset2003margin}, the authors study linear logistic regression, and show convergence of the margin to the max-margin as $\lambda\searrow0$, assuming linearly separable data. 
	The max-margin, support vector machine solution (see \Cref{rem: absence.margin.conv} for details) is a special example among all solutions that fit the training data -- another example includes minimal $\ell^2$--norm solution for linear regression (or generally supervised learning tasks in which the loss may attain its minimum). 
	Both these solutions can be shown to ensure generalization by virtue of explicit generalization error estimates \citep{bartlett2002, kakade2009} for linear models or multi-layer perceptrons. 
	This insight could stipulate a likely generalization capacity of our asymptotic limits as $T\longrightarrow\infty$ or $\lambda \searrow0$.
	
	The results of \citep{rosset2004boosting, rosset2003margin} have subsequently been extended in \citep{wei2019regularization} (and some of the references therein) to perceptrons with ReLU activations, where the intrinsic homogeneity of the network is used -- the authors prove an analog result to \Cref{thm: no.running.lambda} when $\lambda\searrow0$ for regression tasks (the loss is a distance) and two-layer perceptrons. Our results can be seen as an extension of the aforementioned insights to the neural ODE context.
\smallskip

\noindent
	\textit{Implicit regularization of gradient methods.}
	When training without explicit regularization (i.e. $\lambda=0$), a common approach in the literature is to resort to algorithm-dependent generalization analysis, where the limit solutions are akin to the limit solutions when $\lambda\searrow0$. Whilst the former is not a question we address in this work, where generally estimates are provided in terms of the number of algorithm iterations rather than the number of layers, we provide a brief overview of the literature for completeness. We refer the reader to the recent review \citep{bartlett2021deep} for a comprehensive presentation of these artifacts. 
	
	The implicit bias of gradient descent (\citep{soudry2018implicit, gunasekar2018implicit}) indicates that in the overparametrized regime, after training a linear model or perceptron with gradient-based methods until zero training error, without requiring any explicit parameter regularization, among the many predictors which overfit on the training dataset, the algorithm selects the one which performs best on the test dataset (e.g. minimal $\ell^2$--norm solution or max-margin solution). 
	Thus, even-though the approach is significantly different from that when $\lambda\searrow0$, the asymptotic limits oftentimes coincide.
	 
	   Recent works have shown that gradient descent can allow overparametrized multi-layer networks to attain arbitrarily low training error on fairly general datasets (\citep{du2019gradient, allen2019learning, allen2019convergence}), and find minimum-norm/maximum-margin solutions that fit the data in the settings of logistic regression, deep linear networks, and symmetric matrix factorization (\citep{gunasekar2018implicit, soudry2018implicit, ji2018risk, ma2020fcom}). 
	In \citep{chizat2018global, chizat2020implicit} overparametrization is approached from the point of view of the width of the neural network, unlike our depth-inspired perspective.  
	The authors consider a two-layer perceptron with ReLU activation, exhibit a mean-field, Wasserstein gradient flow formulation of the descent scheme yielding the optimal parameters, and prove that these parameters approach global minimizers of the cost functional, with the global minimizer being characterized as a max-margin classifier in a certain non-Hilbert space of functions. 
	We refer the reader to \citep{sirignano2020mean, javanmard2020analysis,  mei2018mean, nitanda2017stochastic, chizat2018lazy} for related works in this direction, and \citep{nguyen2020rigorous} for an extension of the aforementioned convergence results to multi-layer perceptrons.
	
	\subsection{Outline} 
	The remainder of the paper is organized as follows.  
	In \Cref{sec: supervised.learning}, we give a comprehensive presentation of the neural ODE perspective of deep supervised learning.
	In \Cref{sec: wo.tracking}, we present our main results in the context of regularized empirical risk minimization (\Cref{thm: no.running} and \Cref{thm: thm.classification.lambda}).
	In \Cref{sec: turnpike.main}, we present our main result for the augmented empirical risk minimization problem, namely exponential decay of the training error with exponentially small parameters (\Cref{thm: turnpike.P} and also \Cref{thm: turnpike.BV}), as well as several numerical experiments stimulating the validity of our conjectures. 
	Finally in \Cref{sec: resnet.variable}, we present the continuous space-time analog of residual neural networks with variable widths, depict some possible approaches for passing from the continuous to the discrete case, and present extensions of \Cref{thm: no.running} and \Cref{thm: turnpike.P} in this context.
	The proofs of all results may be found in \Cref{sec: proofs}.
	
	\subsection{Notation, conventions, assumptions} \label{sec: notation}
	We set $\dot{x}(t) := \frac{\diff x}{\diff t}(t)$ and $[n]:=\{1,\ldots,n\}$.
	Given $a \in \R^n$, we denote by $a^\tr$ its transpose. 
	We use the notation $x_T$ to display the dependence of a vector $x\in\R^n$ on the time horizon $T$.
	We denote by $\|a\|$ the standard euclidean norm when $a\in\R^n$ is a vector, and the entry-wise euclidean norm (Frobenius norm) when $a\in\R^{n\times m}$ is a matrix -- we recall
	\begin{equation*}
	\|a\| := \left(\sum_{j=1}^n\sum_{k=1}^m |a_{j,k}|^2\right)^{\sfrac{1}{2}}.
	\end{equation*}
	We denote by $\Lip(\R)$ the set of functions $f: \R \longrightarrow \R$ which are globally Lipschitz continuous, and by $L^2(0, T; \R^n)$ (resp. $H^1(0,T;\R^n)$) the Lebesgue (resp. Sobolev) space consisting of all functions $f: (0,T) \longrightarrow \R^n$ which are square integrable (resp. square integrable and with a square integrable weak derivative) -- recall that $H^1(0,T;\R^n)$ is endowed with the norm 
	$\|f\|_{H^1(0,T;\R^n)}^2 := \|f\|^2_{L^2(0,T; \R^n)} + \|\dot{f}\|^2_{L^2(0,T; \R^n)}$. 
	We will use the convention $H^0=L^2$. 
	We recall (see \citep[p. 117]{ambrosio2000functions}) the definition of the space of \emph{bounded variation functions} $\BV(0,T;\R^n)$ as the space of integrable functions whose weak derivative is a finite Radon measure:
	\begin{align*}
	\BV(0,T;\R^n):=\Big\{f&\in L^1(0,T;\R^n)\colon \exists\, \mathrm{D}f \in \mathscr{M}(0,T;\R^n) \text{ such that } \\
	&\sum_{j=1}^n\int_0^T f_j(t)\varphi'_j(t)\diff t = -\sum_{j=1}^n\int_0^T\varphi_j(t)\diff\mathrm{D}f_j(t), \hspace{0.25cm} \forall \varphi\in C^1_{c}(0,T;\R^n) \Big\}.
	\end{align*}
	Here, $\mathscr{M}(0,T;\R^n)$ denotes the set of finite Radon measures on $(0,T)$ with values in $\R^n$.
	The space $\BV(0,T;\R^n)$ is endowed with the norm
	\begin{equation*}
	\|f\|_{\BV(0,T;\R^n)} := \|f\|_{L^1(0,T;\R^n)} + \sum_{j=1}^n |\mathrm{D}f_j|(0,T).
	\end{equation*}
	These definitions apply to matrix valued functions by simply considering the vectorized form of the matrix. Finally, we say that a function $f:\R\to\R$ is $\alpha$--homogeneous, for $\alpha>0$, if $f(cx) = c^\alpha x$ for $c\in\R$ and $x\in\R$.
	\medskip
	
	\noindent
	Throughout the remainder of this work, we will work under the following couple of assumptions, which are universal in the context of machine learning.	
	
	\begin{assumption}
	We henceforth assume that we are given a training dataset 
	\begin{equation*}
	\{\vec{x}_i, \vec{y}_i\}_{i\in[N]}\subset\mathcal{X}\times\mathcal{Y},
	\end{equation*}
	where $\mathcal{X}\subset\R^d$ and $\vec{x}_i\neq \vec{x}_j$ for $i\neq j$. We suppose that $\mathcal{Y}\neq\varnothing$, with $\mathcal{Y}\subset\R^m$ for regression tasks, $\mathcal{Y}=\{-1,1\}$ for binary classification and $\mathcal{Y}=[m]$ for mutli-label classification tasks.
	\end{assumption}
	
	\noindent
	We henceforth denote 
	\begin{equation*}
	d_u:=d\times(d+1), \hspace{1cm} d_x:=dN.
	\end{equation*}
	We will, most importantly, make use of the following abuse of notation: given the input data $\{\vec{x}_i\}_{i\in[N]}$ with $\vec{x}_i\in\R^d$, we define the stacked vector $\*x^0\in\R^{d_x}$ by
	\begin{equation} \label{eq: notation.abuse}
	\*x^0:=\begin{bmatrix}\vec{x}_1\\\vdots\\\vec{x}_N\end{bmatrix} \hspace{1cm} \text{ with } \hspace{1cm} \*x^0_i = \vec{x}_i\hspace{0.25cm} \text{ for } i\in[N].
	\end{equation}
	We shall apply convention \eqref{eq: notation.abuse} to other vectors generally denoted by $\*x\in\R^{d_x}$ and defined by stacking given sub-vectors $\{\*x_i\}_{i\in[N]}$ with $\*x_i\in\R^d$:
	\begin{equation*}
	\*x:=\begin{bmatrix}\*x_1\\\vdots\\\*x_N\end{bmatrix}.
	\end{equation*}
	Thus, when we write $\*x_i$ for $i\in[N]$, we insinuate that $\*x_i\in\R^d$, rather than considering solely the $i$-th row of $\*x\in\R^{d_x}$ (which is a scalar).
	\smallskip
	
	\noindent
	The following assumption is satisfied by most of the commonly used activation functions, including sigmoids such as $\sigma(x) = \tanh(x)$, and rectifiers: $\sigma(x) = \max\{\alpha x, x\}$ for $\alpha \in [0, 1)$.
	
	\begin{assumption}
	Unless stated otherwise, we fix an activation function $\sigma$ satisfying
	\begin{equation*}
	\sigma \in \Lip(\R) \hspace{0.25cm} \text{ and } \hspace{0.25cm} \sigma(0) = 0.
	\end{equation*}
	\end{assumption}

	\section{Roadmap to learning via neural ODEs} \label{sec: supervised.learning}
	
	\subsection{Feed-forward neural networks}	
	
	The canonical example of a feed-forward neural network is the \emph{multi-layer perceptron} (MLP), which generally takes the form
	\begin{equation} \label{eq: mlp}
	\begin{dcases}
	\*x_i^{k+1} = \sigma\left(w^k \*x^k + b^k\right) &\text{ for } k \in \{0, \ldots, {N_{\text{layers}}}-1\}\\
	\*x_i^0 = \vec{x}_i \in \R^d
	\end{dcases}
	\end{equation} 
	for $i\in[N]$; here ${N_{\text{layers}}}\geqslant1$ is the depth of \eqref{eq: mlp}, each $k$ is a layer, the vector $\*x_i^k \in \R^{d_k}$ designates the state at the layer $k$, $d_k$ is referred to as the width of the layer $k$, while $w^k \in \R^{d_{k+1}\times d_k}$ and $b^k \in \R^{d_k}$ are the optimizable weight and bias parameters of \eqref{eq: mlp}. 
	Finally, $\sigma \in \Lip(\R)$ is a fixed nonlinear activation function -- 
	by abuse of notation, we define the vector-valued analog of $\sigma$ component-wise, namely, $\sigma: \R^d \longrightarrow \R^d$ is defined by 
	$\sigma(\*x)_j := \sigma(\*x_j)$ for $j \in [d]$.
	Common choices include sigmoids such as $\sigma(x) = \tanh(x)$ or $\sigma(x) = \frac{1}{1+e^{-x}}$, and rectifiers: $\sigma(x) = \max\{x,ax\}$ for a fixed $0\leqslant a<1$. 
	In practice, the activation $\sigma$ is generally selected using cross-validation. 
	It can readily be seen that the formulation \eqref{eq: mlp} coincides with the more conventional formulation of neural networks as compositional structures of parametric affine operators and nonlinearities, as namely 
	$\*x_i^{N_{\text{layers}}} = (\sigma \circ \Lambda^k \circ \ldots \circ \sigma \circ \Lambda^0)(\vec{x}_i),$ with $\Lambda^k \vec{x}:= w^k \vec{x} + b^k$ for $k\in \{0, \ldots, N_{\text{layers}}\}$.
	
	Note that the iterative nature of the MLP \eqref{eq: mlp} stimulates permuting the order of the parametric affine maps and the nonlinearity $\sigma$, to the effect of considering the almost equivalent, but somewhat simpler system
	\begin{equation} \label{eq: mlp2}
	\begin{dcases}
	\*x_i^{k+1} = w^k\sigma\left(\*x^k\right) + b^k &\text{ for } k \in \{0, \ldots, N_{\text{layers}}-1\}\\
	\*x_i^0 = \vec{x}_i \in \R^d.
	\end{dcases}
	\end{equation} 
	We will henceforth concentrate on residual neural networks (ResNets).
	Contrary to the multi-layer perceptrons \eqref{eq: mlp} -- \eqref{eq: mlp2}, when considering ResNets one typically needs to assume that the width $d_k$ is fixed over every layer $k$, namely $d_k=d$ for every $k$.
	We refer to \Cref{sec: resnet.variable} for variable width ResNets.
	In the fixed width context, a residual neural network generally takes the form
	\begin{equation} \label{eq: resnet.discrete}
	\begin{dcases}
	\*x^{k+1}_i = \*x^k_i + \mathfrak{f}\left(u^k, \*x_i^k\right) &\text{ for } k \in \{0, \ldots, N_{\text{layers}}-1\}\\
	\*x^{0}_i = \vec{x}_i \in \R^d
	\end{dcases}
	\end{equation}
	for $i \in [N]$, where $\*x_i^k \in \R^d$ for any $i, k$, $u^k := [w^k, b^k] \in \R^{d\times(d+1)}$ and $\mathfrak{f}$ as the right hand side in \eqref{eq: mlp} or \eqref{eq: mlp2}.	
	As explained in \citep{lu2018beyond}, other network architectures (including convolutional neural networks) can be fit into the residual network framework.
	
	\subsection{Neural ODEs}
	
	It is readily seen that \eqref{eq: resnet.discrete} corresponds, modulo a scaling factor $\Delta t = \sfrac{T}{N_{\text{layers}}}=1$, to the forward Euler discretization of 
	\begin{equation} \label{eq: standard.dyn.sys}
	\begin{dcases}
	\dot{\*x}_i(t) = \mathfrak{f}(u(t), \*x_i(t)) &\text{ in } (0, T) \\
	\*x_i(0) = \vec{x}_i \in \R^d,
	\end{dcases}
	\end{equation}
	for $i \in [N]$. 
	Here $T>0$ is a fixed time horizon, and $u(t) := [w(t), b(t)] \in \R^{d\times(d+1)}$.
	As per what precedes, the nonlinearity $\mathfrak{f}$ in \eqref{eq: standard.dyn.sys} may take the form
	\begin{equation} \label{eq: sigma.outside}
	\mathfrak{f}(u(t), \*x_i(t)) := \sigma(w(t) \*x_i(t) + b(t))
	\end{equation}
	or 
	\begin{equation} \label{eq: sigma.inside}
	\mathfrak{f}(u(t), \*x_i(t)) = w(t)\sigma(\*x_i(t))+b(t).
	\end{equation}
	for $i \in [N]$. 	
	We will address both cases in our analytical study, and emphasize the stark differences between the two.  
    The above parametrizations are not the lone considered in practice. 
    In fact, one may consider, for instance, combinations of \eqref{eq: sigma.outside} and \eqref{eq: sigma.inside} which allow intermediate exploration (bottlenecks) in different dimensions:
    \begin{equation} \label{eq: outside.inside}
	\mathfrak{f}(u(t), \*x_i(t)) := w_2(t) \sigma(w_1(t) \*x_i(t) + b_1(t)) + b_2(t)
	\end{equation}
	where now $w_1(t) \in \R^{d_\hid \times d}$, $w_2(t) \in \R^{d \times d_\hid}$, $b_1(t) \in \R^{d_\hid}$ and $b_2(t) \in \R^{d}$.
		 		
	\subsection{Training} \label{eq: softmax}
	
	For an input sample $\vec{x}_i \in \R^d$, the output of the neural ODE \eqref{eq: standard.dyn.sys}, which is used for comparison with the corresponding label $\vec{y}_i$, is a projection of the form $P\*x_i(T) \in \R^m$ for some affine map $P:\R^d\to\R^m$: 
	\begin{equation*}
	Px := p_1x+p_2, 
	\end{equation*}
	where $p_1\in \R^{m\times d}$ and $p_2\in \R^m$. In the context of binary classification, namely $m=1$ with $\vec{y}_i=\pm1$, one may also use $Px := \tanh(p_1 x+p_2)$, for instance. The parameters $p_1, p_2$ are usually part of the trainable variables, but, as mentioned in the introduction, we shall assume that they are given (but otherwise arbitrary, e.g., $p_1$ and $p_2$ may be sampled as a random matrix and random vector respectively) for technical reasons.	
	\smallskip
	
	\noindent
	In supervised learning, one seeks to tune the parameters $[w,b]$ so that $P\*x_i(T)$ approaches $\vec{y}_i$ for $i \in [N]$ with respect to some loss function. 
	To this end, the Tikhonov-regularized, empirical risk minimization problem
	\begin{equation} \label{eq: non.regular.standard.sup.learn}
	\inf_{\substack{[w,b]\\ \*x_i(\cdot) \text{ solves } \eqref{eq: standard.dyn.sys}}} \frac{1}{N} \sum_{i=1}^N \loss\big(P\*x_i(T),\vec{y}_i\big) + \lambda \Big\|[w,b]\Big\|_{H^k(0,T; \R^{d_u})}^2
	\end{equation} 
	is considered, where $\lambda>0$ is fixed, and $\*x_i$ solves \eqref{eq: standard.dyn.sys} with $\mathfrak{f}$ as in \eqref{eq: sigma.outside} or \eqref{eq: sigma.inside} (although, more general cases such as \eqref{eq: outside.inside} can also be considered).
	Here 
	\begin{equation*}
	\loss(\cdot, \cdot):\R^m\times\mathcal{Y}\to\R_+
	\end{equation*}
	is a given continuous function, which in our work we will choose on a case-by-case basis. Common examples include $\loss(x,y)=\|x-y\|_{\ell^p}^p$ with $p=1,2$ for regression tasks ($\mathcal{Y}\subset\R^m$), and the cross-entropy loss 
	\begin{equation*}
	\loss(x,y) = -\log\left(\frac{e^{x_{y}}}{\sum_{j=1}^m x_j}\right) \hspace{1cm} x\in\R^m,\, y\in[m]
	\end{equation*}
	for classification tasks.
	Note that \eqref{eq: non.regular.standard.sup.learn} is the empirical and regularized approximation of the expected risk minimization problem: 
	\begin{equation*}
	    \inf_{\substack{[w,b]\\ \*x_i(\cdot) \text{ solves } \eqref{eq: standard.dyn.sys}}} \mathbb{E}\Big[\loss\left(P\*x_{\cdot}(T),\cdot\right)\Big],
	\end{equation*}
	where $\mathbb{E}[f(\cdot,\cdot)] := \int_{\R^d\times\R^m} f(x,y)\diff \rho(x,y)$, with $\*x_{\vec{x}}$ denoting the solution to \eqref{eq: standard.dyn.sys} with initial datum $\vec{x}$.
	Here $\rho: \R^{d}\times \R^m\to[0,1]$ is an unknown joint probability distribution, from which one samples the training dataset $\{\vec{x}_i, \vec{y}_i\}_{i\in[N]}$.
	We shall solely focus on the empirical problem in this work.
	\smallskip
	
	\noindent
	By virtue of the direct method in the calculus of variations, the existence\footnote{Uniqueness ought not to be expected (as in most deep learning tasks) due to the inherent lack of convexity. Notwithstanding, there exist cases (e.g. two-layer perceptrons) where uniqueness can be ensured by lifting the training problem as a minimization problem over measures, which may render the problem convex (see \citep{chizat2018global}).} of minimizers for the learning problems we consider herein can readily be shown when the appropriate parameter regularization is used (see \Cref{lem: compact.flow}). We  put emphasis on the following remark (see also \Cref{rem: BV.need}).
		
	\begin{remark}[Sobolev regularization] \label{rem: sobolev.regularization}
	We stress the consideration of a Sobolev, $H^1$--regularization in the case of \eqref{eq: sigma.outside.i} as this is sufficient to guarantee the existence of a minimizer. 
	Indeed, an issue could arise if solely an $L^2$--regularization is used due to the specific nonlinear form of the neural ODE \eqref{eq: sigma.outside.i}, which could (but might not) be an impediment for passing to the limit in the equation using only weak convergences (recall $\{\sin(nx)\}_{n=1}^{\infty}$ and $\sigma(x) = \max\{x,0\}$). 
	This issue is specific to the continuous-time setting, as in the discrete-time thus finite dimensional optimization setting, weak and strong convergences coincide. To our knowledge, a proof of existence of a minimizer in the context of $L^2$--regularized problems under System \eqref{eq: sigma.outside.i} is not known. 
	\end{remark}

	\section{Empirical risk minimization} \label{sec: wo.tracking}
	
	\noindent
	Throughout the paper we will focus on neural ODEs given by \eqref{eq: standard.dyn.sys} with $\mathfrak{f}$ as in \eqref{eq: sigma.outside} or \eqref{eq: sigma.inside}.  
	We comment on further extensions on a case-by-case basis. 
	It will be rather convenient to work with the full stacked state trajectory 
	\begin{equation*}
	\*x(t) = \begin{bmatrix}\*x_1(t)\\\vdots\\\*x_N(t)\end{bmatrix},
	\end{equation*}
	and we recall that we make use of the convention \eqref{eq: notation.abuse}, and we recall that $d_u := d\times(d + 1)$ and $d_x := dN$.
	Moreover, given $w \in \R^{d\times d}$ and $b \in \R^{d}$, we shall write
	\begin{equation} \label{eq: form.controls}
	\*w:=	\begin{bmatrix}
    w & & \\
    &\ddots& \\
    & & w  \end{bmatrix} \in \R^{d_x\times d_x}, \hspace{1cm} \*b := \begin{bmatrix}b \\\vdots \\b \end{bmatrix}
 \in \R^{d_x}.
	\end{equation}
	In view of the above discussion and noting \eqref{eq: form.controls}, we will consider stacked neural ODEs in $\R^{d_x}$ such as
	\begin{equation} \label{eq: sigma.outside.i}
	\begin{dcases}
	\dot{\*x}(t) = \sigma(\*w(t) \*x(t) + \*b(t)) &\text{ for } t \in (0,T) \\
	\*x(0) = \*x^0 \in \R^{d_x},
	\end{dcases}
	\end{equation}
	and 
	\begin{equation} \label{eq: sigma.inside.i}
	\begin{dcases}
	\dot{\*x}(t) = \*w(t) \sigma(\*x(t))+\*b(t) &\text{ for } t \in (0,T) \\
	\*x(0) = \*x^0 \in \R^{d_x}.
	\end{dcases}
	\end{equation}	
	In this section, we consider the problem of regularized empirical risk minimization. For simplicity of notation, we henceforth denote the empirical risk by
    \begin{equation} \label{eq: phi.def}
	\mathscr{E}(\*x) := \frac{1}{N} \sum_{i=1}^N \loss\big(P\*x_i,\vec{y}_i\big),
	\end{equation}
	for $\*x \in \R^{d_x}$, where $P:\R^d\to\R^m$ and $\loss(\cdot,\cdot) \in C^0(\R^m\times\mathcal{Y}; \R_+)$ are given -- both will change with respect to the task in question (regression, classification), as discussed in \Cref{eq: softmax}.
	\smallskip
	
	\noindent
	For fixed $\lambda>0$, in this section we will study the behavior when $T\gg1$ of global minimizers $[w_T,b_T]$ to the functional
	\begin{equation} \label{eq: JT}
	J_{\lambda, T}(w,b) = \mathscr{E}(\*x(T)) + \lambda \Big\|[w,b]\Big\|^2_{H^k(0,T;\R^{d_u})}
	\end{equation}
	where $\*x \in C^0([0, T]; \R^{d_x})$ is the unique solution to either \eqref{eq: sigma.inside.i} ($k=0$) or \eqref{eq: sigma.outside.i} ($k=1$) corresponding to the parameters $[w,b]  \in H^k(0, T; \R^{d_u})$, noting \eqref{eq: form.controls}.

	\subsection{Non-negative losses} 
	We begin by considering the case wherein the function $(x,y)\mapsto \loss(Px,y)$ defining the empirical risk in \eqref{eq: phi.def} may attain its minimum $0$. 
	Namely, unless and until stated otherwise, we shall suppose
	
	\begin{assumption} \label{ass: regression.ass}
	We suppose that $P:\R^d\to\R^m$ and $\loss\in C^0(\R^m\times\mathcal{Y};\R_+)$ are such that $\mathscr{E}$ defined in \eqref{eq: phi.def} may attain its minimum $0$.
	\end{assumption}
	
	\noindent
	This is the case for several losses used in practice, including 
	\begin{itemize}
	\item $\ell^p$--losses: $\loss(x,y) = \|x-y\|_{\ell^p}^p$, $p\in\{1,2\}$, with $P:\R^d\to\R^m$ an affine map, or, more generally, losses which are radial functions with respect to $x-y$ (such as those induced by a distance inferred by a norm).
	Such modeling assumptions are typically made in the context of regression tasks ($\vec{y}_i\in\mathcal{Y}\subset\R^m$), wherein by minimizing the empirical risk one looks to interpolate the training data by means of the projected neural ODE flow.
	\smallskip
	
	\item The \emph{multi-class hinge loss}
	\begin{equation*}
	\loss(x,y) = \sum_{\substack{j=1\\j\neq y}}^m \max\left\{0,1-x_y+x_j\right\} \hspace{1cm} x\in\R^m,\, y\in[m],
	\end{equation*} 
	with $P:\R^d\to\R^m$ being a matrix, which may be used for classification tasks.
	\end{itemize} 	
	
	\noindent
	With regard to \Cref{ass: regression.ass}, we define the following notion of interpolation of the dataset $\{\vec{x}_i,\vec{y}_i\}_{i\in[N]}$.
		
	\begin{definition}[Interpolation] \label{def: P-interpolation}
	Let $P:\R^d\to\R^m$ be any non-zero map and let $\loss\in C^0(\R^m\times\mathcal{Y};\R_+)$ be such that \Cref{ass: regression.ass} is satisfied. 
	We say that \eqref{eq: sigma.inside.i} (resp. \eqref{eq: sigma.outside.i}) interpolates the dataset $\left\{\vec{x}_i,\vec{y}_i\right\}_{i\in[N]}$ in some time $T>0$ if there exists a time $T>0$ and some parameters $[w,b] \in L^2(0,T; \R^{d_u})$ (resp. in $H^1(0,T; \R^{d_u})$) such that the unique solution $\*x(\cdot)$ to \eqref{eq: sigma.inside.i} (resp. \eqref{eq: sigma.outside.i}), noting \eqref{eq: form.controls}, satisfies
	\begin{equation*}
	\mathscr{E}(\*x(T)) = 0.
	\end{equation*}
	\end{definition}
	
	\noindent
	Note that this is a slight abuse of terminology, as precisely interpolation would rather refer to having 
	\begin{equation*}
	P\*x_i(T)= \vec{y}_i \hspace{1cm} \text{ for all } \quad i\in [N],
	\end{equation*}
	which is equivalent to the stated definition in the context of losses induced by a distance, but not in the context of the hinge loss. We shall, however, make use of the terminology for cohesion.
	
	Let us also note that by means of an elementary time-scaling, if \eqref{eq: sigma.inside.i} or \eqref{eq: sigma.outside.i} interpolates the dataset in some time $T>0$, it interpolates it in any time, in particular, in time $1$. 
	We will make use of this observation to simplify the subsequent presentation and analysis by assuming interpolation in time $1$ without loss of generality.
	
   	We may state our main result in this context. 
	
	\begin{theorem} \label{thm: no.running}
	Fix $\lambda>0$.
	Suppose that $P:\R^d\to\R^m$ is any non-zero affine map, and suppose that $\loss\in C^0(\R^{m}\times\mathcal{Y};\R_+)$ is such that \Cref{ass: regression.ass} is satisfied.
	Assume that \eqref{eq: sigma.inside.i} (resp. \eqref{eq: sigma.outside.i} with $\sigma$ $1$--homogeneous) interpolates the dataset $\left\{\vec{x}_i, \vec{y}_i\right\}_{i\in[N]}$ in time $1$ in the sense of \Cref{def: P-interpolation}.
	For any $T\geqslant1$ let $\left[w_T, b_T\right] \in L^2(0,T;\R^{d_u})$ (resp. in $H^1(0,T; \R^{d_u})$) be any pair of global minimizers to $J_{\lambda,T}$ defined in \eqref{eq: JT}, and let $\*x_T(\cdot)$ be the unique associated solution to \eqref{eq: sigma.inside.i} (resp. \eqref{eq: sigma.outside.i}), noting \eqref{eq: form.controls}.
	The following properties then hold.
	\begin{enumerate}
	\item There exists a constant $\mathfrak{C}=\mathfrak{C}\left(\{\vec{x}_i, \vec{y}_i\}_{i\in[N]},\lambda\right)>0$ independent of $T$ such that
	\begin{equation*}
	\mathscr{E}(\*x_T(T)) \leqslant \frac{\mathfrak{C}}{T}.
	\end{equation*}
	\item There exists a sequence $\{T_n\}_{n=1}^{\infty}$, with $T_n>0$ and $T_n \xrightarrow[n\longrightarrow \infty]{}\infty$, and some $\*x_\circ \in \R^{d_x}$ with $\mathscr{E}(\*x_\circ)=0$ such that, along a subsequence,
	\begin{equation} \label{eq: features.converge}
	\*x_{T_n}(T_n) \xrightarrow[n\longrightarrow \infty]{} \*x_\circ.
	\end{equation}
	\item 
 	For any $n\geqslant 1$, set
	\begin{align*}
	w_n(t) &:= T_n\, w_{T_n} \left(t\,T_n\right) \hspace{1cm} \text{ for } t \in [0, 1], \\
	b_n(t) &:= T_n\, b_{T_n} \left(t \, T_n\right) \hspace{1cm} \text{ for } t \in [0, 1].
	\end{align*}
	Then along a subsequence,
	\begin{equation*}
	\Big\| \left[w_n,b_n\right] - \left[w^*, b^*\right] \Big\|_{H^k(0,1;\R^{d_u})} \xrightarrow[n\longrightarrow \infty]{} 0,
	\end{equation*}
		where $[w^*, b^*]\in H^k(0,1;\R^{d_u})$ is a solution to the minimization problem
	\begin{align*} \label{opt pbm in thm no.running}
	\inf_{\substack{[w,b] \in H^k(0,1;\R^{d_u})\\ \*x(\cdot) \text{ solves } \eqref{eq: sigma.outside.i} \text{ (resp. \eqref{eq: sigma.inside.i}) in } [0,1] \\ \text{ and } \\ \scalebox{0.7}{$\mathscr{E}(\*x(1))=0$} }} \Big\| [w,b]\Big\|^2_{H^k(0,1;\R^{d_u})}.
	\end{align*}
	\end{enumerate}
	\end{theorem}

	\noindent
	\textit{Idea of proof.} The proof of \Cref{thm: no.running} may be found in \Cref{sec: proof.no.running}. Let us motivate the main underlying idea. 
	
	Under the above assumptions, both \eqref{eq: sigma.outside.i} and \eqref{eq: sigma.inside.i} will be $1$--homogeneous with respect to the parameters $[w(t),b(t)]$. 
	Namely, both \eqref{eq: sigma.outside.i} and \eqref{eq: sigma.inside.i} (noting \eqref{eq: form.controls}) can be written as
	\begin{equation} \label{eq: aux.idea.sys}
	\begin{dcases}
	\dot{\*x}(t) = \mathfrak{f}\left([w(t), b(t)], \*x(t)\right) &\text{ in } (0, T) \\
	\*x(0) = \*x^0,
	\end{dcases}
	\end{equation}
	where $\mathfrak{f}([\alpha w, \alpha b], \*x) = \alpha \mathfrak{f}([w, b], \*x)$ for $\alpha > 0$.
	Whilst in the case of \eqref{eq: sigma.inside.i} this homogeneity property holds for any activation function $\sigma$, we require $\sigma$ to be $1$--homogeneous for neural networks such as \eqref{eq: sigma.outside.i}. 
	This includes rectifiers, but excludes sigmoids.
	
	A simple computation (see \Cref{lem: scaling}) then leads to noting that, given some parameters $u^{1} := [w^1, b^1]$ and the solution $\*x^1$ to 
	\begin{equation}  \label{eq: aux.idea.sys1}
	\begin{dcases}
	\dot{\*x}^1(t) = \mathfrak{f}\left(\left[w^1(t), b^1(t)\right], \*x^1(t)\right) &\text{ in } (0, 1) \\
	\*x^1(0) = \*x^0,
	\end{dcases}
	\end{equation}
	then $u^T(t) := \frac{1}{T} u^1(\frac{t}{T})$ for $t\in [0, T]$ is such that
	$\*x^T(t):= \*x^1(\frac{t}{T})$ solves \eqref{eq:  aux.idea.sys}.
	Under the interpolation assumption, we may find $u^{1} \in H^k(0,T_0; \R^{d_u})$ such that the corresponding solution $\*x^{1}$ satisfies $\mathscr{E}(\*x^{1}(1))=0$, and then use the above scaling and the optimality of $u_T$ to deduce
	\begin{align*}
	J_{\lambda,T}(u_T) \leqslant \mathscr{E}\left(\*x^{1}(1)\right) + \frac{\lambda}{T} \left\|u^{1}\right\|_{H^k(0,1; \R^{d_u})}^2 =\frac{\lambda}{T} \left\|u^{1}\right\|_{H^k(0,1; \R^{d_u})}^2
	\end{align*}
	for $T\geqslant 1$.
	This will imply the decay estimate of $\mathscr{E}$, and combined with some more technical compactness arguments, will yield the remaining convergence results as well.
	
	On another hand, considering the case of \eqref{eq: sigma.inside.i} and thus $k=0$, we see that	
	\begin{align*}
	 &\inf_{\substack{u_T=[w_T,b_T] \in L^2(0,T;\R^{d_u}) \\ \*x_T(\cdot) \text{ solves } \eqref{eq: aux.idea.sys}}} \mathscr{E}(\*x_T(T)) + \lambda \int_0^T \left\|u_T(t)\right\|^2 \diff t \\
	 &=\inf_{\substack{u_T=[w_T,b_T] \in L^2(0,T;\R^{d_u}) \\ \*x_T(\cdot) \text{ solves } \eqref{eq: aux.idea.sys}}} \mathscr{E}(\*x_T(T)) + \frac{\lambda}{T} \int_0^1 \left\| Tu_T(sT) \right\|^2\diff s \\
	 &= \inf_{\substack{u^1=[w^1, b^1] \in L^2(0,1;\R^{d_u}) \\ \*x^1(\cdot) \text{ solves } \eqref{eq: aux.idea.sys1}}} \mathscr{E}(\*x^1(1)) + \frac{\lambda}{T} \int_0^1 \left\|u^1(s)\right\|^2\diff s. 
	 \end{align*}
	 This computation indicates that one may consider the behavior when $T\longrightarrow\infty$ for fixed $\lambda>0$ and that when $\lambda\searrow0$ for fixed $T>0$ in the same fashion. 
	 Although this scaling is specific to the $L^2$--regularization setting, it motivates completing \Cref{thm: no.running} with the following result.
	 
	 \begin{theorem} \label{thm: no.running.lambda}
	 Under the assumptions of \Cref{thm: no.running}, fix $T>0$, and for any $\lambda>0$, let $\left[w_\lambda, b_\lambda\right] \in L^2(0,T;\R^{d_u})$ (resp. $H^1(0,T; \R^{d_u})$) be any pair of global minimizers to $J_{\lambda,T}$ defined in \eqref{eq: JT}, and let $\*x_\lambda$ be the unique associated solution to \eqref{eq: sigma.inside.i} (resp. \eqref{eq: sigma.outside.i}), noting \eqref{eq: form.controls}.
	 The following properties then hold.
	\begin{enumerate}
	\item There exists a constant $\mathfrak{C}=\mathfrak{C}\left(\{\vec{x}_i, \vec{y}_i\}_{i\in[N]}, T\right)>0$ independent of $\lambda>0$ such that
	\begin{equation*}
	\mathscr{E}(\*x_\lambda(T)) \leqslant \mathfrak{C}\,\lambda.
	\end{equation*}
	\item There exists a sequence $\{\lambda_n\}_{n=1}^{\infty}$, with $\lambda_n>0$ and $\lambda_n \xrightarrow[n\to\infty]{} 0$, and some $\*x_\circ \in \R^{d_x}$ with $\mathscr{E}(\*x_\circ)=0$ such that, along a subsequence
	\begin{equation*}
	\*x_{\lambda_n}(T) \xrightarrow[n\longrightarrow\infty]{} \*x_\circ.
	\end{equation*}
	\item Moreover, along a subsequence,
	\begin{equation*}
	\Big\| \left[w_{\lambda_n},b_{\lambda_n}\right] - \left[w^*, b^*\right] \Big\|_{H^k(0,T;\R^{d_u})} \xrightarrow[n\longrightarrow\infty]{} 0 ,
	\end{equation*}
	where $[w^*, b^*]^\tr\in H^k(0,T;\R^{d_u})$ is a solution to the minimization problem
	\begin{align*}
	\inf_{\substack{[w,b] \in H^k(0,T;\R^{d_u})\\ \*x(\cdot) \text{ solves } \eqref{eq: sigma.outside.i} \text{ (resp. \eqref{eq: sigma.inside.i})}\\ \text{ and } \\ \scalebox{0.7}{$\mathscr{E}(\*x(T))=0$} }} \Big\| [w,b]\Big\|^2_{H^k(0,T;\R^{d_u})}.
	\end{align*}
	\end{enumerate}
	 \end{theorem}
	 
	 \begin{remark}[Homogeneity]        \begin{itemize}
   
       \item The impediment that appears in the results presented above when considering neural ODEs of the form 
       \begin{equation} \label{eq: compound.neural.ode}
       \begin{dcases}
       \dot{\*x}(t) = \*w^1(t) \sigma\left(\*w^2(t) \*x(t) + \*b^2(t)\right) + \*b^1(t) &\text{ in } (0,T) \\
       \*x(0) = \*x^0,
       \end{dcases}
       \end{equation}
       where $w^1(t) \in \R^{d\times d_{\mathrm{hid}}}$ and $w^2 \in \R^{d_{\mathrm{hid}}\times d}$, with $\*w^1(t)$ and $\*w^2(t)$ defined after stacking the states $\{\*x_i(t)\}_{i\in[N]}$ within $\*x(t)$ following convention \eqref{eq: notation.abuse}, is the lack of homogeneity (and thus scaling) of the dynamics with respect to the parameters.
       Let us elaborate by focusing on $\*b^2\equiv\*b^1\equiv0$ for simplicity.
       Let $\*x_T$ denote the solution to \eqref{eq: compound.neural.ode} with parameters $[\*w_T^1, \*w_T^2]$. We set 
       \begin{equation*}
       \*x_1(t) := \*x_T(tT) \hspace{1cm} \text{ for } t\in[0,1],
       \end{equation*}
       and we see that, assuming $\sigma$ is $\alpha$--homogeneous with $\alpha>0$, $\*x_1$ solves \eqref{eq: compound.neural.ode} on $[0,1]$ with parameters 
       \begin{equation*}
       \left[w^1_1, w^2_1\right]:=\left[T^p w^1_T(tT), T^{\sfrac{q}{\alpha}} w^2_T(tT)\right] \hspace{1cm} \text{ for } t\in(0,1),
       \end{equation*}
       with $p+q=1$, noting \eqref{eq: form.controls}. Solely looking at the squared $L^2$--norms of the parameters, we see that
        \begin{align*}
        \int_0^T \|w^1_T(t)\|^2\diff t + \int_0^T \|w^2_T(t)\|^2\diff t &= T \int_0^1\left\|w^1_T(sT)\right\|^2 \diff s + T\int_0^1\left\|w^2_T(sT)\right\|^2\diff s \\
        &=T^{1-2p}\int_0^1 \left\|w_1^1(s)\right\|^2\diff s + T^{1-\sfrac{2q}{\alpha}} \int_0^1 \left\|w_1^2(s)\right\|^2\diff s.
        \end{align*}
      We see that, in order to have a scaling factor which decays when $T\longrightarrow\infty$, a property which is a cornerstone of our proof, we would simultaneously need $1-2p<0$ and $1-\frac{2q}{\alpha}<0$, i.e. $p>\frac12$ and $q>\frac{\alpha}{2}$. Due to the constraint $p+q=1$, this would entail $\alpha<1$, which means that $\sigma$ cannot be globally Lipschitz continuous. 
       \smallskip
       
       \item Let us, in line with what is discussed just above, check what is the  needed degree of homogeneity of $\sigma$ in the case of System \eqref{eq: sigma.outside.i} to ensure that the results of \Cref{thm: no.running} hold. 
       Let us thus assume that $\sigma$ is $\alpha$--homogeneous with $\alpha>0$. Then, setting $\*x_1(t):=\*x_T(tT)$ for $t\in[0,1]$ where $\*x_T$ solves \eqref{eq: sigma.outside.i} with parameters $[w_T,b_T]$, we see that $\*x_1(t)$ solves \eqref{eq: sigma.outside.i} on $[0,1]$ with parameters 
       $[w_1(t),b_1(t)]=\left[T^{\sfrac{1}{\alpha}}w_T(tT), T^{\sfrac{1}{\alpha}}b_T(tT)\right]$ for $t\in(0,1)$. Then, looking at the $L^2$--norms of the parameters, we see that
       \begin{align*}
       \int_0^T \left\|w_T(t)\right\|^2\diff t + \int_0^T\left\|b_T(t)\right\|^2\diff t &= T\int_0^1 \|w_T(sT)\|^2\diff s + T\int_0^T \|b_T(sT)\|^2\diff s \\
       &= T^{1-\sfrac{2}{\alpha}}\int_0^1 \|w_1(s)\|^2 \diff s + T^{1-\sfrac{2}{\alpha}}\int_0^1\|b_1(s)\|^2 \diff s,
       \end{align*}
       and thus, our methodology really needs $\alpha\in[1,2)$. 
       Note that this brief computation also indicates the precise proximity of $\mathscr{E}(\*x_T(T))$ to $0$ for $\sigma$ is $\alpha$--homogeneous, which is of the order of $\mathcal{O}\left(T^{1-\sfrac{2}{\alpha}}\right)$.
       \end{itemize}
	 \end{remark}
	 
\begin{remark}[The output layer $P$] We note that the output layer parameters given by the affine map $P$ are fixed, but in general arbitrary.
       This is due to the fact that if we were to optimize $P$ as well, we would have to ensure that the optimal $P$ is bounded with respect to the limiting hyper-parameter ($\lambda$ or $T$). This in turn could perhaps be ensured if we were to regularize the output layer as well, but would, in turn, be an impediment to the scaling arguments we use in all proofs since then the parameter regularization norm would not scale polynomially with $T$.
       \end{remark}
	
	\subsection{Positive losses} 
	
	We henceforth consider the standard setting of classification tasks, wherein the labels $\vec{y}_i$ take values in a set of $m\geqslant 2$ classes -- unless stated otherwise, we consider $\vec{y}_i\in[m]$ for all $i\in[N]$. 
	We will focus on the cross-entropy loss in \eqref{eq: phi.def}, which we recall, when evaluated along the neural ODE output, reads
	\begin{equation} \label{eq: cross.entropy.def}
	\loss(P\*x_i(T), \vec{y_i}) := -\log\left(\dfrac{e^{P\*x_i(T)_{\vec{y}_i}}}{\sum_{j=1}^m e^{P\*x_i(T)_j}}\right).
	\end{equation}
	Here $P: \R^d \longrightarrow \R^m$ is an affine map, with more assumptions made precise later on.
	
	An important feature of the cross-entropy loss is the fact that it is not coercive with respect to the first variable -- namely, as $P\*x_i(T)_{\vec{y}_i}$ goes to infinity, the loss goes to zero. In particular, the infimum $0$ is never attained.
	This is in line with intuition regarding the classification task, as one looks to separate the features with respect to their individual class in the label space $\R^m$.
		
	The problem consisting of classifying a given dataset is closely tied to the following rather intuitive notion of separability, which we will need in subsequent results.

	\begin{definition}[Separability] \label{def: separability}
	Let $P: \R^d\longrightarrow \R^m$ be any non-zero affine map. 
	We say that \eqref{eq: sigma.inside.i} (resp. \eqref{eq: sigma.outside.i}) separates the dataset $\left\{\vec{x}_i, \vec{y}_i\right\}_{i\in[N]}$ with respect to $P$ if there exists a time $T>0$ and some parameters $[w,b] \in L^2(0,T; \R^{d_u})$ (resp. in $H^1(0,T; \R^{d_u})$) such that the unique solution $\*x(\cdot)$ to \eqref{eq: sigma.inside.i} (resp. \eqref{eq: sigma.outside.i}) satisfies
	\begin{equation*}
	P\*x_i(T)_{\vec{y}_i} - \max_{\substack{j\in[N]\\j\neq \vec{y}_i}} P\*x_i(T)_j >0 \hspace{1cm} \text{ for all } i \in [N].
	\end{equation*}
	\end{definition}
		
	\noindent
	In other words, a parametrized neural ODE flow separates the given dataset if the corresponding \emph{margin} $\gamma_{[w,b]}$, defined as
	\begin{equation} \label{eq: margin.def}
	\gamma_{[w,b]} := \min_{i\in[N]}\left(P\*x_i(T)_{\vec{y}_i} - \max_{\substack{j\in[N]\\j\neq \vec{y}_i}} P\*x_i(T)_j\right)
	\end{equation}
	is positive.
	We may now state our main result in the classification context, which entails a quantitative rate of decay as $T\longrightarrow\infty$ of the training error with cross-entropy loss for ReLU activated neural ODEs.
	
	\begin{theorem} \label{thm: thm.classification.lambda}
	Let $\left\{\vec{x}_i, \vec{y}_i\right\}_{i\in[N]}$ be a given dataset with $\vec{x}_i\in\R^d$ and $\vec{y}_i\in[m]$.
	Let $\lambda>0$ be fixed, and let $\mathfrak{Q}:\R^{d_x} \longrightarrow \R^d$ be a non-zero affine map such that
$\mathfrak{Q}\vec{x}_i \geqslant 0$ for $i \in [N]$.
	Set (recall convention \eqref{eq: notation.abuse})
	\begin{equation*}
	\*x^0_i := \mathfrak{Q}\vec{x}_i \hspace{1cm} \text{ for } i\in[N],
	\end{equation*}
	and let $P\in\R^{m\times d}$ be any non-zero matrix such that System \eqref{eq: sigma.outside.i}, with $\sigma(x)=\max\{x,0\}$, separates the dataset $\left\{\*x^0_i, \vec{y}_i\right\}_{i\in[N]}$ with respect to $P$ in some time $T_0>0$ as per \Cref{def: separability}, and let $\gamma$ denote the associated margin as defined in \eqref{eq: margin.def}.
	For any $T\geqslant T_0$, let $\left[w_T, b_T\right]\in H^1(0,T;\R^{d_u})$ be any pair of global minimizers to $J_{\lambda,T}$ defined in \eqref{eq: JT}--\eqref{eq: cross.entropy.def}, and let $\*x_T(\cdot)$ be the associated unique solution to \eqref{eq: sigma.outside.i} with $\sigma(x)=\max\{x,0\}$.
Then, there exists a constant $\mathfrak{C}=\mathfrak{C}\left(\{\vec{x}_i, \vec{y}_i\}_{i\in[N]}, \lambda\right)>0$ independent of $T>0$ such that
	\begin{equation} \label{eq: cross.entropy.estimate}
	\mathscr{E}(\*x_T(T)) \leqslant \log \left( 1+ (m-1) e^{-\gamma\, e^{\frac{T^{\alpha}}{2}}} \right) + \mathfrak{C}\,T^{2\alpha-1}
	\end{equation}	
	holds for any $\alpha \in \left(0,\frac{1}{2}\right)$.
	\end{theorem}
	
	\noindent
	By using a Taylor expansion, one sees that the first term in the upper bound is negligible when $T\gg1$. Thus, the training error $\mathscr{E}(\*x_T(T))$ is at most of the order $\mathcal{O}\left(T^{2\alpha-1}\right)$. 
	\smallskip

\noindent
We note that the above theorem is very specific to neural ODEs of the form \eqref{eq: sigma.outside.i} with ReLU activations, and the specific form of the cross-entropy loss, from which the first term in the estimate \eqref{eq: cross.entropy.estimate} is derived. This is due to the proof strategy, which relies on using the positivity of the right hand side to, in some sense, obtain a linear equation for the projected output features for some auxiliary parameters constructed within the proof, and thus have an explicit solution for these parameters of the form $\sim e^t$. This stimulates the appearance of the second exponential within the $\log$ in \eqref{eq: cross.entropy.estimate}.
\medskip

\noindent
Much like what we observed in the setting of losses which attain their minimum, we can expect to link the limit as $T$ goes to infinity with the convergence of the regularization path, namely the limit as $\lambda \searrow0$. This is depicted in the following theorem.

	\begin{theorem} \label{thm: thm.classification.lambda.reg}
	Under the assumptions of \Cref{thm: thm.classification.lambda}, fix $T\geqslant T_0$ and for any $\lambda>0$ let $[w_\lambda, b_\lambda]\in H^1(0,T; \R^{d_u})$ be any pair of global minimizers to $J_{\lambda,T}$ defined in \eqref{eq: JT}--\eqref{eq: cross.entropy.def}, and let $\*x_\lambda(\cdot)$ be the associated unique solution to \eqref{eq: sigma.outside.i} with $\sigma(x)=\max\{x,0\}$.
Then, there exists a constant $\mathfrak{C}=\mathfrak{C}\left(\{\vec{x}_i,\vec{y}_i\}_{i\in[N]}, T\right)>0$ independent of $\lambda>0$ such that
	\begin{equation*}
	\mathscr{E}(\*x_\lambda(T)) \leqslant \log \left( 1+ (m-1) e^{-\gamma\, e^{\frac{\lambda^{-\alpha}}{2}}} \right) + \mathfrak{C}\,\lambda^{-2\alpha+1}
	\end{equation*}	
	holds for any $\alpha \in \left(0,\frac{1}{2}\right)$.
	\end{theorem}	
				
	\begin{remark}[Absence of margin convergence] \label{rem: absence.margin.conv}
	We note that, unlike \Cref{thm: no.running} (and \Cref{thm: no.running.lambda}), in \Cref{thm: thm.classification.lambda} (and \Cref{thm: thm.classification.lambda.reg}) we do not provide any result regarding the behavior of the optimal parameters $[w_T,b_T]$ as $T\longrightarrow\infty$ (resp. $[w_\lambda,b_\lambda]$ as $\lambda\searrow0$). Let us elaborate on this absence in the case $\lambda\searrow0$. 
	\begin{itemize}
	\item
	In \citep{wei2019regularization} (see also \citep{rosset2003margin}), the authors show that for the perceptron without bias\footnote{In fact, the result holds for a multi-layer perceptron of arbitrary but fixed depth -- we concentrate on this simpler scenario for presentational purposes.}
	$\Phi(x,u):=w^2\sigma(w^1x)$, 
	where $\sigma$ is, say, $1$--homogeneous, $u=\left[w^1,w^2\right]$ with $w^2\in\R^{m\times d_{\text{hid}}}$ and $w^1\in\R^{d_{\text{hid}}\times d}$, any global minimizer $u_\lambda$ to 
	\begin{equation*}
	J_\lambda(u):=\frac{1}{N}\sum_{i=1}^N-\log\left(\frac{e^{\Phi(\vec{x}_i, u)_{\vec{y}_i}}}{\sum_{j=1}^m e^{\Phi(\vec{x}_i, u)_j}}\right)+\lambda\|u\|^2
	\end{equation*}
	is such that 
	\begin{equation*}
	\gamma_{\overline{u}_\lambda}\xrightarrow[\lambda\searrow0]{} \gamma^*,
	\end{equation*}
	where $\overline{u}_\lambda = \sfrac{u_\lambda}{\|u_\lambda\|}$, $\gamma_{\overline{u}_\lambda}$ is the margin of $\Phi(\cdot, \overline{u}_\lambda)$ as defined in \eqref{eq: margin.def}, and $\gamma^*>0$ is the max-margin: 
	\begin{equation*}
	\gamma^* := \max_{\|u\|\leqslant1} \gamma_u.
	\end{equation*}
	A set of parameters $u^*$ maximizing the margin are then called a \emph{max-margin} solution (here we focus on $\ell^2$--margins; the $\ell^1$--case has also been studied in the machine literature).
	The proof of this fact relies crucially on the specific form of the perceptron, as the map $u\longmapsto\Phi(x,u)$ is positively homogeneous (of degree $2$) for any $x\in\R^d$. By using the exponential scaling of the cross entropy, $J_\lambda(u_\lambda)$ can be lower bounded roughly by $e^{-\|u_\lambda\|^2\gamma_\lambda}$, and also has an upper bound that scales with $e^{-\|u_\lambda\|^2\gamma^*}$. 
	Again making use of the fact that $u\longmapsto\Phi(x,u)$ is positively homogeneous, one can ensure that $\|u_\lambda\|\longrightarrow\infty$ as $\lambda\searrow0$, and then take $\|u_\lambda\|$ large enough so that $\gamma^*-\gamma_\lambda$ vanishes.
	
	In the context of neural ODEs such as \eqref{eq: sigma.outside.i} (and the corresponding ResNet, for that matter), to replicate such ideas, one would need to ensure that the map $[w,b]\longmapsto P\*x(T)$, where $\*x$ solves \eqref{eq: sigma.outside.i}, is positively homogeneous (of some degree). 
	In fact, let us see under which conditions the following relaxed problem, can be solved. We seek a non-decreasing continuous function $f:(0,\infty)\to(0,\infty)$ such that 
	\begin{equation*}
	\*x_\alpha = f(\alpha)\*x \hspace{1cm} \text{ for all } \alpha>0, 
	\end{equation*}
	where $\*x$ solves \eqref{eq: sigma.outside.i} with parameters $[\*w,0]$ (noting \eqref{eq: form.controls}) and $\*x_\alpha$ solves \eqref{eq: sigma.outside.i} with parameters $[\alpha\*w, 0]$. If such an $f$ were to exist, then by differentiating one sees that 
	\begin{equation*}
	\dot{\*x}_\alpha(t) = f(\alpha)\dot{\*x}(t) = f(\alpha)\sigma(\*w(t)\*x(t)).
	\end{equation*}
	Using the fact that $\dot{\*x}_\alpha(t)$ solves \eqref{eq: sigma.outside.i} with parameters $[\alpha\*w,0]$, we moreover see that
	\begin{equation*}
	\sigma(\alpha\*w(t)f(\alpha)\*x(t)) = f(\alpha)\sigma(\*w(t)\*x(t)).
	\end{equation*}
	Now we see that the above relation translates to having
	\begin{equation} \label{eq: f.impossible}
	\sigma(\alpha f(\alpha)s) = f(\alpha) \sigma(s) \hspace{1cm} \text{ for all } \alpha>0,\, s\in\R.
	\end{equation}
	This implies $\sigma(0)=0$, and would also be an obstruction to the Lipschitz continuity of $\sigma$ at $x=0$ (indeed, fix $s\in\R$ such that $\sigma(s)\neq0$, and let $\alpha\searrow0$). Thus, the only increasing function $f:(0,\infty)\to(0,\infty)$ such that \eqref{eq: f.impossible} holds and $\sigma\in\Lip(\R)$ is the zero function.
	\smallskip
	
	\item The issue presented just above does not occur when the only trainable parameter is the additive bias $\*b(t)$. Indeed, if for instance $\*w(t)=\*w_\circ=\mathrm{diag}(w_\circ)$ is fixed, $\*x(t)$ solves \eqref{eq: sigma.outside.i} with parameters $[\*w_\circ,\*b(t)]$ and $\*x_\alpha(t)$ solves \eqref{eq: sigma.outside.i} with parameters $[\*w_\circ,\alpha\*b(t)]$, then one readily sees that 
	$P\*x_{\alpha, i}(t)=\alpha P\*x_i(t)$ for all $\alpha>0$ whenever $P\vec{x}_i=0$ for $i\in[N]$ (this is the case when augmenting the data $\vec{x}_i$ and projecting orthogonally onto the added euclidean space via $P$). Hence, one may show, by adapting the techniques of \citep{wei2019regularization}, that the margin $\gamma_{[w_\circ,\overline{b}_T]}$ corresponding to normalized global minimizers $\overline{b}_T=\sfrac{b_T}{\|b_T\|}$, converges to the max-margin 
	\begin{equation*}
	\gamma^*:=\sup_{\|b\|_{H^1(0,1;\R^{d_u})}\leqslant1}\gamma_{\left[w_\circ,\overline{b}\right]}
	\end{equation*}
	as $T\longrightarrow\infty$ (with an analog result when $\lambda\searrow0$ and $T$ is fixed). We omit the statement and proof of this result due to the possible limitations of not optimizing the weights in practical contexts.
	\end{itemize}
	In view of this discussion, one sees that in order to deduce what the limit for the margin in the fully trained neural ODE (and ResNet) setting would be, one would likely have to conceive a different strategy to that of \citep{rosset2003margin, wei2019regularization} which relies on the homogeneity of the parameter to output map.
	\end{remark}
	
\begin{remark}[Different losses]
	We note that in the context of binary classification, one could consider $\vec{y}_i \in \{-1, 1\}$ and train  with the logistic loss in \eqref{eq: phi.def}
	\begin{equation*}
	\loss(P\*x_i(T), \vec{y}_i) := \log\left(1+e^{-\vec{y}_iP\*x_i(T)}\right),
	\end{equation*}
	where $P:\R^d\to\R$ is an affine map.	
	The statements of \Cref{thm: thm.classification.lambda} and \Cref{thm: thm.classification.lambda.reg} also hold in the context of the logistic loss defined above by a straightforward repetition of our proofs.
	
	Another class of losses when $\vec{y}_i\in[m]$ is the purely exponential loss
	\begin{equation*}
	\loss(P\*x_i(T),\vec{y}_i) := e^{-\vec{y}_iP\*x_i(T)}.
	\end{equation*}
	Note that, again, by a slight adaptation of our techniques, one can show an analog variant of \Cref{thm: thm.classification.lambda} (and \Cref{thm: thm.classification.lambda.reg}) with an estimate of the form
	\begin{equation*}
	\mathscr{E}(\*x_T(T)) \leqslant e^{-\gamma e^{\frac{T^\alpha}{2}}} + \mathfrak{C}\,T^{2\alpha-1}.
	\end{equation*} 
	\end{remark}

	\section{Augmented empirical risk minimization} \label{sec: turnpike.main}
	
	\noindent
	We are now interested in seeing whether one can obtain better quantitative estimates for the decay of the training error $\mathscr{E}$ to $0$ with respect to the time horizon ($\sim$ number of layers) $T>0$ -- namely, improve the $\mathcal{O}\left(\frac{1}{T}\right)$--rate of convergence of the training error to $0$ manifested in \Cref{thm: no.running} and \Cref{thm: thm.classification.lambda}. 
	We provide a proof of exponential stability in the setting of an augmented supervised learning problem, concentrating solely on $L^2$--parameter regularization (and hence we shall assume the existence of a minimizer in the setting of \eqref{eq: sigma.outside.i}). 
	We shall also consider a $\BV$--parameter regularization, which suffices to ensure the existence of minimizers, but only provide a polynomial convergence of the averages of the training error and optimal parameters to zero as $T\longrightarrow\infty$, whilst ensuring a uniform bound on the total variation of the optimal parameters. 
	\smallskip
	
	\noindent
	Unless stated otherwise, we will henceforth solely concentrate on empirical risks 
	\begin{equation}  \label{eq: loss.turnpike}
	\mathscr{E}(\*x) := \frac{1}{N} \sum_{i=1}^N \loss(P\*x_i,\vec{y}_i) \hspace{1cm} \text{ for } \*x \in \R^{d_x},
	\end{equation}
	induced by a loss satisfying the following condition.
	
	\begin{assumption} \label{ass: loss.turnpike}
	We assume that $\mathcal{Y}\subset\R^m$ and $\loss\in C^0(\R^m\times\mathcal{Y};\R_+)$ is such that there exist constants $c>0$ (possibly depending on $m$) and $\alpha>0$ such that
	\begin{equation*}
	\loss(x,y) \leqslant c\, \|x-y\|_{\ell^2}^\alpha \hspace{1cm} \text{ for } x,y\in\R^m.
	\end{equation*}
	\end{assumption}
	
	\noindent
	Typical examples of such losses are $\loss(x,y)=\|x-y\|_{\ell^p}^p$ for $p\geqslant1$ and $\mathcal{Y}\subset\R^m$; note however that this assumption excludes the log-loss and the cross-entropy loss. 
	In \eqref{eq: loss.turnpike}, $P\in\Lip(\R^d; \R^m)$ is any given surjective and non-zero map, which, in the context of regression tasks, is additionally an affine map, while in the context of binary classification, e.g. $\vec{y}_i\in\mathcal{Y}=\{-1,1\}$, may be an affine map composed with a thresholding nonlinearity with range $[-1,1]$. 
	\smallskip
	
	\noindent
	Inspired from insights in optimal control theory, for fixed $\lambda>0$, we will study the behavior when $T\gg 1$ of global minimizers to the functional
	\begin{equation} \label{eq: time.dep.func}
	J_{T}(w,b) := \mathscr{E}(\*x(T))+\frac{1}{N}\int_0^T \|\*x(t)-\overline{\*x}\|^2 \diff t +\lambda  \Big\|[w,b]\Big\|_{L^2(0,T;\R^{d_u})}^2,
	\end{equation}
	with $\mathscr{E}$ as in \eqref{eq: loss.turnpike}--\Cref{ass: loss.turnpike}, and where $\overline{\*x}_i\in P^{-1}(\{\vec{y}_i\})$ for all $i\in[N]$ are arbitrary but fixed. 
	
	We note that, contrary to the case where we minimize the training error at the final time $T$, here the same scaling does not appear, which allows us to deduce an equivalence with $\lambda\longrightarrow0$. Hence, we will solely be interested in the behavior when $T\gg1$. In fact, we ought to expect a result of slightly different nature to \Cref{thm: no.running}. 
	This is because the trajectory tracking term introduces a stronger time-scale in the behavior of the optimization problem as $T\gg1$. 
	To see this, consider the neural ODE \eqref{eq: sigma.inside.i} for simplicity, and note that 
	 \begin{align}
	 &\inf_{\substack{u_T \in L^2(0,T;\R^{d_u}) \\ \*x_T \text{ solves } \eqref{eq: aux.idea.sys}}} \mathscr{E}(\*x_T(T))+\frac{1}{N}\int_0^T \|\*x_T(t)-\overline{\*x}\|^2\diff t +  \lambda \int_0^T \left\|u_T(t)\right\|^2 \diff t  \nonumber \\
	 &= \inf_{\substack{u_T\in L^2(0,T;\R^{d_u}) \\ \*x_T \text{ solves } \eqref{eq: aux.idea.sys}}} \mathscr{E}(\*x_T(T)) + \frac{T}{N} \int_0^1 \left\|\*x_T\left(\frac{s}{T}\right)-\overline{\*x}\right\|^2 \diff s + \frac{\lambda}{T} \int_0^1 \left\|Tu_T(sT)\right\|^2 \diff s \nonumber\\
	 &= \inf_{\substack{u^1 \in L^2(0,1;\R^{d_u}) \\ \*x^1 \text{ solves } \eqref{eq: aux.idea.sys1}}} \mathscr{E}\left(\*x^1(1)\right) + \frac{T}{N} \int_0^1 \|\*x^1(s)-\overline{\*x}\|^2\diff s + \frac{\lambda}{T} \int_0^1 \left\| u^1(s) \right\|^2\diff s \label{eq: 3.20}.
	  \end{align}
	We see that, unlike \Cref{thm: no.running}, the trajectory tracking term in \eqref{eq: 3.20} carries significance when $T\gg 1$.	
	\smallskip
	
	\noindent
	We will require the following \emph{controllability} notion, which is rather natural in the context of the result that follows.
	
	\begin{definition}[Controllability with linear cost]
	\label{def: ctrl}
	We say that 
	\begin{itemize}
	\item System \eqref{eq: sigma.inside.i} is \emph{controllable with linear cost} if for any $T>0$ and $r>0$ there exists a constant $\mathfrak{C}(T,r)>0$ such that for any $\*x^0\in\R^{d_x}$ and $\*x^1\in\R^{d_x}$ satisfying $\left\|\*x^0-\*x^1\right\|\leqslant r$, there exists a pair of parameters $[w,b] \in L^2(0,T; \R^{d_u})$ for which the corresponding unique solution $\*x(\cdot)$ to \eqref{eq: sigma.inside.i}, noting \eqref{eq: form.controls}, satisfies $\*x(T) = \*x^1$, and 
	\begin{align}  \label{eq: linear.control.cost.1} 
	\Big\|[w,b]\Big\|_{L^2(0, T; \R^{d_u})} \leqslant \mathfrak{C}(T,r) \left\|\*x^0-\*x^1\right\| 
	\end{align}
	holds.

	\smallskip
	\item System \eqref{eq: sigma.outside.i} is \emph{controllable with linear cost}  if for any $T>0$ and $r>0$ there exists a constant $\mathfrak{C}(T,r)>0$ such that for any $\*x^0\in\R^{d_x}$ and $\*x^1\in\R^{d_x}$ satisfying $\left\|\*x^0-\*x^1\right\|\leqslant r$, there exists a pair of parameters $[w,b]\in C^0([0,T];\R^{d_u})\cap\BV(0,T;\R^{d_u})$  for which the corresponding unique solution $\*x(\cdot)$ to \eqref{eq: sigma.outside.i}, noting \eqref{eq: form.controls}, satisfies $\*x(T) = \*x^1$, and
	\begin{align}
	\Big\|[w,b]\Big\|_{L^\infty(0, T; \R^{d_u})}+\Big\|[w,b]\Big\|_{\BV(0, T; \R^{d_u})} \leqslant \mathfrak{C}(T,r) \left\|\*x^0-\*x^1\right\|  \label{eq: linear.control.cost.2} 
	\end{align}
	holds.
	\end{itemize}
	\end{definition}
	
	\noindent
	We refer to \Cref{prop_Nleqd+1} for further analysis and comments regarding \Cref{def: ctrl}, in particular regarding the amplitude estimates \eqref{eq: linear.control.cost.1} -- \eqref{eq: linear.control.cost.2}. Note that \eqref{eq: linear.control.cost.2} is a weaker condition than assuming an $H^1(0,T;\R^{d_u})$--estimate due to the Sobolev embedding. On another hand, due to the homogeneity of the dynamics, one readily sees (as per \Cref{lem: scaling}) that the controllability assumption holds in any time $T>0$ if it holds in some time $T_0>0$.

	We refer to  \eqref{eq: linear.control.cost.1} -- \eqref{eq: linear.control.cost.2} as \emph{linear estimates on the cost} of $[w,b]$ as, in addition to the linear growth with respect to $\left\|\*x^0-\*x^1\right\|$ such estimates are a hallmark of linear controlled dynamical systems of the form $\dot{x}(t)=Ax(t)+Bu(t)$, where $u(t)\in\R^m$ is the control, $x(t)\in\R^n$ is the state, with $A\in\R^{n\times n}$ and $B\in\R^{n\times m}$ being given matrices (see \citep{zuazua2007controllability}).
	
	We are in a position to state our main result in the context of the augmented supervised learning problem consisting of minimizing \eqref{eq: time.dep.func}.
		
	\begin{theorem}[Exponential stability] 
	\label{thm: turnpike.P}
	Fix $\lambda>0$, let $P \in \Lip(\R^d;\R^m)$ be any given non-zero surjective map, and let $\overline{\*x} \in \R^{d_x}$ with $\overline{\*x}_i \in P^{-1}(\{\vec{y}_i\})$ for $i\in[N]$ be arbitrary but fixed.
	Suppose that system \eqref{eq: sigma.inside.i} (resp. \eqref{eq: sigma.outside.i} with $\sigma$ $1$--homogeneous) is controllable with linear cost in the sense of \Cref{def: ctrl}. 
	Then, there exists $T^*>0$ and constants $\mathfrak{C}=\mathfrak{C}\left(\{\vec{x}_i, \vec{y}_i\}_{i\in[N]},\lambda,N,\alpha, P\right)>0$ and $\mu=\mu\left(\{\vec{x}_i, \vec{y}_i\}_{i\in[N]},\lambda,N,\alpha\right)>0$ such that for any $T\geqslant T^*$, any pair of parameters $\left[w_T, b_T\right] \in L^2(0,T; \R^{d_u})$ minimizing $J_T$ defined in \eqref{eq: time.dep.func}, and the corresponding unique solution $\*x_T(\cdot)$ to \eqref{eq: sigma.inside.i} (resp. \eqref{eq: sigma.outside.i}) satisfy
	\begin{equation*} 
	\big\|w_T(t)\big\|+\big\|b_T(t)\big\| \leqslant \mathfrak{C}\,e^{-\mu t}
	\end{equation*}
	for a.e. $t \in [0,T]$ and
	\begin{equation*} 
	\mathscr{E}(\*x_T(t)) + \|\*x_T(t)-\overline{\*x}\| \leqslant \mathfrak{C}\,e^{-\mu t}
	\end{equation*}
	for all $t \in [0,T]$.
	\end{theorem}	
	
	\noindent
	The convergence/stability rate entailed by \Cref{thm: turnpike.P} is not only noticeably stronger than those in \Cref{thm: no.running} and \Cref{thm: thm.classification.lambda}, but the estimate holds in any time $t\in[0,T]$ (i.e., at every layer when viewed from the ResNet perspective) and not only for the output features.
	
	\Cref{thm: turnpike.P} can be proven by following the framework presented\footnote{As discussed in the introduction, \Cref{thm: turnpike.P} is a specific manifestation of the so-called \emph{turnpike property}, a paradigm dating back to the works of John von Neumann (\citep{von1945model}), and works in economics by Paul Samuelson et al. (\citep{Samuelson1}).
A turnpike theory under smallness assumptions on the inputs and/or targets, combining the Pontryagin Maximum Principle, linearization arguments and precise estimates on Riccati equations, and covering a wide variety of nonlinear optimal control problems is developed in \citep{trelat2015turnpike} -- with extensions to Lipschitz nonlinearities and avoiding smallness conditions found in \citep{esteve2020turnpike}. } in \citep{esteve2020turnpike}. 
We provide most of the arguments in \Cref{sec: proof.turnpike} due to technical changes in the underlying model.

	\begin{remark}[Existence of minimizers, $\BV$--regularization] \label{rem: BV.need}
	In \Cref{thm: turnpike.P} we have subjacently assumed the existence of a minimizer of \eqref{eq: time.dep.func} for the neural ODE \eqref{eq: sigma.outside.i}. By solely regularizing the $L^2$--norm of the parameters, we are, a priori, not aware how to ensure the strong convergence in $L^1$ of a minimizing sequence of parameters, which would suffice for applying the direct method in the calculus of variations. 
	On the other hand, using a Sobolev $H^1$--regularization would entail that optimizable parameters are continuous (by the Sobolev embedding) and thus render our specific proof strategy incompatible, as we rely on constructing discontinuous suboptimal parameters by zero extensions
	
	In view of this, in the setting of \eqref{eq: sigma.outside.i} we may also make use of a $\BV$--regularization rather than $H^1$, as while either of these regularizations suffice to guarantee compactness of the flow, which is sufficient to ensure the existence of a minimizer, the $\BV$--regularization allows for discontinuous parameters: one would then rather consider 
	\begin{equation} \label{eq: time.dep.func.bv}
	J_{T}(w,b) := \mathscr{E}(\*x(T))+\frac{1}{N}\int_0^T \|\*x(t)-\overline{\*x}\|^2 \diff t +\lambda  \Big\|[w,b]\Big\|_{\BV(0,T;\R^{d_u})}^2.
	\end{equation}
	\end{remark}
	
	\noindent
	Taking the above remark into consideration, we are able to prove the following partial stability and convergence result in the setting of $\BV$--regularized parameters. We focus on \eqref{eq: sigma.outside.i}, since \eqref{eq: sigma.inside.i} can be covered by means of $L^2$--regularization.
	
	\begin{theorem}[Convergence of averages, stability] \label{thm: turnpike.BV}
	Fix $\lambda>0$, let $P \in \Lip(\R^d;\R^m)$ be any given non-zero surjective map, and let $\overline{\*x} \in \R^{d_x}$ with $\overline{\*x}_i \in P^{-1}(\{\vec{y}_i\})$ for $i\in[N]$ be arbitrary but fixed.
	Suppose that \eqref{eq: sigma.outside.i}, with $\sigma$ $1$--homogeneous, is controllable with linear cost in the sense of \Cref{def: ctrl}. 
	The following facts then hold.
	\begin{enumerate}
	\item (Stability). There exists a constant $\mathfrak{C}=\mathfrak{C}\left(\{\vec{x}_i, \vec{y}_i\}_{i\in[N]}, \left\|\*x^0-\overline{\*x}\right\|, \lambda, N, P\right)>0$ such that for any $T\geqslant 1$, any pair of parameters $\left[w_T, b_T\right] \in \BV(0,T; \R^{d_u})$ minimizing $J_T$ defined in \eqref{eq: time.dep.func.bv}, and the corresponding unique solution $\*x_T(\cdot)$ to \eqref{eq: sigma.outside.i} satisfy
	\begin{equation*} 
	\mathscr{E}(\*x_T(t))+\|\*x_T(t)-\overline{\*x}\|+\Big|\mathrm{D}\left[w_T,b_T\right]\Big|(0,T) \leqslant \mathfrak{C} 
	\end{equation*}
	for all $t \in [0,T]$, where $\left|\mathrm{D}\left[w_T,b_T\right]\right|(0,T)$ denotes the total variation of $[w_T,b_T]$.
	\smallskip
	\item (Convergence of averages). Any pair of parameters $\left[w_T, b_T\right] \in \BV(0,T; \R^{d_u})$ minimizing $J_T$ defined in \eqref{eq: time.dep.func.bv}, and the corresponding unique solution $\*x_T(\cdot)$ to \eqref{eq: sigma.outside.i} also satisfy
	\begin{equation*}
	\frac{1}{T} \int_0^T \mathscr{E}(\*x_T(t))\diff t+ \frac{1}{T}\int_0^T \|\*x(t)-\overline{\*x}\|\diff t + \frac{1}{T} \int_0^T \Big\|[w_T(t),b_T(t)\Big\|\diff t \xrightarrow[T\longrightarrow\infty]{} 0.
	\end{equation*}
	\end{enumerate}

	\end{theorem}
	
	\noindent
	The convergence of averages are very much related to the so-called \emph{integral turnpike} property, which has been studied in some optimal control contexts (see e.g. \citep{trelat2018integral}). Note that, in addition to these convergences, the stability estimates of item \emph{(i)} entail that the oscillations of the training error may be controlled, in any time $t$, uniformly with respect to the time horizon $T$. Moreover, the total variation of the optimal parameters is also uniformly bounded with respect to $T$.
	
	We are unable to provide exponential stability estimates as in \Cref{thm: turnpike.P} due to some constructions specific to the proof of \Cref{thm: turnpike.P}, wherein suboptimal parameters with jumps at time instances are constructed. While the $L^2$--norm does not see these jump singularities, the $\BV$--norm does and renders our strategy incompatible.
	
	\begin{remark}[Extensions] We only stated \Cref{thm: turnpike.P} for neural ODEs of the form \eqref{eq: sigma.inside.i}, or \eqref{eq: sigma.outside.i} with $\sigma$ $1$--homogeneous. This is solely to guarantee the exponential stability estimate of the optimal parameters, for which our proof requires using the underlying scaling endowed by the homogeneity of the dynamics. But in fact, the exponential stability estimate of the training error $\mathscr{E}(\*x_T(T))$ and a uniform-in-$T$ bound of the optimal parameters can be shown for more complicated neural ODE dynamics such as \eqref{eq: compound.neural.ode}, solely by a small adaptation of the proof.
	\end{remark}
		
	\begin{remark}[Dependence on $N$]
	All of the results presented in what precedes hold for a fixed but otherwise arbitrary number of data samples $N$.
	In \Cref{thm: turnpike.P}, one also notes an explicit dependence of the constants $\mathfrak{C}>0$ and $\mu>0$ on $N$ -- in fact, both constants might have a tendency to depend in an exponential manner with respect to $N$ due to the subjacent application of a Gr\"onwall inequality for the stacked neural ODE system, for which the parameters are pasted $N$ times. 
	The dependence on $N$ provided in our proof might not be sufficient or sharp for studying a possible large data limit (e.g., via a law of large numbers argument of some kind) -- we leave this issue open.
	\end{remark}
	
	\begin{example}[A numerical experiment] \label{ex: num.exp.turnpike}
	In \Cref{fig1.turnpike} -- \Cref{fig2.turnpike}, we depict\footnote{
	Software experiments were done using \texttt{PyTorch} \citep{paszke2017automatic} (and may be found at \href{https://github.com/borjanG/dynamical.systems}{\textcolor{dukeblue}{\texttt{https://github.com/borjanG/dynamical.systems}}}), using the Adam optimizer \citep{kingma2014adam} with learning rate equal to $10^{-3}$. Experiments were conducted on a personal MacBook Pro laptop (2.4 GHz Quad-Core Intel Core i5, 16GB RAM, Intel Iris Plus Graphics 1536 MB).} a manifestation of the exponential stability insinuated by \Cref{thm: turnpike.P} on a toy binary classification task (namely $\vec{y}_i \in \{-1,1\}$) with $N=2400$ training samples and $600$ test samples, where $P(\cdot) = \tanh(p_1 \cdot + p_2)$ with $p_1\in\R^{1\times2}$ and $p_2\in\R$ randomly sampled from a uniform distribution on $[0,10]$. Note that, while in theory $\tanh:\R\to[-1,1]$ is not surjective (it is only bijective onto $(-1,1)$), we use it as a thresholding nonlinearity for simplicity, as of course numerically the lack of surjectivity does not appear due to floating point accuracy.
	 
	To discretize the full continuous-time optimization problem, we use direct shooting, which is a \emph{first discretize then optimize} approach. 
	We consider the neural ODE \eqref{eq: sigma.inside.i} with $\sigma(x) = \tanh(x)$ (we use the ResNet \eqref{eq: resnet.discrete}), with $T=15$ (and thus $15$ layers) and $\lambda=0.01$. Finally, we discretize the integrals using an elementary trapezoidal quadrature. 
	We note that the learned flow has a distinctly simple variation in \Cref{fig2.turnpike}, and, albeit on a toy task, we observe satisfactory generalization properties in \Cref{fig3.turnpike}.
	\end{example}
	
	\begin{figure}[h]
\includegraphics[scale=0.55]{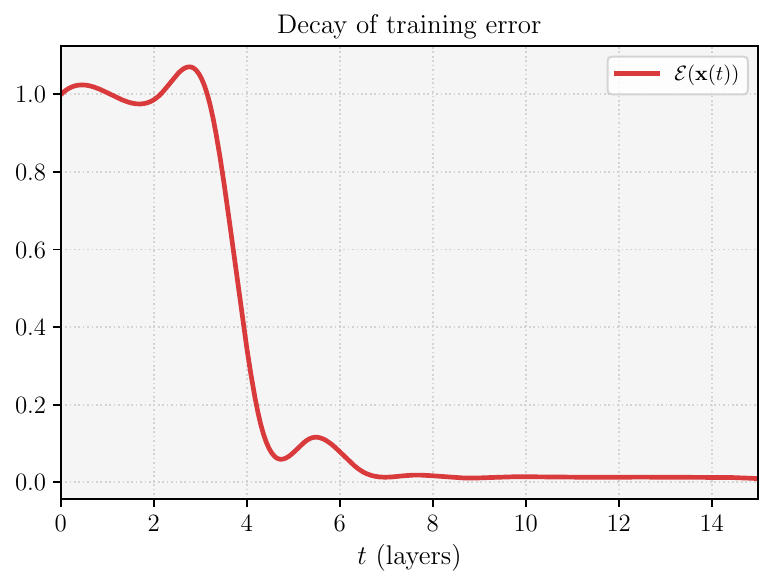}
\includegraphics[scale=0.55]{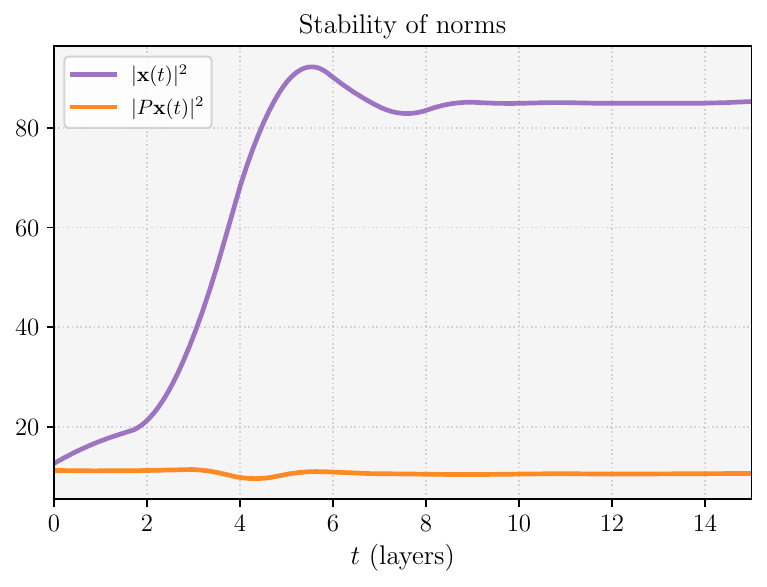}
\caption{\textbf{\Cref{ex: num.exp.turnpike}:} We depict a manifestation of the stability results of \Cref{thm: turnpike.P} for the state $\*x_T(t)$ (\textit{right}) and the training error $\mathscr{E}(\*x_T(t))$ (\textit{left}) over $t \in [0,T]$. We observe that, after a finite time, the training error and trajectory remain at a steady configuration, so further times/layers could be discarded from training.}
\label{fig1.turnpike}
\end{figure}

\begin{figure}
\hspace{0.5cm} {\small $t=0$} \hspace{6cm} {\small $t\leqslant15$}

\includegraphics[scale=0.525]{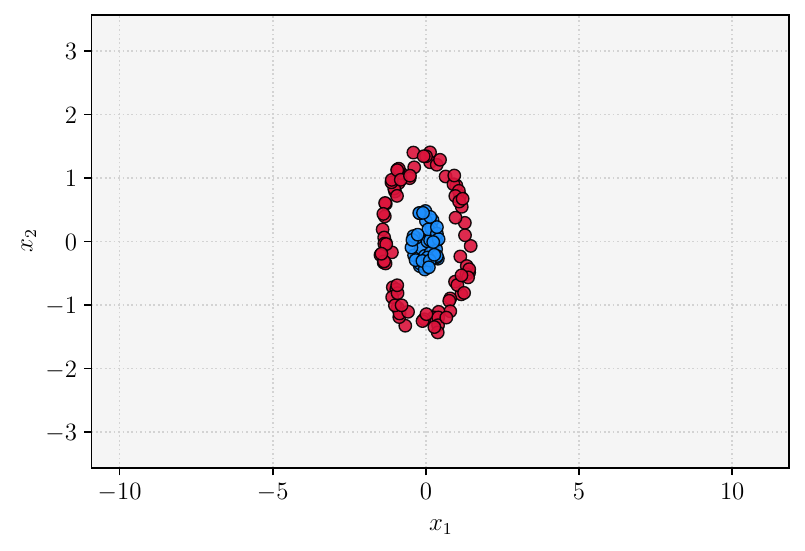}
\includegraphics[scale=0.525]{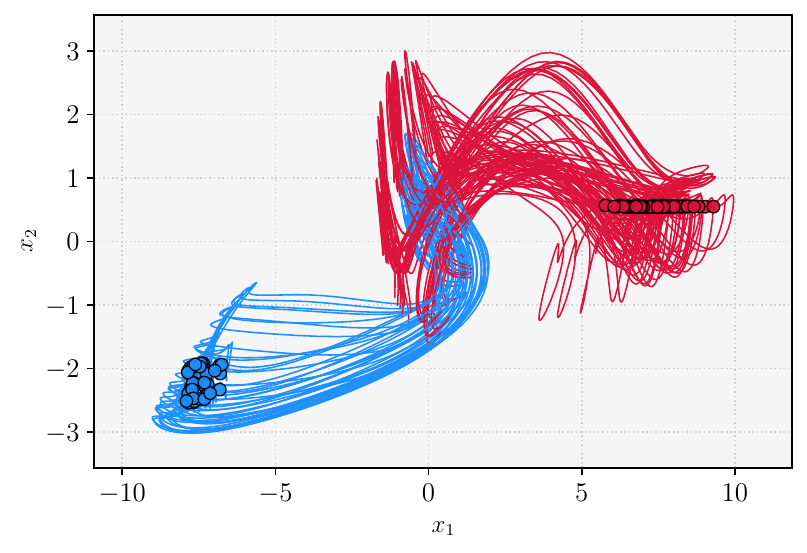}
\caption{\textbf{\Cref{ex: num.exp.turnpike}:} A batch of training data (\textit{left}) and the evolution of the corresponding trained trajectories $\*x_{T,i}(t)$ (\textit{right}) in the phase plane. The learned flow is simple, with moderate variations, due to the exponentially small parameters. We refer to \href{https://github.com/borjanG/dynamical.systems/blob/master/videos/example4-1.mp4}{\textcolor{dukeblue}{\texttt{https://github.com/borjanG/dynamical.systems/blob/master/\\
videos/example4-1.mp4}}} for a movie of the evolution of the trajectories, where the stability phenomenon depicted in \Cref{fig1.turnpike} can be visualized.}
\label{fig2.turnpike}
\end{figure}

\begin{SCfigure}[1][h]
	\caption{\textbf{\Cref{ex: num.exp.turnpike}:} The trained classifier on $[-2,2]^2$ and evaluated on a batch of the test set. The simplicity of the learned flow ensures relatively satisfactory generalization as the shape of the dataset is learned adequately, and the test set is correctly classified.}
	\includegraphics[scale=0.55]{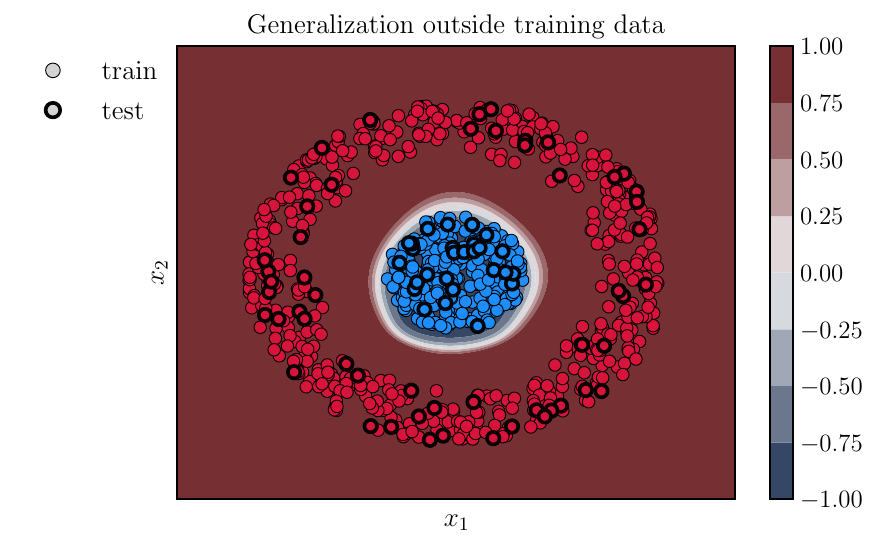}
	  \label{fig3.turnpike}
\end{SCfigure}
			    
	 \subsection{The motivating problem} \label{sec: motivating.problem} 
	Due to the specific nature of the proof of \Cref{thm: turnpike.P} we have restricted our study to a trajectory tracking term consisting of the squared $L^2(0,T; \R^{d_x})$--norm, even-though the final cost $\mathscr{E}(\*x_T(T))$ allows us to address both classification and regression tasks. 
	However, having to look for targets $\overline{\*x}$ in the preimage of the labels $\vec{y}_i$ by $P$ for any general task may not scale computationally and might bias the prediction.  
	\smallskip
	
	\noindent
	Interestingly enough, at least numerically (any analytical result remains an open problem), we observe that the stabilization phenomenon persists (although, in theory, perhaps not with the same rate) when the term $\|\*x(t)-\overline{\*x}\|^2$ is replaced by the training error $\mathscr{E}(\*x(t))$ with a general and possibly non-coercive loss, for instance, the cross-entropy loss on a multi-label classification tasks as seen in \Cref{fig1.ex3} \& \Cref{fig: figure.mnist.turnpike.1}.
        In fact, one could stipulate that this stabilization phenomenon (be it exponential or not, but in any $t\in[0,T]$ rather than just for the output features at time $T$) holds for global minimizers of functionals of the form
        	\begin{equation} \label{eq: time.dep.func.wanted}
	J_{T}(w,b) := \int_0^T \mathscr{E}(\*x(t)) \diff t +\lambda \Big\|[w,b]\Big\|_{\mathscr{H}(0,T;\R^{d_u})}^2,
	\end{equation}
	with $\mathscr{E}$ as in \eqref{eq: phi.def} and $\loss$ being continuous and nonnegative, but otherwise arbitrary; $\mathscr{H}$ is for instance $L^2$ or $\BV$. 
	We perform several numerical experiments to justify\footnote{We do not insinuate that our numerical experiments are comparable with state of the art configurations, as we only look to motivate and illustrate the mathematical phenomena studied in what precedes. Indeed, we generally make use of a forward Euler scheme compared to more advanced adaptive schemes, as used for instance in \citep{chen2018neural}. Our experiments should rather be seen as proof of concept.} this claim (see \Cref{ex: 3} -- \Cref{ex: 6}). 
	
	Note that in these experiments, even-though the loss is taken as cross-entropy and is thus not coercive, in addition to a stability property for $\mathscr{E}(\*x(t))$ to $0$ we also see (e.g. in \Cref{fig1.ex3}) that the trajectories $\*x(t)$ and features $\{P\*x_i(t)\}_{i\in[N]}$ stabilize towards some targets in sufficiently large time. 
	However, due to the fact that $\mathscr{E}$ does not attain its minimizer, it is a priori not clear how one may characterize the targets to which $\*x(t)$ and features $\{P\*x_i(t)\}_{i\in[N]}$ stabilize. We refer to \citep{yague2021sparse} for a result illustrating a similar stability phenomenon in the setting of \eqref{eq: time.dep.func.wanted} with $L^1$--parameter regularization.
	
	\begin{example}[Concentric spheres]
	\label{ex: 1}
	For comparison purposes, we consider an identical dataset setting to that of \Cref{fig1.turnpike} -- \Cref{fig3.turnpike}.
	We consider the neural ODE \eqref{eq: sigma.inside.i} with $\sigma\equiv\tanh$, cross-entropy loss and $\lambda=0.01$, with the output layer having the form $Px = p_1x+p_2$, with $p_1\in\R^{2\times2}, p_2\in\R^2$ both being part of the trainable parameters.
	We visualize the output of the experiments in \Cref{fig1.ex1} -- \Cref{fig3.ex1} below.
	
		\begin{figure}[h]
\includegraphics[scale=0.54]{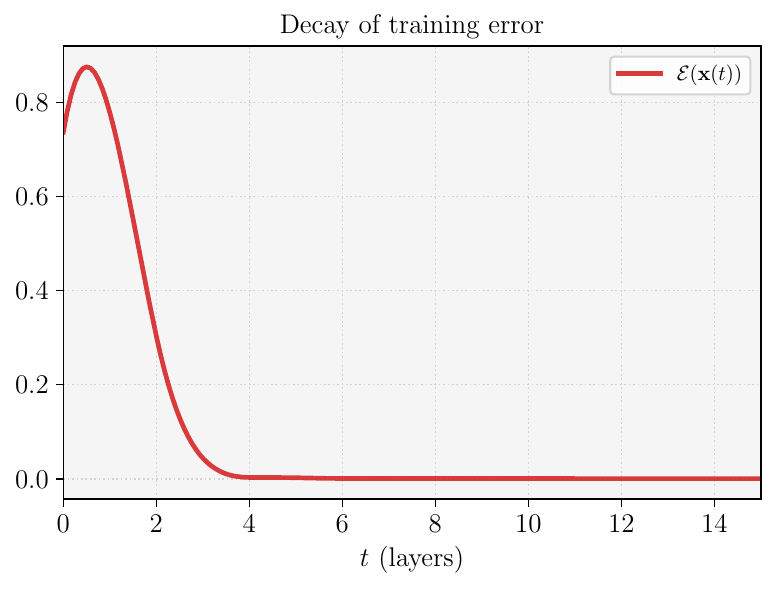}
\includegraphics[scale=0.54]{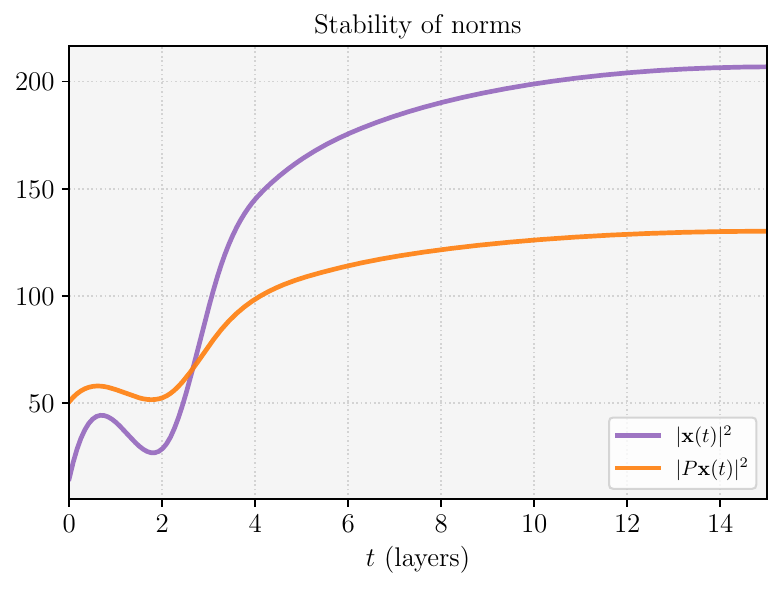}
\caption{\textbf{\Cref{ex: 1}:} The decay of the training error (\textit{left}) and stabilization of the trained trajectories $\*x(t)$ and $\{P\*x_i(t)\}_{i\in[N]}$ (\textit{right}).}
\label{fig1.ex1}
\end{figure}

\begin{figure}
\hspace{0.5cm} {\small $t=0$} \hspace{6cm} {\small $t\leqslant15$}

\includegraphics[scale=0.515]{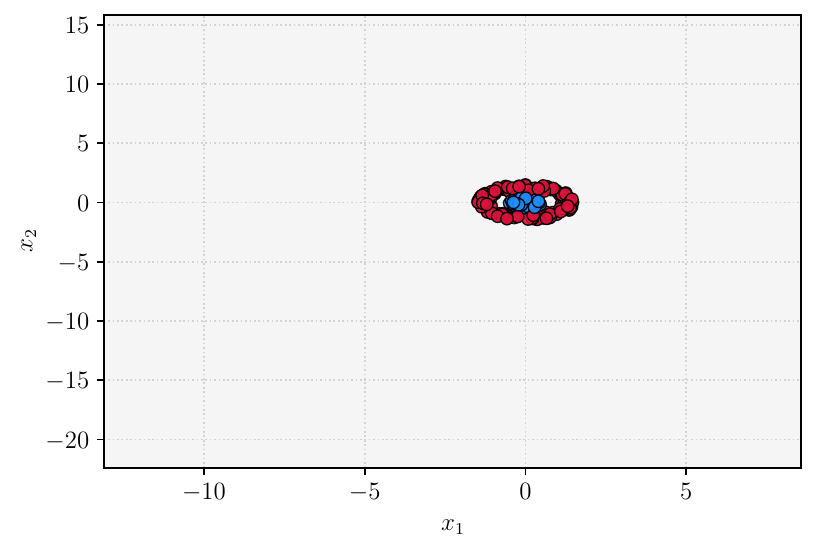}
\includegraphics[scale=0.515]{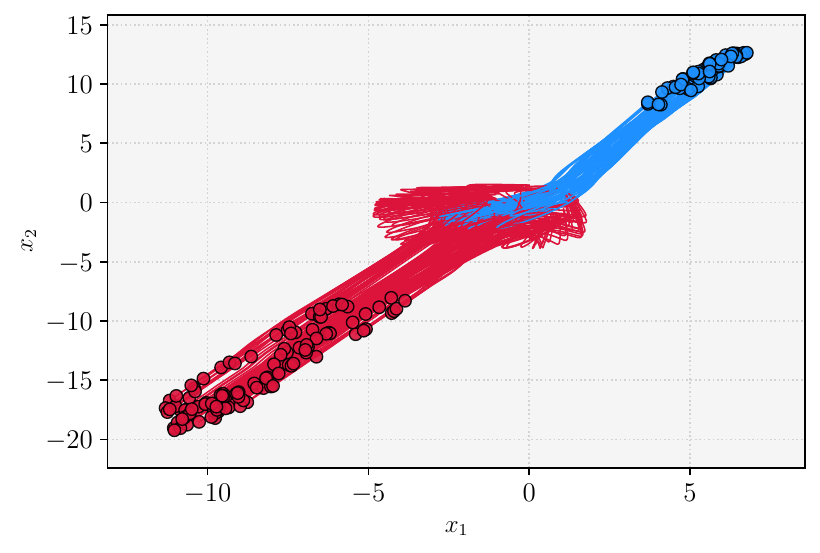}
\caption{\textbf{\Cref{ex: 1}:} A batch of training data (\textit{left}) and the evolution of the corresponding trained trajectories $\*x_{T,i}(t)$ (\textit{right}).
See \href{https://github.com/borjanG/dynamical.systems/blob/master/videos/exampleA-1.mp4}{\textcolor{dukeblue}{\texttt{https://github.com/borjanG/dynamical.systems/blob/master/\\
videos/exampleA-1.mp4}}} for a movie of the evolution of the trajectories, where the stability phenomenon depicted in \Cref{fig1.ex1} can be visualized.}
\label{fig2.ex1}
\end{figure}

\begin{SCfigure}[1][h]
{\caption{\textbf{\Cref{ex: 1}:} The trained classifier on $[-2,2]^2$ and its evaluation on a batch of the test data. The shape of the data is captured accurately.}} {\includegraphics[scale=0.54]{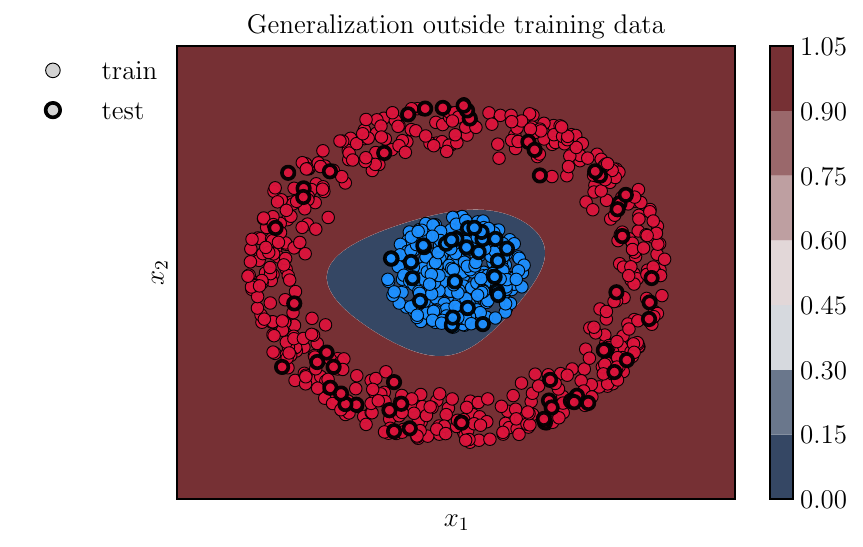}}
\label{fig3.ex1}
\end{SCfigure}

\end{example}

	\begin{example}[Annuli in three colors] \label{ex: 3}
	
	We consider a toy classification task with three labels, namely $\vec{y}_i \in [3]$, with each label corresponding to a different color. The dataset consists of $N=3200$ training samples and $800$ test samples.
	We consider the cross-entropy loss \eqref{eq: cross.entropy.def} in the training error in \eqref{eq: time.dep.func.wanted} and $\lambda=0.01$, and we consider the neural ODE \eqref{eq: sigma.inside.i} with $T=15$ (we use a forward Euler scheme to obtain a corresponding ResNet with fixed time-step equal to $1$), where $\sigma\equiv\tanh$. No augmentation of the initial data is used, and the trajectories evolve in the ambient dimension $d=2$.
	The output layer is parametrized by $Px = p_1x+p_2$, where $p_1\in\R^{3\times2}$, $p_2\in\R^3$ are part of the trainable parameters. 
	We display the results of the experiments in \Cref{fig1.ex3} -- \Cref{fig3.ex3}.
	
	\begin{figure}[h]
	\includegraphics[scale=0.54]{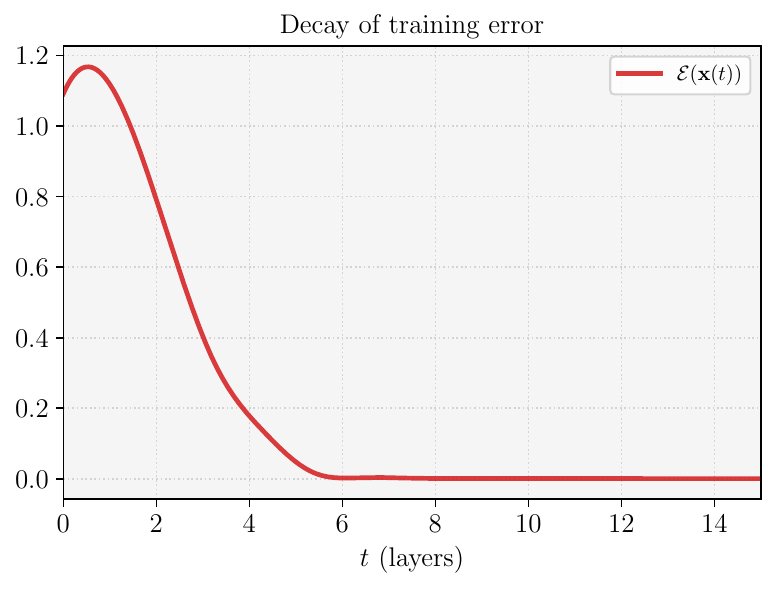}
\includegraphics[scale=0.54]{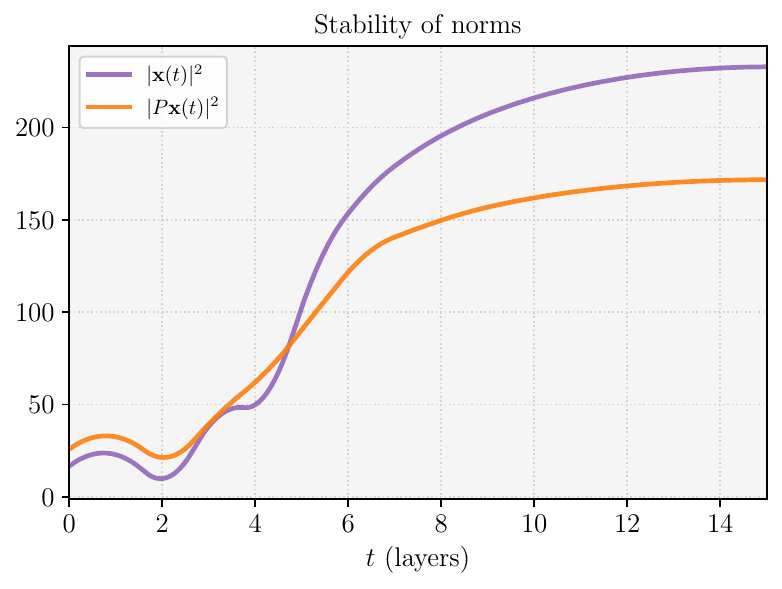}
\caption{\textbf{\Cref{ex: 3}:} The decay of the training error (\textit{left}) and stabilization of the trained trajectories $\*x(t)$ and $\{P\*x_i(t)\}_{i\in[N]}$ (\textit{right}).}
\label{fig1.ex3}
\end{figure}

\begin{figure}
\hspace{0.5cm} {\small $t=0$} \hspace{6cm} {\small $t\leqslant15$}

\includegraphics[scale=0.515]{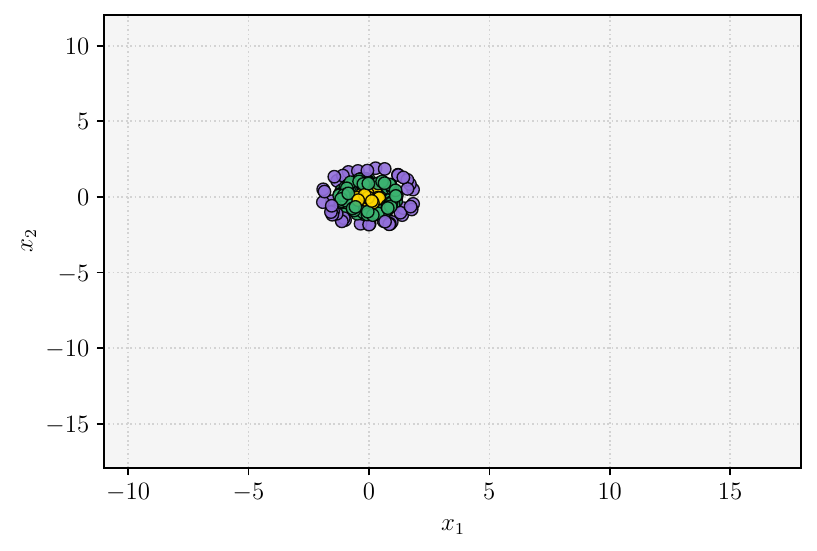}
\includegraphics[scale=0.515]{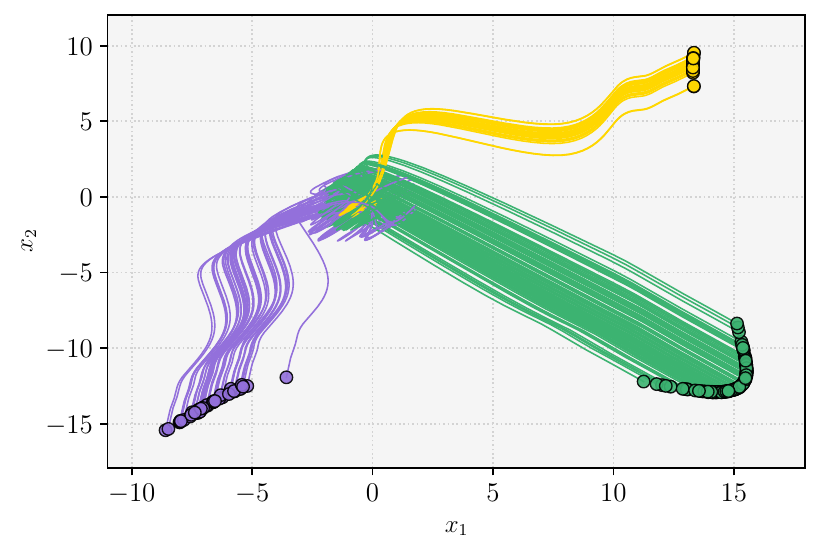}
\caption{\textbf{\Cref{ex: 3}:} A batch of training data (\textit{left}) and the evolution of the corresponding trained trajectories $\*x_{T,i}(t)$ (\textit{right}).
See \href{https://github.com/borjanG/dynamical.systems/blob/master/videos/example4-2.mp4}{\textcolor{dukeblue}{\texttt{https://github.com/borjanG/dynamical.systems/blob/master/
videos/example4-2.mp4}}} for a movie of the evolution of the trajectories, where the stability phenomenon depicted in \Cref{fig1.ex3} can be visualized.}
\label{fig2.ex3}
\end{figure}

\begin{SCfigure}[1][h]
	\caption{\textbf{\Cref{ex: 3}:} The trained classifier on $[-2.5,2.5]^2$ and its evaluation on a batch of the test data. The shape of the dataset is captured accurately.}
{\includegraphics[scale=0.54]{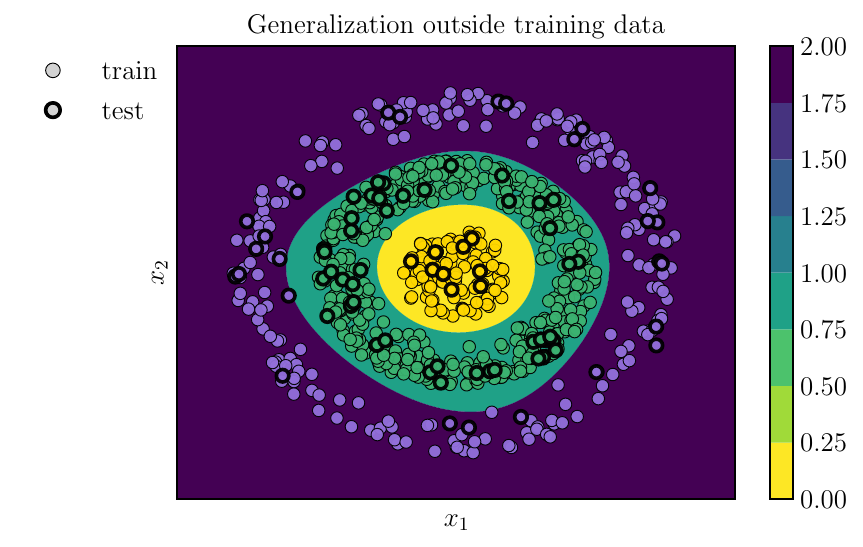}}
\label{fig3.ex3}
\end{SCfigure}

\end{example}

\begin{example}[XOR] \label{ex: 2}
	
	We consider a canonical binary classification task ($\vec{y}_i\in[2]$) -- the XOR dataset (\Cref{fig2.ex2}) consisting of $N=3200$ training samples and $800$ test samples. 
	To further illustrate the genericity of the stability phenomenon, we now consider the neural ODE \eqref{eq: compound.neural.ode}, with $T=15$ (we use a forward Euler scheme to obtain a corresponding ResNet with fixed time-step equal to $1$), where $\sigma\equiv\tanh$ and $d_{\text{hid}}=3$. We again consider cross-entropy loss in the empirical risk $\mathscr{E}$ defined in \eqref{eq: time.dep.func.wanted} and $\lambda=0.01$, and we focus solely on $L^2$--parameter regularization (for simplicity). No augmentation of the initial data is used, and the trajectories evolve in the ambient dimension $d=2$.
	The output layer is defined as $Px = p_1x+p_2$, with $p_1\in\R^{2\times2}$, $p_2\in\R^2$ both being part of the trainable parameters.
	We display the results in \Cref{fig1.ex2} -- \Cref{fig3.ex2}.
	
	\begin{figure}[h]
	\centering
	\includegraphics[scale=0.54]{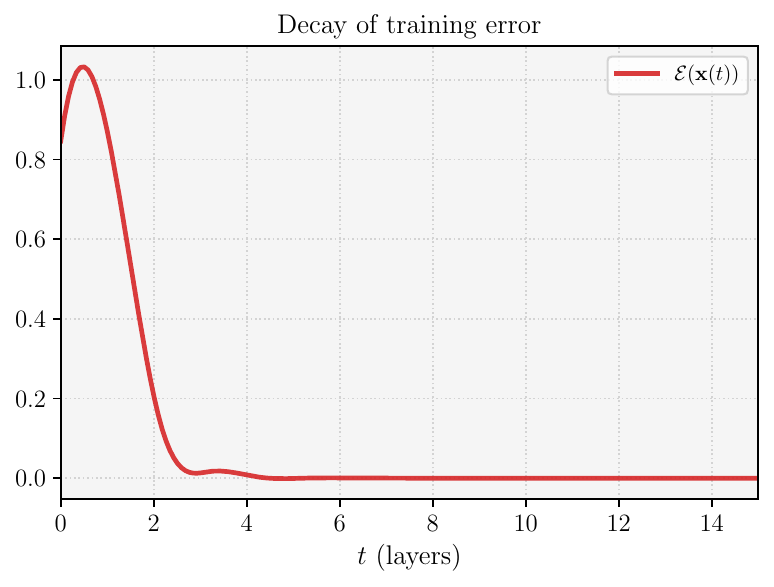}
\includegraphics[scale=0.54]{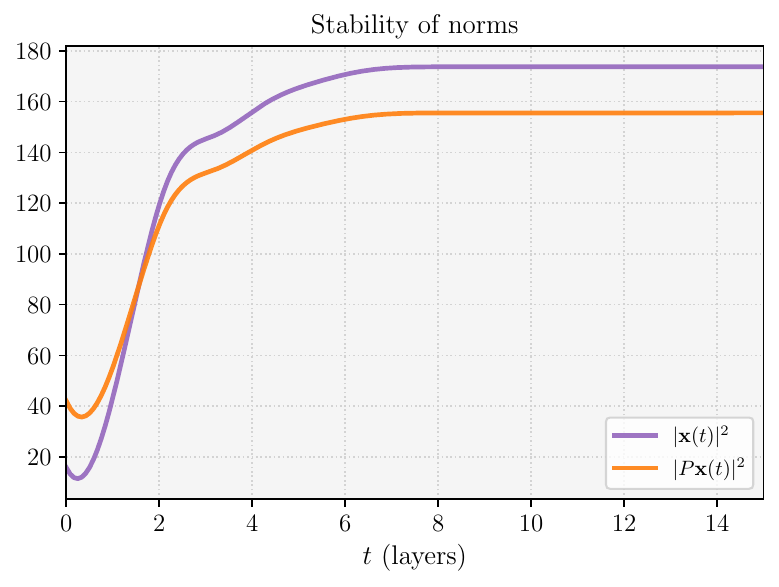}
\caption{\textbf{\Cref{ex: 2}:} The decay of the training error (\textit{left}) and stabilization of the trained trajectories $\*x(t)$ and $\{P\*x_i(t)\}_{i\in[N]}$ (\textit{right}).}
\label{fig1.ex2}
\end{figure}

\begin{figure}
\hspace{0.5cm} {\small $t=0$} \hspace{6cm} {\small $t\leqslant15$}

\includegraphics[scale=0.515]{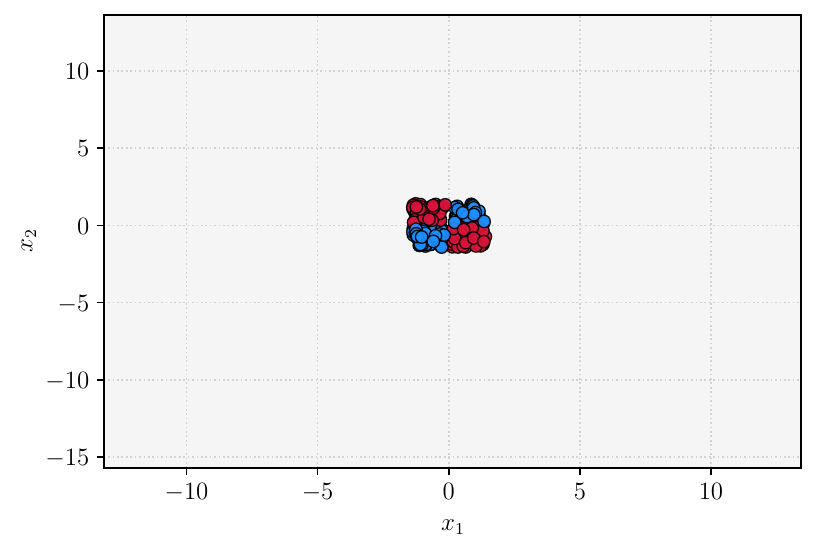}
\includegraphics[scale=0.515]{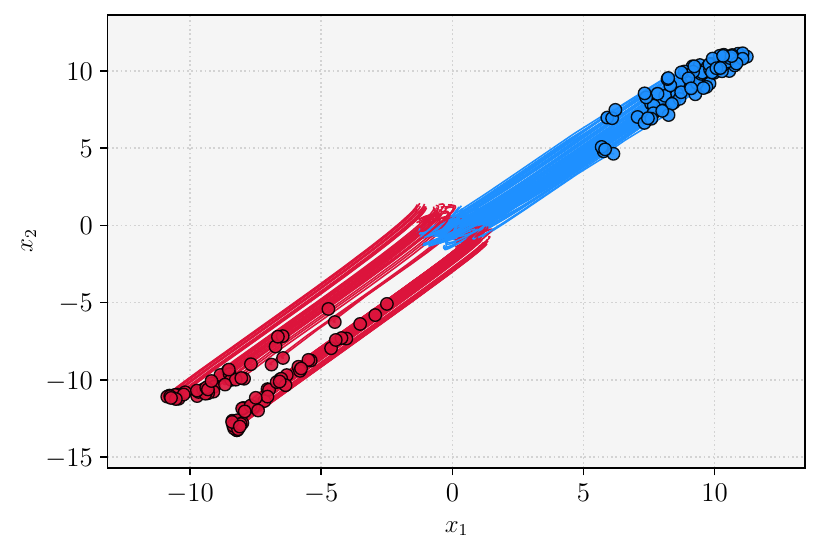}
\caption{\textbf{\Cref{ex: 2}:} A batch of training data (\textit{left}) and the evolution of the corresponding trained trajectories $\*x_{T,i}(t)$ (\textit{right}).
See \href{https://github.com/borjanG/dynamical.systems/blob/master/videos/example4-3.mp4}{\textcolor{dukeblue}{\texttt{https://github.com/borjanG/dynamical.systems/blob/master/\\
videos/example4-3.mp4}}} for a movie of the evolution of the trajectories, where the stability phenomenon depicted in \Cref{fig1.ex2} can be visualized.}
\label{fig2.ex2}
\end{figure}

\begin{SCfigure}[1][h]
\caption{\textbf{\Cref{ex: 2}:} The trained classifier on $[-2.5,2.5]^2$ and its evaluation on a batch of the test data. The shape of the dataset is captured accurately.}
\includegraphics[scale=0.54]{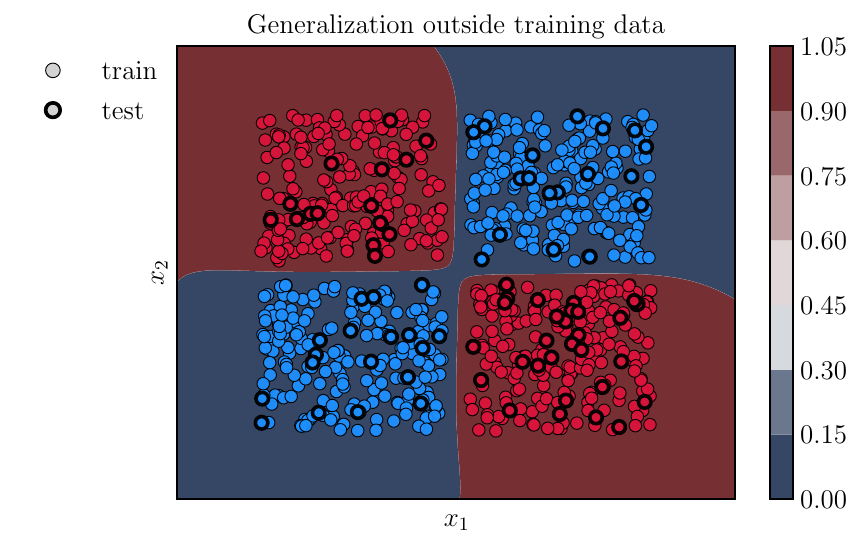}
\label{fig3.ex2}
\end{SCfigure}

\end{example}
	
	\begin{example}[MNIST] \label{ex: 4}
	
	We now show that the stabilization phenomenon may also be observed on more realistic datasets such as MNIST \citep{lecun2010mnist}. 
	MNIST is a dataset consisting of handwritten digits ranging from $0$ to $9$, with a training set consisting of 60000 samples, and a test set consisting of 10000 samples. 
	Each input sample $\vec{x}_i$ is a grayscale, $28\times 28$ image of a handwritten digit, and thus an element of $\R^{784}$; the dataset has $10$ labels: $\vec{y}_i \in [10]$. 
	We consider a similar setup as in \Cref{ex: 2} -- the model is parametrized as \eqref{eq: standard.dyn.sys} -- \eqref{eq: outside.inside} (we use a forward Euler scheme to obtain a corresponding ResNet with fixed time-step) where $d_\hid = 32$ and $\sigma\equiv\tanh$, we consider cross-entropy loss in the training error in \eqref{eq: time.dep.func.wanted} and only $L^2$--regularization of the parameters, with $T=20$. 
	We emphasize that we do not use any convolutional layers nor other commonly used operations in image classification (e.g. batch normalization, max-pooling) in the underlying ResNet architecture, and we solely concentrate on basic matrix weights.
	The output layer is parametrized by $Px = p_1x+p_2$, where $p_1\in\R^{10\times784}$, $p_2\in\R^{10}$ are part of the trainable parameters.
	We show the results of the experiments in \Cref{fig: figure.mnist.turnpike.1} -- \Cref{fig2.ex4}.
	
	\begin{figure}[h!] 
	\centering
	\includegraphics[scale=0.525]{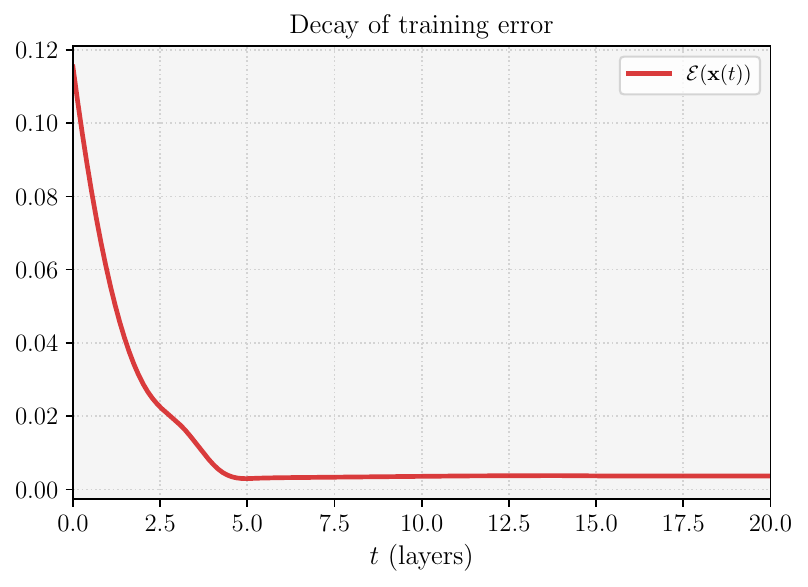}
	\includegraphics[scale=0.525]{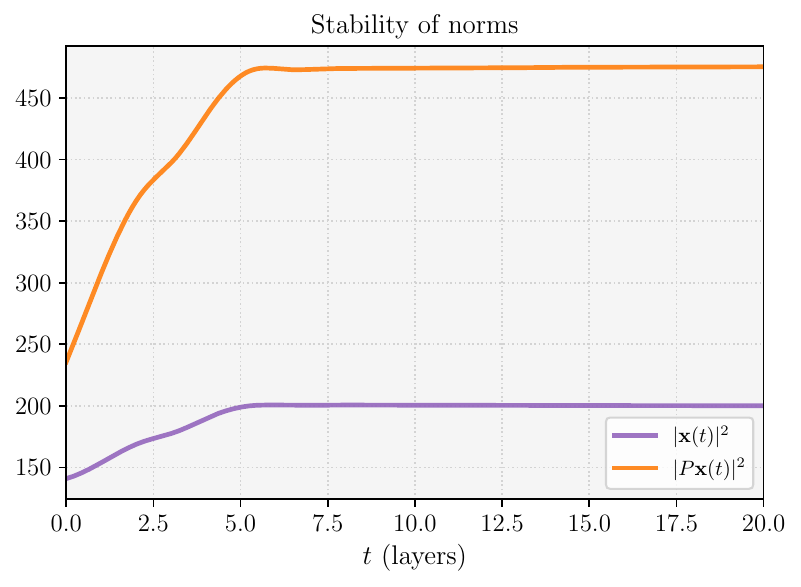}
	\caption{\textbf{\Cref{ex: 4}:} The decay of the training error (\textit{left}) and stabilization of the trained trajectories $\*x(t)$ and $\{P\*x_i(t)\}_{i\in[N]}$ (\textit{right}).}
	\label{fig: figure.mnist.turnpike.1}
	\end{figure}
	
	\begin{figure}
	\includegraphics[scale=1]{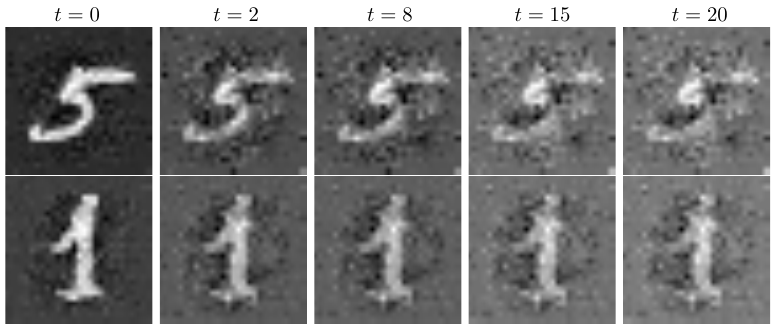}
%
%
	\caption{\textbf{\Cref{ex: 4}:} We depict the evolution of two individual samples $\*x_i(t)\in\R^{784}$ mapped onto a $28\times28$ grid.  We see that each trajectory stabilizes to some stationary configuration. The trained model tends to "diffuse" (in a caloric sense) the input signal ahead of classifying via the softmax applied to $P\*x_i(t) \in \R^{10}$.}
	\label{fig2.ex4}
	\end{figure}
	
	\begin{figure}
	\centering
	\includegraphics[scale=0.5]{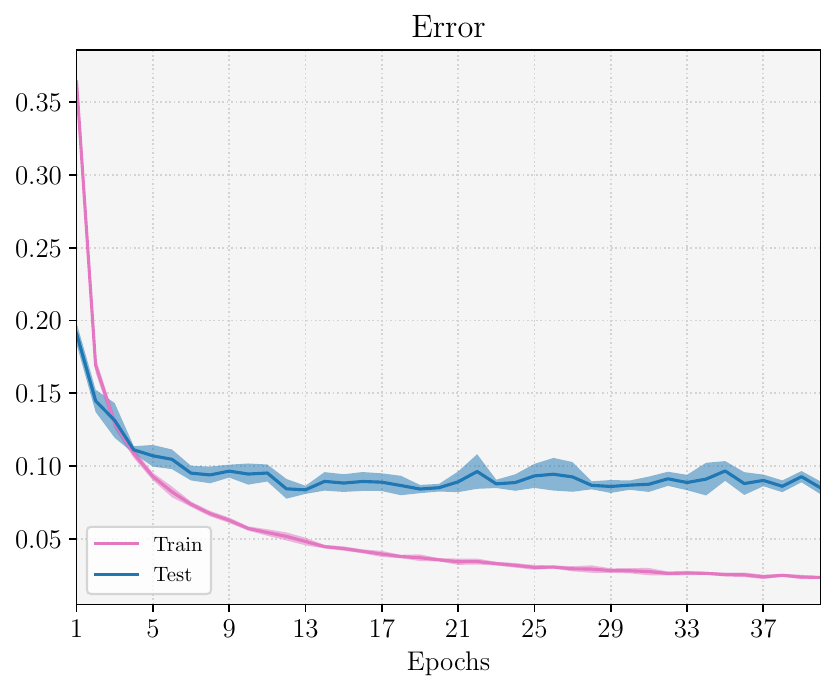}
	\includegraphics[scale=0.5]{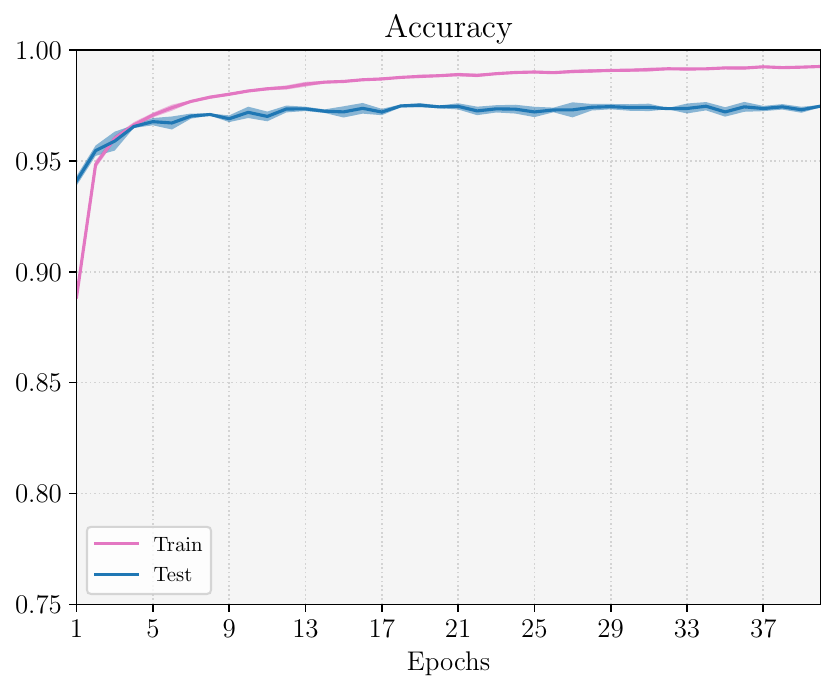}
	\caption{\textbf{\Cref{ex: 4}:} The validation error and accuracy over training epochs (experiments repeated $10$ times); in this simplified dataset setting, generalization is not necessarily compromised due to the introduction of an integrated empirical risk.}
	\end{figure}
	
	\end{example}
		
	\begin{example}[Fashion MNIST] \label{ex: 6}
	
	Fashion-MNIST is intended to serve as a direct drop-in replacement for the original MNIST dataset for benchmarking machine learning algorithms. It shares the same image size and structure of training and testing splits.
	We consider the same setup as in \Cref{ex: 4}, and we show the results of the experiments in \Cref{fig: figure.fashion.mnist.turnpike.1} -- \Cref{fig3.ex6}.
	
	\begin{figure}[h!]
	\centering
	\includegraphics[scale=0.525]{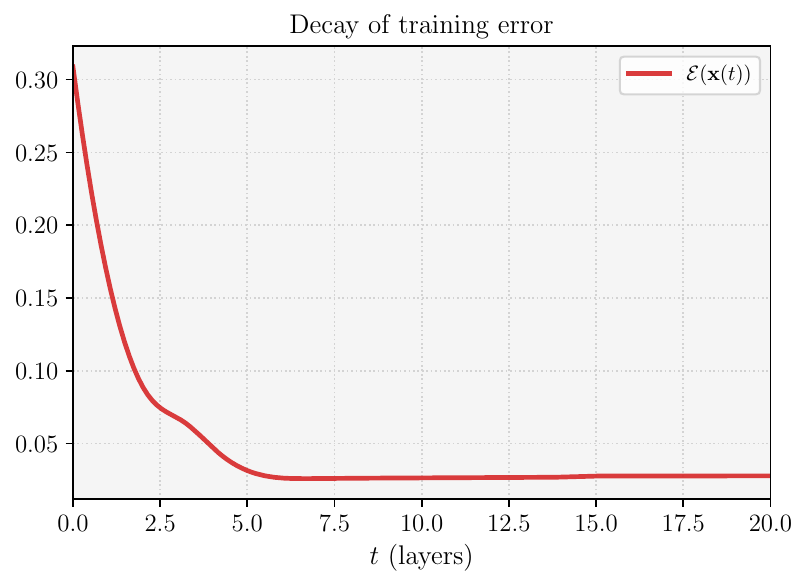}
	\includegraphics[scale=0.525]{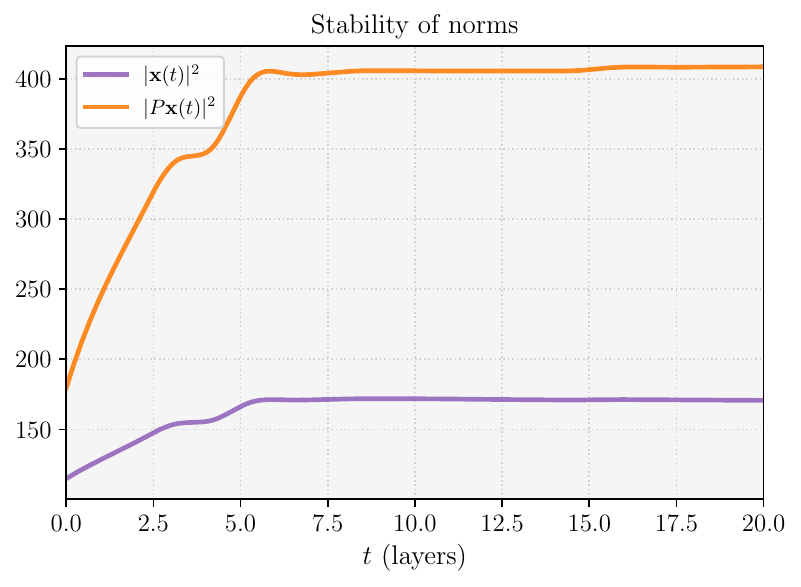}
	\caption{\textbf{\Cref{ex: 6}:} The decay of the training error (\textit{left}) and stabilization of the trained trajectories $\*x(t)$ and $\{P\*x_i(t)\}_{i\in[N]}$ (\textit{right}).}
	\label{fig: figure.fashion.mnist.turnpike.1}
	\end{figure}

	\begin{figure}
	\includegraphics[scale=1]{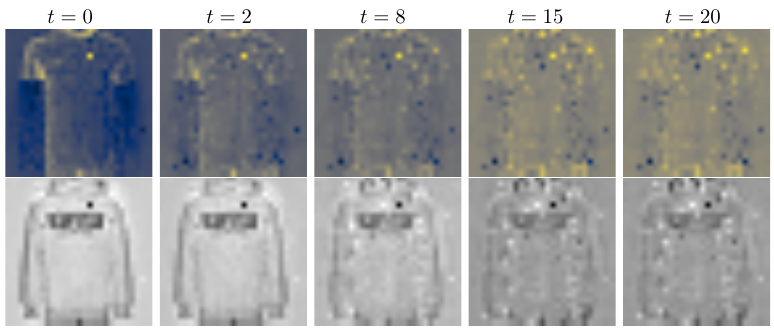}
%
%
	\caption{\textbf{\Cref{ex: 6}:} We depict the evolution of two individual samples $\*x_i(t)\in\R^{784}$ mapped onto a $28\times28$ grid (both sets of images are grayscale, but a different colormap is used to enhance visibility).}
	\label{fig2.ex6}
	\end{figure}
	
	\begin{figure}
	\centering
	\includegraphics[scale=0.5]{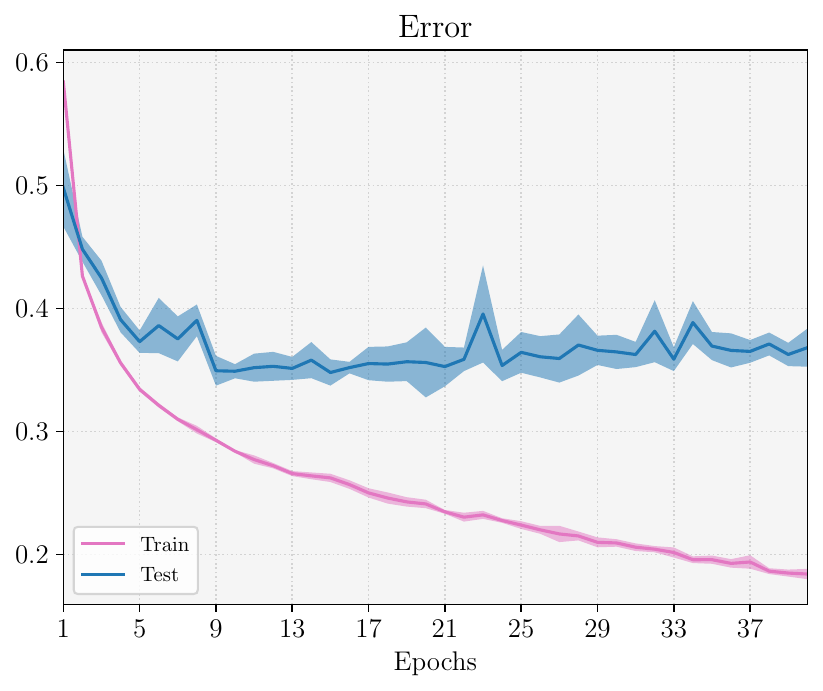}
	\includegraphics[scale=0.5]{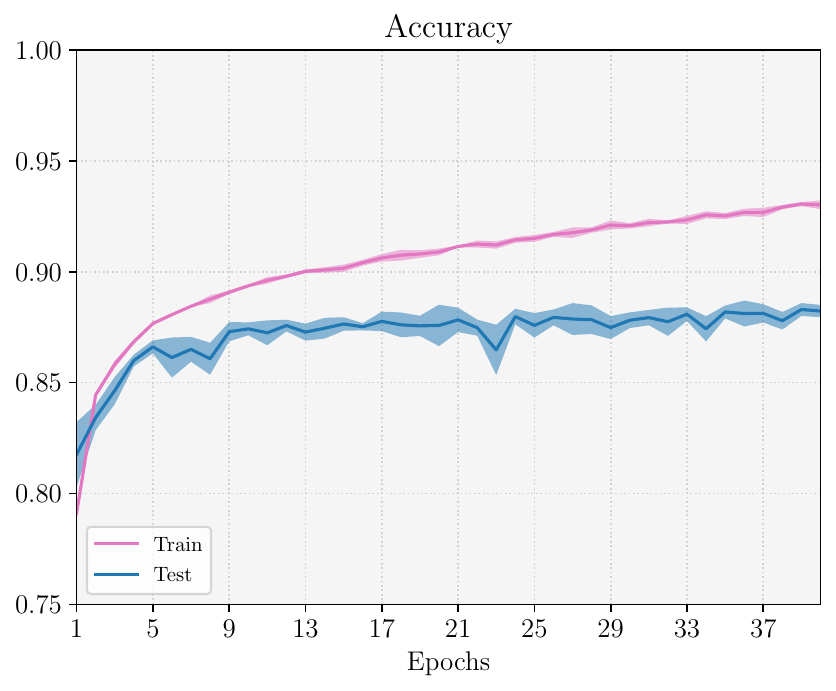}
	\caption{\textbf{\Cref{ex: 6}:} The validation error and accuracy over training epochs (experiments repeated $10$ times); as anticipated, generalization is not as good as for the simpler MNIST dataset. The lower accuracy with respect to state of the art configurations could also be due to the fact that we do not make use of convolutional layers.}
	\label{fig3.ex6}
	\end{figure}
	
	\end{example}

	\section{The interpolation regime}
	
	\noindent
	The majority of our results stated in the preceding sections stipulate whether and how the output $\*x(T)$ of the neural ODE trajectory approaches the so-called interpolation regime ($\mathscr{E}(\*x(T))=0$ with $\mathscr{E}$ given in \eqref{eq: phi.def}) when $T$ increases. 
	 It is thus of interest to also illuminate some of the properties of the parameters which allow the trajectory to reach a minimizer of the empirical risk $\mathscr{E}$, and to see whether such parameters indeed exist. 
	
	By means of an elementary Gr\"onwall argument, we can show the following illustrative result, which stipulates a lower bound for the amplitude of the weights $w$ in terms of the dispersion or concentration of the input data.
	
	\begin{proposition} \label{prop_lower_bound}
	Let $P:\R^d\to\R^m$ be surjective, and let $T>0$. 
    	Assume that for some parameters $[w,b]\in L^1(0,T;\R^{d_u})$ the solution $\*x\in C^0([0,T];\R^{d_x})$ to either \eqref{eq: sigma.inside.i} or \eqref{eq: sigma.outside.i} satisfies
    	\begin{equation} \label{eq: regression.interpolation}
    	P\*x_i(T) = \vec{y}_i \hspace{1cm} \text{ for all } i \in [N].
    	\end{equation}
    	Then
	    \begin{equation}\label{controllability_L1_lower_bound}
        \|w\|_{L^1(0, T; \R^{d_u})} \geqslant \mathfrak{C}(\sigma)\max_{\substack{(i,j)\in [N]^2\\  i\neq j }}\inf_{\substack{\*x_{i}^1\in P^{-1}\left(\left\{\vec{y}_i\right\}\right)\\  \*x_{j}^1\in P^{-1}\left(\{\vec{y}_j\}\right)}}\log\left( \frac{\left\|\*x_{i}^1-\*x_{j}^1\right\|}{\left\|\*x_{i}^0-\*x_{j}^0\right\|}\right),
    \end{equation}
    where $\mathfrak{C}(\sigma)>0$ is the Lipschitz constant of $\sigma\in\Lip(\R)$.
	\end{proposition}
	
	\noindent
	By virtue of Cauchy-Schwarz, \eqref{controllability_L1_lower_bound} clearly implies
	\begin{equation*}
    \|w\|_{L^2(0, T; \R^{d_u})} \geqslant  \frac{\mathfrak{C}(\sigma)}{\sqrt{T}}\max_{\substack{(i,j)\in[N]^2\\  i\neq j }}\inf_{\substack{\*x_{i}^1\in P^{-1}\left(\{\vec{y}_i\}\right)\\  \*x_{j}^1\in P^{-1}\left(\{\vec{y}_j\}\right)}}\log\left( \frac{\left\|\*x_{i}^1-\*x_{j}^1\right\|}{\left\|\*x_{i}^0-\*x_{j}^0\right\|}\right).
    \end{equation*}
    This observation leads us to stipulate the interest of $L^2$--regularization and increasing $T$: should the training data contains inputs which are very concentrated in the ambient space, parameters in the interpolation regime (\eqref{eq: regression.interpolation} is equivalent to $\mathscr{E}(\*x(T))=0$ whenever $\mathscr{E}$ attains its minimum) will dissipate in $L^2$--norm over longer time horizons.
    \smallskip
 
   \noindent On another hand, the difference quotient
	\begin{equation*}\label{overfitting_indicator}
        \kappa\left(\*x_i^0,\*x_i^1\right)\coloneqq \frac{\left\|\*x_{i}^1-\*x_{j}^1\right\|}{\left\|\*x_{i}^0-\*x_{j}^0\right\|}
    \end{equation*}
	can roughly be seen as an indicator of the variations of the flow, and thus the weight matrix would control the latter's simplicity. 
	Indeed, denoting by $\Phi_T:\R^d\to\R^d$ the map defined by $\Phi_T(\*x_i^0) = \*x_i(T)$, where $\*x_i(t)$ solves either \eqref{eq: sigma.inside.i} or \eqref{eq: sigma.outside.i}, one can roughly stipulate that $\kappa$ is an approximation $\nabla \Phi_T$ (evaluated at some intermediate point -- here $\nabla \Phi_T$ denotes the Jacobian matrix of the flow map $\Phi_T$). This can be argued already by using, for instance, the Cauchy mean-value theorem. 
	Let us provide further detail and focus on \eqref{eq: sigma.inside.i} for simplicity, and assume $\sigma\in C^1(\R)\cap\Lip(\R)$. We may linearize \eqref{eq: sigma.inside.i} with respect to the initial datum $\*x^0_i$ to obtain 
	\begin{equation*}
	\begin{dcases}
	\dot{\*z}_i(t) = w(t)\begin{bmatrix}\sigma'\left(\*x^0_{i,1}\right)\*z_{i,1}(t)\\\vdots\\\sigma'\left(\*x^0_{i,d}\right)\*z_{i,d}(t)\end{bmatrix} := \widehat{w}(t)\*z_i(t) &\text{ for } t\in(0,T)\\
	\*z_i(0) = \*z^0_i
	\end{dcases}
	\end{equation*}
	for $i\in[N]$. One then sees that
	\begin{equation*}
	\exp\left(\int_0^T \widehat{w}(t)\diff t\right) = \nabla\Phi_T(\*x^0_i)
	\end{equation*}
	for $i\in[N]$. One may thus see $\int_0^T \widehat{w}(t)\diff t$ as a "multi-dimensional" logarithm of $\nabla\Phi_T(\*x^0_i)$. Thus, the Jacobian of the flow, which is an indicator of its variations and thus of its simplicity, can be measured by the weight matrix.
	
   \subsection{On \Cref{def: ctrl}}

	\noindent
	To complete this section, we state the following interpolation result, which includes an estimate on the parameters with respect to the distance of the target and the initial datum, which somewhat enhances the validity of the assumption we make in \Cref{thm: no.running}.
	While such an estimate is standard in the setting of linear models, it is not provided by sufficient controllability conditions for nonlinear systems such as the Chow-Rashevski theorem \cite[Chapter 3, Section 3.3]{coron2007control}.

	 \begin{theorem}\label{prop_Nleqd+1}
	    	Let $T>0$ and assume that $N\leqslant d$. Let $\*x^1\in \R^{d_x}$ be given, and assume that the activation function $\sigma \in C^1(\R) \cap \Lip(\R)$ is such that
	    \begin{equation*} \label{svectors_system}
	        \Big\{\sigma\left(\*x_{1}^1\right),\dots,\sigma\left(\*x_{i}^1\right),\dots,\sigma\left(\*x_{N}^1\right)\Big\}
	    \end{equation*}
	    is a system of linearly independent vectors in $\R^{d}$. 
	    Then, there exist universal constants $r>0$ and $\mathfrak{C}>0$ such that for any datum $\*x^0\in \R^{d_x}$ satisfying  $\left\|\*x^0-\*x^1\right\|\leqslant r$, there exists a weight matrix $w\in L^{\infty}(0,T;\R^{d\times d})$ such that the unique solution $\*x(\cdot)$ to 
	\begin{equation*} \label{eq: sigma.inside.i_Nleqd+1}
    	\begin{dcases}
    	\dot{\*x}(t) = \*w(t) \sigma(\*x(t)) &\text{ in } (0,T) \\
    	\*x(0) = \*x^0,  
    	\end{dcases}
    	\end{equation*}
	satisfies
	\begin{equation*}
	\*x(T) = \*x^1,
	\end{equation*}
    	and the following estimate holds
    	\begin{equation*} \label{prop_Nleqd+1_estimate}
    	    \left\|w\right\|_{L^{\infty}(0,T;\R^{d\times d})}\leqslant \frac{\mathfrak{C}}{T}\left\|\*x^0-\*x^1\right\|.
    	\end{equation*}
	\end{theorem}

	\begin{remark}
	
        \begin{itemize} 
            \item For simplicity of presentation, we have not exhibited the bias parameter $b(t)$.  One can readily check that, in the presence of this additional parameter, the assumption $N\leqslant d$ can be relaxed to $N\leqslant d+1$.
            \smallskip
            	    
            \item The case $N>d+1$ may be treated by, for instance, appending additional features (e.g. zeros as in \citep{dupont2019augmented}) to the input data $\*x^0_i$ for $i\in[N]$ as to guarantee that the augmented datum is of dimension $d_\aug\geqslant N$.
        \end{itemize}
        \end{remark}
	
	\noindent
	In the discrete-time context of neural networks such as \eqref{eq: mlp} or \eqref{eq: resnet.discrete}, the property analog to \Cref{def: P-interpolation} is also well explored in the literature, and is commonly called \emph{finite sample expressivity} (\citep{zhang2016understanding}). 
	An additional interest is that of estimating the number of parameters -- referred to as \emph{the memorization capacity} -- needed to manifest this property. 
	  For further results in this direction we refer the reader to \citep{sra2019, bubeck2020network, bubeck2020law} and the references therein.

	 In the ODE context, the property of finite sample expressivity finds its analog in the \emph{ensemble} or \emph{simultaneous controllability}, wherein one requires only $1$ pair of parameters/controls to steer $N$ trajectories of the same system to $N$ prescribed targets -- this is the property we show in \Cref{prop_Nleqd+1}. 
	 The literature on such controllability results of neural ODEs, mostly relying on geometrical techniques such as Lie brackets techniques 
	 (see \cite[Chapter 3, Section 3.3]{coron2007control}), under specific constraints on the activations function, is vast (see e.g. \citep{cuchiero2019deep, agrachev2020control, tabuada2020}). 
	 We refer to \citep{ruiz2020universal} for further results in this direction.

	\section{Continuous space-time neural networks} \label{sec: resnet.variable}
	
	\noindent
	We now come back to the scheme \eqref{eq: resnet.discrete} defining a ResNet with $N_{\text{layers}}\geqslant2$ layers.
	 Whilst such networks are widely used in practice, in the discrete-time context, they do not take into account variations of the dimensions of the weights and states over layers. Such variations may arise when considering \emph{convolutional} and/or \emph{pooling} layers, which are ubiquitous in tasks in computer vision. 
	 In such tasks, it is moreover of interest to view the data itself as being continuum objects. 
	 \smallskip
	 
	To be more specific, we note that in the simplest nonlinear context, a residual network with variable dimensions analog to \eqref{eq: resnet.discrete} takes the form (see \citep{he2016deep})
	\begin{equation} \label{eq: variable.dim}
	\begin{dcases}
	\*x_i^{k+1} = \Pi^k\*x_i^k + \sigma\left(w^{k} \*x_i^k + b^k\right) &\text{ for } k \in \{0, \ldots, N_{\text{layers}}-1\} \\
	\*x_i^0 = \vec{x}_i.
	\end{dcases}
	\end{equation}
	Here, contrary to \eqref{eq: resnet.discrete}, we have $w^k \in \R^{d_{k+1}\times d_k}$ and $b^k \in \R^{d_{k+1}}$, and thus $\*x^k \in \R^{d_{k}}$ for $k \in \{0, \ldots, N_{\text{layers}}\}$, where $\{d_k\}_{k=0}^{N_{\text{layers}}}$ are given positive integers, called widths of the layers $k$. One imposes $d_0 = d$, and $\Pi^k \in \R^{d_{k+1}\times d_k}$ is a projection/embedding operator which serves to match dimensions.  Much like in the fixed width case, we may also write the variable-width ResNet when $\mathfrak{f}$ is parametrized as in \eqref{eq: sigma.inside} or otherwise.
	\smallskip
	
	\noindent
	\textbf{The continuous space-time network.}
	It is not immediately obvious how one can see \eqref{eq: variable.dim} as a numerical scheme for some continuous-time dynamical system in the flavor of \eqref{eq: standard.dyn.sys}.
	Nevertheless, this can be achieved by viewing the changing dimension over time-steps as an additional (spatial) variable, thus yielding an integro-differential equation in the continuum.
	\smallskip
	 
	 \noindent
	 To be more precise, for any $i \in [N]$ we consider the scalar integro-differential equation
	\begin{equation} \label{eq: nonlocal.sigma.outside}
	\begin{dcases}
	\del_t \*x_i(t, x) = \sigma\left(\int_\Omega w(t, x, \xi) \*x_i(t, \xi) \diff \xi + b(t,x) \right) &\text{ for } (t,x)\in (0, T) \times \Omega \\
	\*x_i(0, x) = \*x_i^\ini(x) &\text{ for }  x\in \Omega.
	\end{dcases}
	\end{equation}
	Here $\Omega \subset \R^{d_\Omega}$ is a bounded domain, where $d_\Omega\geqslant 1$. 
	We emphasize that $\*x_i(t, x) \in \R$ for $(t, x) \in (0, T)\times \Omega$, and similarly, $w(t, x, \xi) \in \R$ and $b(t,x) \in \R$ for $(x, \xi) \in \Omega\times\Omega$.
	The initial datum $\*x_i^\ini \in C^0(\overline{\Omega})$ is such that there exist $\{x_j\}_{j=1}^d \subset \Omega$ such that $\*x_i^\ini(x_j) = (\vec{x}_i)_{j}$. 
	Such a datum can always be found (e.g. by interpolation). 
	The continuum model \eqref{eq: nonlocal.sigma.outside} is proposed in \citep{liu2019selection} where well-posedness is established, and is also suggested in \citep{weinan2017proposal} albeit in a slightly different context. 
	We distinguish two typical cases for choosing the shape of $\Omega$ as well as $d_\Omega$.
	 
	\begin{itemize}
	\item\textbf{Variable-width ResNets.}
	If in the discretized level, we seek to simply obtain a variable-width residual network such as \eqref{eq: variable.dim} (or even the standard ResNet analog \eqref{eq: resnet.discrete}), it suffices to consider $\Omega = (0, 1)$, thus $d_\Omega=1$. 
	We give more detail on possible possible discretizations in \Cref{sec: numerical.analysis} and \Cref{rem: particle.method}.
	\smallskip

	\item \textbf{Convolutional Neural Networks.} 
	The situation is slightly more delicate in the case of CNNs\footnote{The mathematical theory of structural properties of CNNs in feed-forward form (without skip-connections) is well-established -- see for instance \citep{mallat2016understanding} and the references therein.}, which are typically used in tasks arising in computer vision. We provide a proposal covering the continuous-time analog of CNNs with partial generality.
	 
	Assume that the dataset $\{\vec{x}_i\}_{i\in[N]}$ consists of $N$ images: $\vec{x}_i \in \R^{d_1\times d_2\times d_\text{ch}}$ for any $i$; here $d_1$ (resp. $d_2$) denote the number of horizontal (resp. vertical) pixels in the image $\vec{x}_i$, whereas $d_\text{ch}$ denotes the number of channels, i.e. the color format (e.g. $d_\text{ch} = 1$ for grayscale, $d_\text{ch} = 3$ for RGB).  
	In this case, we consider $\Omega := \Omega_\img \times (0, 1)$, where $\Omega_\img \subset \R^2$ is a rectangle. Thus $d_\Omega = 3$. Moreover, we assume that the weights $w$ in \eqref{eq: nonlocal.sigma.outside} are compactly supported and of a specific \emph{convolutional} form (as indicated in most works, this is more so a \emph{cross-correlation} form), namely, for any $i$, the equation takes the form
	\begin{equation*}
	\del_t \*x_i(t,x,\zeta) = \sigma\left(\int_0^1 \int_{\Omega_\img} w(t,x+\xi,\omega,\zeta)\*x_i(t,\xi,\omega) \diff \xi \diff \omega + b(t,x,\zeta) \right) 
	\end{equation*}
	for $(t,x,\zeta)\in(0,T)\times\Omega_\img\times (0,1)$. We note that the variable $x\in \Omega_\img$ denotes a pixel, whereas $\zeta \in (0,1)$ is a continuous variable indicating, when discretized, the number of extracted features (namely the number of filters). 
	The bias parameter $b$ can be omitted in this case, if desired.
	
	One possible way to discretize the above continuous-time model and obtain a CNN-ResNet as in \citep{he2016deep} is to follow the arguments in \Cref{sec: numerical.analysis}, where one would use a time-dependent grid for discretizing with respect to the variable $\zeta \in (0,1)$ as well, as the number of filters commonly varies over layers in CNNs.
	By discretizing $\Omega_\img$ with a "shrinking" or "expanding" time-dependent rectangular grid, some effects of padding or pooling (but not max-pooling a priori) may also be considered. However, a full CNN--applicable theory is beyond the scope of this work.   
	\end{itemize}
		
	\begin{remark} \label{rem: particle.method}
	Observe that the continuous space-time model \eqref{eq: nonlocal.sigma.outside} (resp. \eqref{eq: nonlocal.sigma.inside}) is more general and englobes \eqref{eq: standard.dyn.sys} -- \eqref{eq: sigma.outside} (resp. \eqref{eq: standard.dyn.sys} -- \eqref{eq: sigma.inside}), where only the time variable is considered to be continuous.
	Indeed, fix $d$ different points $\left\{ x_1,\ldots, x_d\right\}\in \Omega$, 
	and let $\delta_{x_j}$ denote the Dirac mass centered at $x_j$.
	For any $i\in [N]$, we consider the initial datum
	\begin{equation*}
	\*x_i^\ini(x) := \sum_{j=1}^{d} (\vec{x}_i)_j\, \delta_{x_j}(x) \hspace{1cm} \text{ for } x \in \Omega.
	\end{equation*}
	We write the weight $w$ as 
	\begin{equation*}
	w(t,x,\zeta) := \sum_{j=1}^d\sum_{\ell=1}^d w_{j,\ell}(t)\delta_{x_j}(x)\delta_{x_\ell}(\zeta) \hspace{1cm} \text{ for } (t,x,\zeta) \in (0,T)\times \Omega\times\Omega,
	\end{equation*}
	yielding the matrix $[w_{j,\ell}(t)]_{1\leqslant j, \ell \leqslant d}$ of weights at time $t$, whereas the bias $b(t,x)$ is written as
	\begin{equation*}
	b(t,x) := \sum_{j=1}^d b_{j}(t)\delta_{x_j}(x) \hspace{1cm} \text{ for } (t,x) \in (0,T)\times \Omega,
	\end{equation*}
	yielding the vector $[b_j(t)]_{1\leqslant j\leqslant d}$ of biases at time $t$.
	As $\*x_i^\ini$, $w$ and $b$ are all linear combinations of Dirac masses, by plugging them in \eqref{eq: nonlocal.sigma.outside}, we rewrite the integrals as sums, and setting, for any $i \in [N]$,
	\begin{equation*}
	(\*x_i)_j(t) := \displaystyle\int_\Omega \*x_i(t,x) \diff\delta_{x_j}(x)
	\end{equation*}
	for $j \in [d]$, we see that $(\*x_i)_j$ solves
	\begin{equation*}
	\begin{dcases}
	(\dot{\*x}_i)_j(t) = \sigma \left( \sum_{\ell=1}^d w_{j,\ell}(t)\, (\*x_i)_\ell(t) + b_j(t)\right) &\text{ for } t\in (0,T) \\
	(\*x_i)_j(0) = (\vec{x}_i)_j.
	\end{dcases}
	\end{equation*}
	This is just the $j$--th equation of the \eqref{eq: standard.dyn.sys} -- \eqref{eq: sigma.outside} for $i \in [N]$.
	\end{remark}
	
	\noindent		
	Correspondingly for $i \in [N]$ we may consider
	\begin{equation} \label{eq: nonlocal.sigma.inside}
	\begin{dcases}
	\del_t \*x_i(t, x) = \int_0^1 w(t,x,\xi) \sigma(\*x_i(t,\xi)) \diff \xi + b(t,x) &\text{ in } (0, T)\times \Omega \\
	\*x_i(0, x) = \*x_i^\ini(x) &\text{ in } \Omega.
	\end{dcases}
	\end{equation}
	All of the above discussions also apply for this system.
	
	\subsection{From continuous to discrete} \label{sec: numerical.analysis}
	
	The passage from \eqref{eq: nonlocal.sigma.outside} to a discrete-time scheme such as \eqref{eq: variable.dim} is not immediately obvious, and to our knowledge has not been presented in the literature.
	To proceed, it is important to observe the inherent link between the layer $k$ and the width $d_k$ in \eqref{eq: variable.dim}. 
	This motivates discretizing \eqref{eq: nonlocal.sigma.outside} in the spatial variable $x \in (0, 1)$ by using a \emph{time-dependent grid}, which has a different number of nodes $d_k$ at each time-step. 
	We give more detail on this in what follows.
	\smallskip 
	
	Let us demonstrate that \eqref{eq: nonlocal.sigma.outside} which reads\footnote{The choice of the spatial interval $[0, 1]$ is completely arbitrary -- one may of course consider any bounded interval of $\R$.} (we omit the dependence on $i$ for notational simplicity)
	\begin{equation*}
	\begin{dcases}
	\del_t \*x(t, x) = \sigma\left(\int_0^1 w(t, x, \xi) \*x(t, \xi) \diff \xi + b(t, x) \right) &\text{ in } (0, T) \times (0, 1) \\
	\*x(0, x) = \*x^\ini(x) &\text{ in } (0, 1),
	\end{dcases}
	\end{equation*}
	where $\*x^\ini\in C^0([0,1])$ is such that $\*x^\ini(x_j) = \vec{x}_{, j}$ for some $\{x_j\}_{j=1}^d \subset [0, 1]$, can be discretized to read exactly as
	\begin{equation} \label{eq: discrete.resnet.overload}
	\begin{dcases}
	\*x^{k+1} = \Pi^k \*x^k + \sigma\left(w^k \*x^k + b^k \right) &\text{ for } k \in \{0, \ldots, N_{\text{layers}}-1\} \\
	\*x^0 = \vec{x}.
	\end{dcases}
	\end{equation}
	Here $\*x^k \in \R^{d_k}$, $w^k \in \R^{d_{k+1}\times d_k}$ and $b^k \in \R^{d_{k+1}}$, with $d_0:=d$ and $\{d_k\}_{k=1}^{N_{\text{layers}}}$ given positive integers, and $\Pi^k \in \R^{d_{k+1}\times d_k}$.
	
	The derivation below is purely for illustrative purposes -- an adaptive solver ought to perform better than an adaptation of an Euler scheme as \eqref{eq: discrete.resnet.overload}. Moreover, the subsequent arguments will of course also apply to \eqref{eq: nonlocal.sigma.inside}.
	
	Let 
	\begin{equation*}
	\left\{t^0, \ldots, t^{N_{\text{layers}}}\right\}, \hspace{1cm} \text{ with } t^0:= 0 \text{ and } t^{N_{\text{layers}}} :=T,
	\end{equation*}
	be a given, non-decreasing sequence of time-steps. 
	For simplicity of presentation, let us assume that the time-steps are uniform, namely $t^{k} =k\Delta t$ with $\Delta t = \frac{T}{N_{\text{layers}}}$, but more general time-adaptive sequences can be considered.
	For any $k \in \{0, \ldots, N_{\text{layers}}\}$, let us assume that we are given a grid
	\begin{equation*}
	\left\{x_j\left(t^k\right)\right\}_{j=1}^{d_k} \subset [0, 1]
	\end{equation*}
	which is ordered and uniformly distributed. 
	For simplicity of presentation, in our discussion we will assume that $x_1(t^k) = 0$ and $x_{d_k}(t^k) = 1$ for any $k$.
	However by means of a time-step-dependent dilation, this restriction may be removed.
	We note that, not only there might be no overlap of grid nodes over different time-steps, but moreover, the number of grid nodes changes at each time-step $k$.
	
	We will seek for an appropriate discretization of
	\begin{equation} \label{eq: on.nodes}
	\del_t \*x(t^{k+1}, x_j(t^{k+1})) = \sigma\left( \int_0^1 w\left(t^{k+1}, x_j(t^{k+1}), \xi\right) \*x(t^k, \xi) \diff \xi + b\left(t^{k+1}, x_j(t^{k+1})\right) \right)
	\end{equation}
	for $k \in \{0, \ldots, N_{\text{layers}}-1\}$ and $j \in \{1, \ldots, d_{k+1}\}$. 
	Hence, in view of the preceding discussion, some kind of interpolation may be needed to justify a backward Euler discretization of the time derivative $\del_t\*x$ appearing in \eqref{eq: on.nodes} at the grid nodes. 
	
	For any given $k \in \{0, \ldots, N_{\text{layers}}-1\}$ and $j \in \{1, \ldots, d_k\}$, we shall henceforth denote
	\begin{equation*}
	x_j^k := x_j(t^k), \hspace{1cm} \*x_j^k := \*x(t^k, x_j^k).
	\end{equation*}
	Following through the above discussion, the main issue in writing down a forward difference discretization to $\del_t \*x(t^{k+1}, x_j(t^{k+1}))$ appears whenever for a given $k$ one has $d_k\neq d_{k+1}$, as it is a priori not possible to make sense of the expression $\*x(t^{k+1}, x_j(t^{k+1}))-\*x(t^k, x_j(t^k))$ for $j\neq1$. 
	Indeed, all $\iota \in \{2, \ldots, d_{k}\}$ are such that $x_\iota(t^{k}) \notin \{x_j(t^{k+1})\}_{j=1}^{d_{k+1}}$, due to the uniformity of the grid.
	
	Let us give an elementary argument for addressing this issue. 
	Given $k$ and given any $j \in \{1, \ldots, d_{k+1}\}$, there clearly exists $\iota \in \{2,\ldots, d_k \}$ such that $x_j^{k+1} \in [x_{\iota-1}^k, x_\iota^k]$.
	For such indices, we may thus define the linear interpolant
	\begin{equation} \label{eq: interpolant}
	\widehat{\*x}_j^k := \*x_\iota^k + \frac{\*x_\iota^k -\*x_{\iota-1}^k}{x_\iota^k - x_{\iota-1}^k}\left(x_j^{k+1}-x_\iota^k\right).
	\end{equation}
	This is nothing but an approximation of the first order Taylor expansion of $\*x(t^{k+1}, x_j(t^{k+1}))$ with respect to the second variable. 
	Using this interpolant, we may consider the simple forward difference 
	\begin{equation} \label{eq: difference}
	\del_t \*x(t^{k+1}, x_j(t^{k+1})) \approx \frac{\*x_j^{k+1} - \widehat{\*x}_j^k}{\Delta t}
	\end{equation}
	for any $k \in \{0, \ldots, N_{\text{layers}}-1\}$ and any $j \in \{1, \ldots, d_{k+1}\}$.
	We may now use any Newton-Cotes formula to discretize the integral term in \eqref{eq: on.nodes}: for $j \in \{1, \ldots, d_{k+1}\}$, we write
	\begin{equation} \label{eq: newton.cotes}
	\int_0^1 w\left(t^{k+1}, x_j(t^{k+1}), \xi\right) \*x(t^k, \xi) \diff \xi \approx \sum_{\iota=1}^{d_k} \alpha_\iota w\left(t^{k+1}, x_j(t^{k+1}), x_\iota(t^k)\right) \*x(t^k, x_\iota(t^k)).
	\end{equation}
	Here, $\alpha_\iota>0$ are the corresponding weights of the chosen Newton-Cotes formula.
	 
	Let us now define 
	\begin{equation*}
	\*x^k := \begin{bmatrix} \*x(t^k, x_1(t^k)) \\ \vdots \\ \*x(t^k, x_{d_k}(t^k))\end{bmatrix} \in \R^{d_k}, \hspace{1cm} b^k := \begin{bmatrix} b(t^{k+1}, x_1(t^{k+1})) \\ \vdots \\ b(t^{k+1}, x_{d_{k+1}}(t^{k+1})) \end{bmatrix} \in \R^{d_{k+1}}	
	\end{equation*}
	and
	\begin{equation*}
	w^k := \left[ \alpha_\iota w(t^{k+1}, x_j(t^{k+1}), x_\iota({t^k})) \right]_{1 \leqslant j \leqslant d_{k+1} \\ 1 \leqslant \iota \leqslant d_k} \in \R^{d_{k+1}\times d_k}.
	\end{equation*}
	The above definitions, as well as \eqref{eq: difference} and \eqref{eq: newton.cotes} applied to \eqref{eq: on.nodes}, lead us to \eqref{eq: discrete.resnet.overload}, where $\Delta t$ has been "omitted" as a factor of the nonlinearity. In view of \eqref{eq: interpolant}, the operator $\Pi^k\in\R^{d_{k+1}} \times \R^{d_k}$ takes the explicit form
	\begin{equation*}
	\Pi^k = \sum_{j=1}^{d_{k+1}} \left( \left\{ 1+\frac{x_j^{k+1}-x_{\iota(j)}^k}{x_{\iota(j)}^k-x_{\iota(j)-1}^k} \right\}\overline{e}_j e_{\iota(j)}^\tr - \frac{x_j^{k+1}-x_{\iota(j)}^k}{x_{\iota(j)}^k-x_{\iota(j)-1}^k} \overline{e}_j e_{\iota(j)-1}^\tr \right),
	\end{equation*}
	where $\iota(j) \in \{2, \ldots, d_k\}$ is such that $x_j^{k+1} \in [x_{\iota(j)-1}^k, x_{\iota(j)}^k]$, while $\{\overline{e}_j\}_{j=1}^{d_{k+1}}$ and $\{ e_j\}_{j=1}^{d_k}$ denote the canonical bases of $\R^{d_{k+1}}$ and $\R^{d_k}$ respectively. 
	We notice that the matrix $\Pi^k$ only has $2$ non-zero elements at every row $j \in \{1, \ldots, d_{k+1}\}$. This concludes our derivation.
	
	\begin{remark}[Generating moving grids]
	Whilst we have assumed a very simple given time-dependent grid, one may certainly generate more sophisticated moving grids  -- we refer to \citep[Section 3]{budd2009adaptivity} for a comprehensive overview on the existing methods, which have found extensive use in the discretization of partial differential equations manifesting shock waves and/or free boundaries.
	\end{remark}
		
	\subsection{The supervised learning problem} 
	Given a training dataset $\{\vec{x}_i, \vec{y}_i\}_{i\in[N]}$ with $\vec{x}_i \in \R^d$ for any $i\in[N]$, just as in the finite dimensional context, we begin by writing the equation satisfied by the stacked vector of states $\*x := [\*x_1, \ldots, \*x_N]^\top$ corresponding to the stacked vector of data $\*x^\ini := [\*x_1^\ini, \ldots, \*x^\ini_N]^\top$, 
where each $\*x_i$ is the solution to either \eqref{eq: nonlocal.sigma.outside} or \eqref{eq: nonlocal.sigma.inside} corresponding to the datum $\*x^\ini_i$, and control parameters $[w,b]$ which are the same for all $i$. 
	The stacked continuous space-time neural networks we consider are thus either
	\begin{equation} \label{eq: stacked.nonlocal.sigma.outside}
	\begin{dcases}
	\del_t \*x(t,x) = \sigma\left(\int_\Omega \*w(t,x,\xi)\*x(t,\xi)\diff\xi+\*b(t,x) \right) &\text{ in } (0,T)\times \Omega\\
	\*x(0, x) = \*x^\ini(x) &\text{ in } \Omega
	\end{dcases}
	\end{equation}
	or
	\begin{equation} \label{eq: stacked.nonlocal.sigma.inside}
	\begin{dcases}
	\del_t \*x(t,x) = \int_\Omega \*w(t,x,\xi)\sigma(\*x(t,\xi))\diff\xi+\*b(t,x) &\text{ in } (0,T)\times \Omega\\
	\*x(0, x) = \*x^\ini(x) &\text{ in } \Omega.
	\end{dcases}
	\end{equation}
	Just as in the finite-dimensional case, it is important to note how $[w(t,x,\xi),b(t,x)]$ for $(t,x,\xi) \in (0,T)\times \Omega\times \Omega$ enter the systems:
	\begin{equation*} \label{eq: controls.nonlocal}
	\*w(t,x,\xi):=	\begin{bmatrix}
    w(t,x,\xi) & & \\
    &\ddots& \\
    & & w(t,x,\xi)  \end{bmatrix} \in \R^{N\times N}, \hspace{0.5cm} \*b(t,x) := \begin{bmatrix}b(t,x) \\\vdots \\b(t,x) \end{bmatrix}
 \in \R^{N}.
	\end{equation*}
	
	\subsubsection{Empirical risk minimization} 
	As before, we first consider the regularized empirical risk minimization problem
	\begin{equation} \label{eq: nonlocal.learning}
	\inf_{\substack{[w,b] \in H^k(0, T; \mathfrak{U} )\\ \text{ subject to } \eqref{eq: stacked.nonlocal.sigma.outside} \text{ (resp.} \eqref{eq: stacked.nonlocal.sigma.inside})}} \mathscr{E}(\*x(T)) + \lambda \Big\| [w,b] \Big\|^2_{H^k(0,T;\mathfrak{U} )},
	\end{equation}
	where $\alpha>0$ is fixed, $k=0$ for \eqref{eq: stacked.nonlocal.sigma.inside} and $k=1$ for \eqref{eq: stacked.nonlocal.sigma.outside}, 
	$
	\mathfrak{U}  := L^2(\Omega\times \Omega)\times L^2(\Omega),
	$
	and we define the training error as 
	\begin{equation} \label{eq: training.error.nonlocal}
	\mathscr{E}(\*x(T)) := \frac{1}{N} \sum_{i=1}^N \loss(P\*x_i(T),\vec{y}_i),
	\end{equation}
	where $\loss\in C^0(\R^m\times\mathcal{Y};\R_+)$ and $P:L^2(\Omega)\to\R^m$ is given.
	The optimization problem \eqref{eq: nonlocal.learning} admits a solution by the direct method in the calculus of variations.
	\smallskip
	
	\noindent
	In view of the rather universal nature of the proof to \Cref{thm: no.running} and \Cref{thm: turnpike.P} in the finite-dimensional case, one may in fact roughly repeat the exact same proofs at most points, replacing throughout the finite dimensional euclidean spaces $\R^{d_x}$ and $\R^{d_u}$, by $L^2(\Omega)^N$ and $\mathfrak{U} $ respectively. 
	Whence, we state the infinite-dimensional (partial) analog to \Cref{thm: no.running}.	
	
	\begin{theorem} \label{thm: non.local.scaling}
		Let $\lambda>0$ be fixed, and let $\*x^\ini \in \left(C^0(\overline{\Omega})\right)^N$ be such that $\*x^\ini_i(x_j) = (\vec{x}_{i})_j$. Suppose that $P: L^2(\Omega)\to\R^m$ is any non-zero affine map, and suppose that $\loss\in C^0(\R^{m}\times\mathcal{Y}; \R_+)$ is such that \Cref{ass: regression.ass} is satisfied.
		Assume that \eqref{eq: stacked.nonlocal.sigma.outside} (resp. \eqref{eq: stacked.nonlocal.sigma.inside} with $\sigma$ positively homogeneous of degree $1$) interpolates the dataset $\{\*x^\ini_i, \vec{y}_i\}_{i\in[N]}$ in time $1$ in the sense of \Cref{def: P-interpolation}. 
	    For any $T\geqslant1$, let $\*x_T \in C^0\big([0,T]; L^2(\Omega)^N\big)$ be the unique solution to \eqref{eq: stacked.nonlocal.sigma.outside} (resp. \eqref{eq: stacked.nonlocal.sigma.inside}), associated to any global minimizer $[w_T,b_T] \in H^k(0,T;\mathfrak{U})$ of the functional in \eqref{eq: nonlocal.learning}, where $k=0$ in the case of \eqref{eq: stacked.nonlocal.sigma.inside} and $k=1$ in the case of \eqref{eq: stacked.nonlocal.sigma.outside}.
	    The following properties then hold.
	    \begin{enumerate}
	    \item There exists a constant $\mathfrak{C} = \mathfrak{C}\left(\{\vec{x}_i, \vec{y}_i\}_{i\in[N]}, \lambda\right)>0$ independent of $T$ such that
	    \begin{equation*}
	    \mathscr{E}(\*x_T(T)) \leqslant \frac{\mathfrak{C}}{T}. 
	    \end{equation*}
	    \item There exists a sequence $\{T_n\}_{n=1}^{\infty}$, with $T_n>0$ and $T_n \xrightarrow[n\longrightarrow \infty]{}\infty$, and some $\*x_\circ \in L^2(\Omega)^N$ with $\mathscr{E}(\*x_\circ)=0$ such that, along a subsequence,
	    \begin{equation*}
	    \mathscr{E}\left(\*x_{T_n}(T_n)\right) \xrightarrow[n\longrightarrow \infty]{} 0 
	    \end{equation*}
	    and
	    \begin{equation*}
	    \*x_{T_n}(T_n) \xrightharpoonup[n\longrightarrow \infty]{} \*x_\circ \hspace{1cm} \text{ weakly in } L^2(\Omega)^N.
	    \end{equation*}
	    \end{enumerate}
	\end{theorem}
	
	\noindent
	For the sake of completeness, we give a sketch of the proof -- by indicating the only changes with respect to that of \Cref{thm: no.running}.
	
	\begin{proof}[Proof of \Cref{thm: non.local.scaling}]
	    We note that the infinite-dimensional analog of \Cref{lem: scaling} may easily be shown to hold, and one may readily repeat precisely the same arguments as in the proof of \Cref{thm: no.running}, replacing $\R^{d_u}$ and $\R^{d_x}$ by $\mathfrak{U} $ and $L^2(\Omega)^N$ respectively throughout.
	    The only difference occurs in regarding the arguments on strong $L^2$--convergence of the sequence of controls in the case $k=1$ -- in the infinite dimensional case, we may exhibit the Aubin-Lions compactness lemma instead of Rellich-Kondrachov to conclude.
	\end{proof} 
	
	\subsubsection{Augmented empirical risk minimization}
	
	We similarly consider the augmented supervised learning problem
	\begin{equation} \label{eq: nonlocal.learning.tracking}
	\inf_{\substack{[w,b] \in L^2(0, T; \mathfrak{U} )\\ \text{ subject to } \eqref{eq: stacked.nonlocal.sigma.inside}}} \mathscr{E}(\*x(T))+ \frac{1}{N}\int_0^T \|\*x(t)-\overline{\*x}\|^2_{L^2(\Omega)}\diff t + \lambda \Big\| [w,b] \Big\|^2_{L^2(0,T;\mathfrak{U} )},
	\end{equation}
	where $\mathscr{E}$ is as in \eqref{eq: training.error.nonlocal} and $\loss(\cdot,\cdot)$ satisfies \Cref{ass: loss.turnpike}. We solely concentrate on System \eqref{eq: stacked.nonlocal.sigma.inside} to avoid a possibly abundance of technical details in the context of $\BV_t L^2_x$ analysis. Again, $P:L^2(\Omega)\longrightarrow\R^m$ is a surjective map, and $\*x_i\in P^{-1}(\{\vec{y}_i\})\subset L^2(\Omega)$ for $i\in[N]$ are arbitrary, but fixed.
	As expected, the analog exponential decay result holds for \eqref{eq: nonlocal.learning.tracking}.
	
	\begin{theorem}[Exponential stability] 
	\label{thm: sigma.inside.penalised.nonlocal}
	Fix $\lambda>0$, let $P \in \Lip(L^2(\Omega);\R^m)$ be any given surjective map, and let $\overline{\*x} \in L^2(\Omega)^N$ with $\overline{\*x}_i \in P^{-1}(\{\vec{y}_i\})$ for $i\in[N]$ be arbitrary but fixed.
	Suppose that  \eqref{eq: stacked.nonlocal.sigma.inside} is controllable with linear cost in some time $T_0>0$ in the sense of \Cref{def: ctrl}. 
	Then, there exists $T^*>0$ and constants $\mathfrak{C}=\mathfrak{C}\left(\{\vec{x}_i, \vec{y}_i\}_{i\in[N]},\lambda,N,\alpha, P\right)>0$ and $\mu=\mu\left(\{\vec{x}_i, \vec{y}_i\}_{i\in[N]},\lambda,N,\alpha\right)>0$ such that for any $T\geqslant T^*$, any pair of parameters $\left[w_T, b_T\right] \in L^2(0,T; \mathfrak{U})$ solving the minimization problem \eqref{eq: nonlocal.learning.tracking}, and the corresponding unique solution $\*x_T(\cdot)$ to \eqref{eq: stacked.nonlocal.sigma.inside} satisfy
	\begin{equation*} 
	\big\|w_T(t)\big\|+\big\|b_T(t)\big\| \leqslant \mathfrak{C}
	\end{equation*}
	for a.e. $t \in [0,T]$ and
	\begin{equation*} 
	\mathscr{E}(\*x_T(t)) + \|\*x_T(t)-\overline{\*x}\|_{L^2(\Omega)} \leqslant \mathfrak{C}\,e^{-\mu t}
	\end{equation*}
	for all $t \in [0,T]$.
	\end{theorem}
	
	\noindent
	The proof is omitted and left to the reader, as it follows precisely the same arguments as that of \Cref{thm: turnpike.P}.
	
	\section{Proofs} \label{sec: proofs}
	
	\subsection{Proof of \Cref{thm: no.running}} \label{sec: proof.no.running}
		   
	We note that both \eqref{eq: sigma.outside.i} and  \eqref{eq: sigma.inside.i} can be written in the compact form
	\begin{equation} \label{eq: compact.form.homogeneous}
	\begin{dcases}
	\dot{\*x}(t) = \mathfrak{f}([w(t), b(t)], \*x(t)) &\text{ in } (0, T) \\
	\*x(0) = \*x^0 \in \R^{d_x},
	\end{dcases}
	\end{equation}
	with 
	\begin{equation} \label{eq: homogeneity}
	\mathfrak{f}([0, 0], \*x) = 0, \hspace{0.75cm} \mathfrak{f}([\alpha w,\alpha b],\*x) = \alpha \mathfrak{f}([w, b], \*x) \hspace{0.25cm} \text{ for } \alpha > 0.
	\end{equation}
	We will refer to $u:= [w,b]$ as \emph{the control} of the ODE system, in accordance with control theory vocabulary.
    	We begin with following short but key lemma.
    
    	 \begin{lemma} \label{lem: scaling}
	Let $T_0>0$ and $\left[w_{T_0}, b_{T_0}\right] \in L^1(0, T_0; \R^{d_u})$ be given, and let $\*x_{T_0}(\cdot)$ be the unique solution to 
	\begin{equation} \label{eq: 5.7}
	\begin{dcases}
	\dot{\*x}_{T_0}(t) = \mathfrak{f}\left(\left[w_{T_0}(t), b_{T_0}(t)\right], \*x_{T_0}(t)\right) &\text{ in } (0, T_0) \\
	\*x_{T_0}(0) = \*x^0 \in \R^{d_x},
	\end{dcases}
	\end{equation}
	(i.e. \eqref{eq: compact.form.homogeneous} on $(0,T_0)$) with $\mathfrak{f}$ as in either \eqref{eq: sigma.inside.i} or \eqref{eq: sigma.outside.i}, thus satisfying \eqref{eq: homogeneity}. 
	Let $T>0$, and define
	\begin{equation*} \label{eq: uT}
	w_T(t):= \frac{T_0}{T} w_{T_0}\left(t\frac{T_0}{T} \right), \hspace{0.5cm} b_T(t):=\frac{T_0}{T} b_{T_0}\left(t\frac{T_0}{T} \right)  \hspace{1cm} \text{ for } t\in [0, T],
	\end{equation*}
	and 
	\begin{equation*} \label{eq: xT}
	\*x_T(t) := \*x_{T_0}\left(t\frac{T_0}{T}\right) \hspace{1cm} \text{ for } t \in [0, T].
	\end{equation*}
	Then $\*x_T(\cdot)$ is the unique solution to \eqref{eq: compact.form.homogeneous} (with the same $\mathfrak{f}$ as in \eqref{eq: 5.7}) associated to $\left[w_T, b_T\right]$.
	\end{lemma}
	
	\noindent
	This sort of time-scaling in the context of \emph{driftless control affine} systems is commonly used in control theoretical contexts -- a canonical example is the proof of the Chow-Rashevskii controllability theorem, see \cite[Chapter 3, Section 3.3]{coron2007control}. 
	We sketch the short proof for completeness.
    
    	\begin{proof}[Proof of \Cref{lem: scaling}]
	Since $\*x_{T_0}$ is the solution to \eqref{eq: 5.7}, the change of variable $\tau = s\frac{T}{T_0}$ as well as \eqref{eq: homogeneity}, we have
	\begin{align*}
	\*x_T(t) :=  \*x_{T_0}\left(t\frac{T_0}{T}\right) &= \*x^0 +\int_0^{t\frac{T_0}{T}} \mathfrak{f}\left(\left[w_{T_0}(s), b_{T_0}(s)\right], \*x_{T_0}(s)\right) \diff s \nonumber\\
	&= \*x^0  + \int_0^t \frac{T_0}{T} \mathfrak{f}\left(\left[w_{T_0}\left(\tau\frac{T_0}{T}\right), b_{T_0}\left(\tau\frac{T_0}{T}\right)\right], \*x_{T_0}\left(\tau\frac{T_0}{T}\right) \right)\diff \tau\nonumber\\
	&= \*x^0 + \int_0^t \mathfrak{f}\left(\left[w_T(\tau), b_T(\tau)\right], \*x_T(\tau)\right) \diff \tau.
	\end{align*}
	It follows that $\*x_T$ solves \eqref{eq: compact.form.homogeneous}, and we conclude by uniqueness.
	\end{proof}
    	
	\noindent
	We will also need the following lemma.
	
	\begin{lemma}[Compactness of the flow]
	\label{lem: compact.flow}
	Let $T>0$ be fixed. 
	The maps 
	\begin{enumerate}
	\item $\Phi_T: [w,b]\longmapsto\*x(\cdot)$ mapping $L^2(0,T;\R^{d_u})$ to $C^0([0,T]; \R^{d_x})$ where $\*x(\cdot)$ solves \eqref{eq: sigma.inside.i}, 
	\smallskip
	\item $\Phi_T: [w,b]\longmapsto\*x(\cdot)$ mapping $L^2(0,T;\R^{d_u})\cap\BV(0,T;\R^{d_u})$ to $C^0([0,T]; \R^{d_x})$ where $\*x(\cdot)$ solves \eqref{eq: sigma.outside.i}, 
	\smallskip
	\item $\Phi_T: [w,b]\longmapsto\*x(\cdot)$ mapping $H^1(0,T;\R^{d_u})$ to $C^0([0,T]; \R^{d_x})$ where $\*x(\cdot)$ solves \eqref{eq: sigma.outside.i}, 
	\end{enumerate} 
	are all compact.
	\end{lemma}
	
	\noindent
	We postpone the proof to the appendix.
	We are now in a position to prove the main result.

\begin{proof}[Proof of \Cref{thm: no.running}]
	We will henceforth, for notational convenience, extensively make use of the notation $u:=[w,b]$. We will focus on the neural ODE \eqref{eq: sigma.inside.i} and hence $k=0$. The case \eqref{eq: sigma.outside.i} and $k=1$ follows exactly the same arguments, and we will comment on the key differences at the end of the proof.
	\smallskip
	
	\noindent
	\textbf{Part 1.} We begin by showing 
	\begin{equation} \label{eq: first.step.regression}
	\mathscr{E} \left(\mathbf{x}_T(T)\right) \lesssim T^{-1}
	\end{equation}
	uniformly in $T$. By the interpolation assumption, there exists some $u^{1} \in L^2(0,1; \R^{d_u})$ such that the associated solution $\*x^{1}$ to \eqref{eq: sigma.inside.i} on $[0,1]$ satisfies $\mathscr{E}(\*x^1(1)) = 0$.
	Using the optimality of $u_T$ and the scaling relations from \Cref{lem: scaling}, we obtain
	\begin{align*}
	J_{\lambda, T}\left(u_T\right) &= \mathscr{E}\left(\*x_T(T)\right) + \lambda \left\|u_T\right\|^2_{L^2(0,T; \R^{d_u})} \\
	&\leqslant 
	\mathscr{E}(\*x^{1}(1)) + \frac{\lambda}{T}\left\| u^{1}\right\|_{L^2(0,1; \R^{d_u})}^2
	\end{align*}
	for all $T>0$. 
	Since $\mathscr{E}(\*x^{1}(1))=0$ by the interpolation assumption, the above inequality implies
	\begin{align} \label{eq: 5.28}
	0 \leqslant \mathscr{E}\left(\*x_T(T)\right)\leqslant \frac{\lambda}{T}\left\| u^{1}\right\|_{L^2(0,1; \R^{d_u})}^2
	\end{align}
	for all $T>0$. Estimate \eqref{eq: 5.28} clearly implies \eqref{eq: first.step.regression}. 
	
	\smallskip
	
	\noindent
	\textbf{Part 2.} 
	We now look to prove \eqref{eq: features.converge}. 
	To this end, we will look to show that $\{ \*x_T(T)\}_{T>0}$ is a bounded subset of $\R^{d_x}$. This will allow us to extract a converging sequence, whose limit will be shown to lie in $\{\mathscr{E}=0\}$.
	
	For any $T>0$, set 
	\begin{equation*}
	u^\aux(t) := \frac{1}{T} u^{1}\left(\frac{t}{T}\right) \hspace{1cm} \text{ for } t \in [0, T].
	\end{equation*}
	We argue similarly as in Part 1.
	Making use of \Cref{lem: scaling} once again, and since $\mathscr{E}(\*x^{1}(1))=0$, we see that
	\begin{align} \label{eq: A.9}
	J_{\lambda,T} \left(u^\aux\right) &=  \mathscr{E}\left(\*x^{1}(1)\right) + \frac{\lambda}{T} \left\| u^{1}\right\|^2_{L^2(0,1;\R^{d_u})} \nonumber\\ 
	&=\frac{\lambda}{T} \left\| u^{1} \right\|^2_{L^2(0,1;\R^{d_u})}.
	\end{align}
	Using the optimality of $u_T$, one sees that
	\begin{equation} \label{eq: A.10}
	J_{\lambda,T} \left(u^\aux\right) \geqslant J_{\lambda, T} \left(u_T\right) \geqslant  \lambda \left\| u_T\right\|_{L^2(0,T; \R^{d_u})}^2. 
	\end{equation}
	Combining \eqref{eq: A.10} and \eqref{eq: A.9}, we deduce that
	\begin{equation} \label{eq: bound_uT}
	\left\| u_T \right\|^2_{L^2(0,T;\R^{d_u})} \leqslant \frac{1}{T} \left\| u^{1} \right\|^2_{L^2(0,1;\R^{d_u})}
	\end{equation}
	for any $T>0$.
	Now by integrating \eqref{eq: sigma.inside.i}, and using the fact that $\sigma$ is globally Lipschitz continuous with constant $C(\sigma)>0$ and satisfies $\sigma(0) = 0$, for any $t \in [0,T]$ we have
    \begin{align*}
    \left\|\*x_T(t) - \*x^0\right\|\leqslant N^{\sfrac{1}{2}} C(\sigma) \int_0^t \left\|w_T(s)\right\|\left\|\*x_T(s)\right\| \diff s+N^{\sfrac{1}{2}}\left\|b_T\right\|_{L^1(0,T; \R^{d})}.
    \end{align*}
    By using the Gr\"onwall inequality, we obtain
    \begin{equation*}
        \left\|\*x_T(T) - \*x^0\right\|\leqslant N^{\sfrac{1}{2}}\left\|b_T\right\|_{L^1(0,T;\R^{d})}\exp\left(N^{\sfrac{1}{2}}C(\sigma)\int_0^T \left\|w_T(s)\right\|\diff s\right),
    \end{equation*}
    whereas by Cauchy-Schwarz, it follows that
    \begin{equation*}
        \left\|\*x_T(T) - \*x^0\right\|\leqslant T^{\sfrac{1}{2}}N^{\sfrac{1}{2}}\left\|b_T\right\|_{L^2(0,T; \R^{d})}\exp\left(T^{\sfrac{1}{2}}N^{\sfrac{1}{2}}C(\sigma) \left\|w_T\right\|_{L^2(0,T; \R^{d\times d})}\right).
    \end{equation*}
    At this point, employing \eqref{eq: bound_uT}, we deduce
    \begin{equation*}
        \left\|\*x_T(T) - \*x^0\right\|\leqslant N^{\sfrac{1}{2}} \left\|u^{1}\right\|_{L^2(0,1;\R^{d_u})}
   \,\exp\left(N^{\sfrac{1}{2}}C(\sigma) \left\|u^{1}\right\|_{L^2(0,1;\R^{d_u})}
 \right).
    \end{equation*}
    Since $u^{1}$ is independent of $T$, we conclude that the set $\{\*x_T(T)\}_{T>0}$ is bounded. 
    Whence, there exists a sequence $\{T_n\}_{n=1}^{\infty}$ with $T_n>0$ and $T_n\longrightarrow \infty$ as $n\longrightarrow \infty$ and some $\*x_\circ \in \R^{d_x}$ such that 
	\begin{equation} \label{eq: final.time.conv.xn}
	\*x_{T_n}(T_n) \longrightarrow \*x_\circ \hspace{1cm} \text{ as } n\longrightarrow\infty.
	\end{equation}
	Since $\mathscr{E}\left(\*x_{T_n}(T_n)\right)\longrightarrow 0$ as $n\longrightarrow\infty$ by \eqref{eq: first.step.regression}, by continuity of $\mathscr{E}$, we have $\mathscr{E}(\*x_\circ)=0$.
	This concludes the proof of \eqref{eq: features.converge}.
	\smallskip
	
	\noindent
	\textbf{Part 3.} We now address the third statement of the theorem.
	To this end, we will first show that the sequence $\{u_n\}_{n=1}^{\infty}$ defined in the statement is bounded in $L^2(0,1; \R^{d_u})$.
	
	 Let $u^{\dagger} \in L^2(0,1;\R^{d_u})$ be any solution to 
	\begin{equation} \label{eq: A.13}
		\inf_{\substack{u\in L^2(0,1;\R^{d_u}) \\ \*x(\cdot) \text{ solves } \eqref{eq: sigma.inside.i}  \\ \text{ and } \\ \mathscr{E}(\*x(1))=0}} \int_0^{1} \| u(t)\|^2 \diff t.		
	\end{equation}
	Denote by $\*x^\dagger$ the corresponding solution to \eqref{eq: sigma.inside.i} on $[0,1]$.
	We claim that
	\begin{equation}\label{claim proof thm no.ruuning}
	\|u_n\|_{L^2(0,1; \R^{d_u})} \leqslant \left\| u^\dagger \right\|_{L^2(0,1; \R^{d_u})}, \hspace{1cm} \text{ for all } n\geqslant 1.
	\end{equation}
	We prove this claim by contradiction.
	Indeed, assume that we had 
	\begin{equation*}
	\left\|u^\dagger\right\|_{L^2(0,1; \R^{d_u})} <\|u_n\|_{L^2(0,1; \R^{d_u})} \hspace{1cm} \text{ for some } n\geqslant 1.
	\end{equation*}
	We consider 
	\begin{equation*}
	u^{\dagger}_{n}(t):= \frac{1}{T_n}u^{\dagger}\left( \frac{t}{T_n}\right) \hspace{1cm} \text{ for } t \in [0, T_n],
	\end{equation*}
	whose corresponding state $\*x_{n}^\dagger$, solution to \eqref{eq: sigma.inside.i} on $[0,T_n]$, satisfies $\*x^{\dagger}_{n}(T_n)= \*x^\dagger(1)$ by \Cref{lem: scaling}. On another hand, by assumption we have $\mathscr{E}(\*x^{\dagger}(1))=0$.
	It then follows that
	\begin{align*}
		J_{\lambda, T_n}\left(u^{\dagger}_{n}\right) &= \frac{\lambda}{T_n}\left\| u^{\dagger}\right\|^2_{L^2(0,1; \R^{d_u})} \\
		&< \mathscr{E}\left(\*x_{T_n}(T_n)\right) + \frac{\lambda}{T_n} \|u_n\|^2_{L^2(0,1; \R^{d_u})} = J_{T_n}\left(u_{T_n}\right),
	\end{align*}
	which contradicts the fact that $u_{T_n}$ minimizes $J_{T_n}$.
	Hence, \eqref{claim proof thm no.ruuning} holds, and $\{u_n\}_{n=1}^{\infty}$ is bounded in $L^2(0, 1; \R^{d_u})$.
	Consequently, by the Banach-Alaoglu theorem, there exists $u^{*} = \left[w^{*}, b^{*}\right] \in L^2(0,1;\R^{d_u})$ such that
	\begin{equation*}
	u_n \rightharpoonup u^{*} \hspace{1cm} \text{ weakly in } L^2(0,1; \R^{d_u}),
	\end{equation*}
	along some subsequence as $n\longrightarrow\infty$. 
	Moreover, using the properties of equation \eqref{eq: sigma.inside.i} (\Cref{lem: compact.flow}), we deduce that
	 the trajectory $\*x_n$ associated to $u_n$ satisfies
	 \begin{equation} \label{eq: strong.conv.xn}
	 \*x_n \longrightarrow \*x^{*} \hspace{1cm} \text{ strongly in } C^0([0, 1]; \R^{d_x})
	 \end{equation}
	 as $n\longrightarrow\infty$, where $\*x^{*}$ is the solution to \eqref{eq: sigma.inside.i} on $[0,1]$, associated to $u^{*}$.
	 On another hand, note that by \Cref{lem: scaling}, $\*x_{T_n}(t) = \*x_n(\frac{t}{T_n})$ for $t \in [0, T_n]$, whence $\*x_{T_n}(T_n) = \*x_n(1)$ and thus, combining \eqref{eq: strong.conv.xn} and \eqref{eq: final.time.conv.xn}, we see that $\*x^{*}(1)=\*x_\circ$.
	 Consequently, $u^{*}$ is a control such that $\mathscr{E}(\*x^{*}(1))=\mathscr{E}(\*x_\circ) =0$, thus satisfying the constraint in \eqref{eq: A.13}.
	In view of this, we may also use \eqref{claim proof thm no.ruuning} and the weak lower semicontinuity of the $L^2$--norm to write
	\begin{align} \label{eq: aux.long.estimates}
	\left\|u^{\dagger}\right\|_{L^2(0,1; \R^{d_u})} \leqslant \left\|u^{*}\right\|_{L^2(0,1; \R^{d_u})} &\leqslant \liminf_{n\longrightarrow \infty} \|u_n\|_{L^2(0,1;\R^{d_u})} \nonumber \\
	&\leqslant \lim_{n\longrightarrow \infty} \|u_n\|_{L^2(0,1;\R^{d_u})} \nonumber \\
	&\leqslant \limsup_{n\longrightarrow \infty} \|u_n\|_{L^2(0,1;\R^{d_u})} \nonumber \\
	&\leqslant \left\|u^\dagger\right\|_{L^2(0,1;\R^{d_u})},
	\end{align}
	clearly implying that
	\begin{equation*} \label{eq: conv.norms}
	\lim_{n\longrightarrow \infty} \left\|u_n\right\|_{L^2(0,1;\R^{d_u})} = \left\|u^{*}\right\|_{L^2(0,1;\R^{d_u})}.
	\end{equation*}
	Hence, as weak convergence and convergence of the norms in $L^2$ implies strong convergence in $L^2$, we deduce that
	\begin{equation*}
	u_n \longrightarrow u^{*} \hspace{1cm} \text{ strongly in } L^2(0, 1; \R^{d_u})
	\end{equation*}
	along some subsequence as $n\longrightarrow\infty$.
	Moreover, from \eqref{eq: aux.long.estimates} we deduce that, since $u^\dagger$ is a solution to \eqref{eq: A.13} and since $u^{*}$ satisfies the constraints therein, $u^{*}$ is a solution to \eqref{eq: A.13} as well, which concludes the proof for \eqref{eq: sigma.inside.i} and $k=0$.
		\smallskip
		
		In the case \eqref{eq: sigma.outside.i} and $k=1$, one may clearly repeat the above reasoning, replacing $L^2(0,T; \R^{d_u})$ by $H^1(0,T;\R^{d_u})$ throughout, with some key additions. 
		
		In Part 1, we first note that instead of \eqref{eq: A.9}, one has
		\begin{align*}
		    J_{\lambda,T}(u^\aux) &= \mathscr{E}(\*x^{1}(1)) + \frac{\lambda}{T} \left\|u^{1}\right\|_{L^2(0,1;\R^{d_u})}^2 + \frac{\lambda}{T^3} \left\|\dot{u}^{1}\right\|_{L^2(0,1;\R^{d_u})}^2\\ 
		    &=\frac{\lambda}{T} \left\|u_{T_0}\right\|_{L^2(0,1;\R^{d_u})}^2 + \frac{\lambda}{T^3} \left\|\dot{u}^{1}\right\|_{L^2(0,1;\R^{d_u})}^2.
		\end{align*}
		This is not an impediment to \eqref{eq: A.10}, which remains true, and one can clearly deduce that $\{\*x_T(T)\}_{T>0}$ is bounded as well. 
		Similarly, \eqref{eq: 5.28} holds with a bound of the form
		\begin{equation*}
		    0 \leqslant \mathscr{E}(\*x_T(T)) \leqslant \frac{\lambda}{T} \left\|u^{1}\right\|_{L^2(0,1;\R^{d_u})}^2 + \frac{\lambda}{T^3} \left\|\dot{u}^{1}\right\|_{L^2(0,1;\R^{d_u})}^2.
		\end{equation*}
		Whence the remainder of parts 1 and 2 hold in this context as well.
		
		In Part 3, we emphasize the sole key difference between \eqref{eq: sigma.inside.i} and \eqref{eq: sigma.outside.i} -- the weak $L^2$--convergence of $\{u_n\}_{n=1}^{\infty}$ is a priori not sufficient to entail the strong convergence in \eqref{eq: strong.conv.xn} in the case of \eqref{eq: sigma.outside.i}. 
		However, by the Rellich-Kondrachov compactness theorem, the weak $H^1$--convergence of $\{u_n\}_{n=1}^{\infty}$ implies a strong $L^2$--convergence along a subsequence, which would yield \eqref{eq: strong.conv.xn} by arguing just as in the proof of \Cref{lem: compact.flow}.
 		
		This concludes the proof.
\end{proof}

	\subsection{Proof of \Cref{thm: no.running.lambda}}
	
	The proof closely follows the lines of that just above.
	Let us consider $k=1$, since the case $k=0$ is equivalent to \Cref{thm: no.running.lambda}.
	We present minimal details for completeness.
	
	\begin{proof}[Proof of \Cref{thm: no.running.lambda}]
	We again make use of the notation $u:=[w,b]$.	
	We first show 
	\begin{equation} \label{eq: linear.lambda.proof}
	\mathscr{E}(\*x_\lambda(T)) \lesssim \lambda 
	\end{equation}
	uniformly in $\lambda>0$ -- we argue as in the proof of \Cref{thm: no.running} just above, exhibiting, by the interpolation assumption, parameters $u^1 \in L^2(0,1; \R^{d_u})$ such that $\mathscr{E}(\*x^1(1))=0$. We may obtain an estimate like \eqref{eq: 5.28} and conclude.
	Now, the same arguments as in Part 2 of the proof of \Cref{thm: no.running} may be used to deduce that $\{\*x_\lambda(T)\}_{\lambda>0}$ is a bounded subset of $\R^d$, and hence there exists a sequence $\{\lambda_n\}_{n=1}^{\infty}$ of positive numbers with $\lambda_n\searrow0$ as $n\longrightarrow\infty$ and some $\*x_\circ \in \R^{d_x}$ such that 
	\begin{equation*}
	\*x_{\lambda_n}(T) \xrightarrow[n\longrightarrow\infty]{} \*x_\circ.
	\end{equation*}
	Using \eqref{eq: linear.lambda.proof} we deduce that $\mathscr{E}(\*x_\circ)=0$.
	Finally, the proof of the last fact is identical to that done for \Cref{thm: no.running}, so we omit it.
	\end{proof}

	\subsection{Proof of \Cref{thm: thm.classification.lambda}}
	
	We now provide a proof of our main result in the context of classification tasks.
		
	\begin{proof}[Proof of \Cref{thm: thm.classification.lambda}]
	
	We recall that, by assumption, $\*x_i^0:=\mathfrak{Q}\vec{x}_i\geqslant0$ for $i\in[N]$.
	Now let $\left[\widehat{w},\widehat{b}\right]\in H^k(0,T_0;\R^{d_u})$ be a pair of parameters which separates the training dataset $\left\{\*x_i^0,\vec{y}_i\right\}_{i\in[N]}$ with respect to $P$ in time $T_0>0$, i.e., such that the solution $\widehat{\*x} =[\widehat{\*x}_1, \ldots, \widehat{\*x}_N]$ to \eqref{eq: sigma.outside.i}, with initial condition $\*x^0=[\*x_1^0,\ldots,\*x_N^0]$, corresponding to $\left[\widehat{w},\widehat{b}\right]$, satisfies
	\begin{equation}\label{sparated.gamma}
	\min _{i\in [N]} \left\{P\,\widehat{\*x}_i (T_0)_{\vec{y}_i} - \max_{\substack{j\in[N]\\ j\neq \vec{y}_i}} P\,\widehat{\*x}_i(T_0)_j \right\}=:\gamma >0.
	\end{equation}
	Now fix an arbitrary $\alpha \in \left(0,\frac{1}{2}\right)$, and, for any $T>0$, define
	\begin{equation*}
	\left[w^\dagger(t), b^\dagger(t)\right]:=
	\begin{dcases}
	\dfrac{2T_0}{T}	\left[\widehat{w}\left(t \dfrac{2T_0}{T}  \right), \widehat{b}\left(t \dfrac{2T_0}{T}  \right)\right] & \text{for} \ t\in \left[ 0,\dfrac{T}{2}\right] \\
	T^{\alpha-1}  \left[\text{Id}_d,0_d \right] &\text{for} \ t\in \left(\dfrac{T}{2} ,T\right],
	\end{dcases} 
	\end{equation*}
	where $\text{Id}_d$ is the identity matrix in $\R^{d\times d}$ and $0_d$ is the zero vector in $\R^d$.
	By virtue of the scaling in \Cref{lem: scaling}, for $t\in \left[\frac{T}{2},T\right]$, the trajectories $\*x^\dagger = \left[\*x^\dagger_1, \ldots, \*x^\dagger_N\right]$ associated to $\left[w^\dagger, b^\dagger\right]$ are given by the solution to
	\begin{equation}\label{dyn.with.ident}
	\begin{dcases}
	\dot{\*x}^\dagger_i (t) = 
	\sigma\left(T^{\alpha-1} \*x^\dagger_i(t)\right) &\text{ for } t\in \left[ \dfrac{T}{2},T\right] \\
	\*x^\dagger_i \left( \dfrac{T}{2} \right) = \widehat{\*x}_i \left(T_0\right) .
	\end{dcases}
	\end{equation}
	Moreover, since $\sigma(x) = \max\{x,0\}\geqslant0$, the right hand side in \eqref{eq: sigma.outside.i} is nonnegative. 
	Using the assumption that the initial conditions are of the form $\*x^0_i = \mathfrak{Q}\vec{x}_i \geqslant 0$, it follows that $\widehat{\*x}_i(T_0)\geqslant 0$ for all $i\in [N]$. 
	We can therefore drop $\sigma$ in \eqref{dyn.with.ident} and deduce that 
	$P\*x^\dagger_i(t)$ solves
	\begin{equation}\label{dyn.with.ident}
	\begin{dcases}
	\frac{\diff}{\diff t}P\*x^\dagger_i(t) =T^{\alpha-1}P\*x^\dagger_i(t) &\text{ for } t\in \left[\dfrac{T}{2},T\right] \\
	P\*x^\dagger_i\left(\dfrac{T}{2} \right)= P\widehat{\*x}_i\left( T_0 \right).
	\end{dcases}
	\end{equation}
	Hence, we have
$$
P\*x^\dagger_i(t) = P\widehat{\*x}_i(T_0) e^{T^{\alpha-1}\left(t-\sfrac{T}{2}\right)}, \hspace{1cm} \text{ for all } t\in\left[\dfrac{T}{2} ,T\right].
$$	
Now, using the definition of the cross-entropy loss and the margin $\gamma$ in \eqref{sparated.gamma}, 
we compute, for any $i\in [N]$,
\begin{align*}
\loss\left(\*x^\dagger_i (T), \vec{y}_i\right) &= -\log \left(\dfrac{e^{P\widehat{\*x}_i(T_0)_{\vec{y}_i}e^{\frac{T^{\alpha}}{2}}} }{\sum_{j=1}^m e^{P\widehat{\*x}_i(T_0)_j e^{\frac{T^{\alpha}}{2}}}} \right) \\
&= \log \left( 1+ \sum_{j\neq \vec{y}_i} e^{\left(P\widehat{\*x}_i(T_0)_j e^{\frac{T^{\alpha}}{2}}\right)-\left(P\widehat{\*x}_i(T_0)_{\vec{y}_i}e^{\frac{T^{\alpha}}{2}}\right)} \right) \\
&\leqslant \log \left( 1+ (m-1) \exp \left(-\gamma\, \exp \left(\frac{T^{\alpha}}{2}\right)\right) \right).
\end{align*}
Then, we can estimate
\begin{equation}\label{error estimate log.exp.exp}
\mathscr{E}\left(\*x^\dagger (T)\right) \leqslant \log \left( 1+ (m-1) \exp \left(-\gamma\, \exp \left(\frac{T^{\alpha}}{2}\right)\right) \right).
\end{equation}
On the other hand, using the definition of $\left[w^\dagger, b^\dagger\right]$, we deduce 
\begin{align*}
\left\| \left[w^\dagger, b^\dagger\right] \right\|_{H^1 (0,T; \R^{d_u})}^2 &=
\left\| \left[w^\dagger, b^\dagger\right] \right\|_{H^1\left(0,\frac{T}{2}; \R^{d_u}\right)}^2 + \left\| \left[w^\dagger, b^\dagger\right]\right\|_{H^1\left(\frac{T}{2},T\right)}^2 \\
&\leqslant \dfrac{C_1}{T} + C_2\, T^{2(\alpha-1)} T,
\end{align*}
for some constants $C_1,C_2>0$ depending only on $\lambda, T_0$ and $\left[\widehat{w},\widehat{b}\right]$.
From this estimate, together with \eqref{error estimate log.exp.exp}, we obtain, for $T>T_0$,
$$
J_{\lambda,T}\left(w^\dagger, b^\dagger\right) \leqslant \log \left( 1+ (m-1) \exp \left(-\gamma\, \exp \left(\frac{T^{\alpha}}{2}\right)\right) \right) + C T^{2\alpha-1}, 
$$
for some constant $C>0$ depending on $\lambda$, $T_0$,  $\left[\widehat{w},\widehat{b}\right]$, but independent of $T$.
Using the above estimate, we may conclude from the optimality of $[w_T, b_T]$, as
\begin{align*}
\mathscr{E}\left(\*x_T(T)\right) \leqslant J_{\lambda,T}\left(w_T,b_T\right) &\leqslant 
J_{\lambda,T} \left(w^\dagger, b^\dagger\right) \\
&\leqslant \log \left( 1+ (m-1) \exp\left(-\gamma\, \exp\left(\frac{T^{\alpha}}{2}\right)\right) \right) + C T^{2\alpha-1}.
\end{align*}
	\end{proof}

	\subsection{Proof of \Cref{thm: turnpike.P}} \label{sec: proof.turnpike}
	
	The proof of \Cref{thm: turnpike.P} is rather involved -- we will first require the following lemmas. We will focus on System \eqref{eq: sigma.outside.i}, but with small adaptations the proofs (except the exponential stability estimate for the parameters) can be shown for systems of the form \eqref{eq: compound.neural.ode}. System \eqref{eq: sigma.inside.i} is addressed in \citep{esteve2020turnpike}.
	
	\begin{lemma}[Gr\"onwall-like estimate] 
	\label{lem: turnpike.1}
	Let $T>0$ be given, and let $\overline{\*x}\in\R^{d_x}$. For any $[w,b]\in L^1(0,T;\R^{d_u})$ and $\*x^0\in\R^{d_x}$, let $\*x\in C^0([0,T];\R^{d_x})$ denote the unique solution to \eqref{eq: sigma.outside.i}, noting \eqref{eq: form.controls}. The following facts then hold.
	\begin{enumerate}
	\item There exist a couple of constants $C_1=C_1\left(\sigma,\|\overline{\*x}\|,\max\left\{1,\left\|\*x^0-\overline{\*x}\right\|\right\}, N\right)>0$ and $C_2=C_2(\sigma, \|\overline{\*x}\|, N)>0$ independent of $T$ such that
	\begin{equation*}
	\|\*x(t)-\overline{\*x}\|\leqslant \mathfrak{C}\,\left(\left\|\*x^0-\overline{\*x}\right\|+\Big\|[w,b]\Big\|_{L^1(0,T;\R^{d_u})} + \|\*x-\overline{\*x}\|_{L^2(0,T;\R^{d_x})}\right)
	\end{equation*}
	holds for all $t\in[0,T]$, where 
	\begin{equation*}
	\mathfrak{C}:=C_1\max\left\{1, \Big\|[w,b]\Big\|_{L^1(0,T;\R^{d_u})}\right\}\exp\left(C_2\big\|w\big\|_{L^1(0,T;\R^{d\times d})}\right).
	\end{equation*}
	\item 
	If moreover $[w,b]\in L^2(0,T;\R^{d_u})$, then 
	\begin{equation*}
	\|\*x(t)-\overline{\*x}\|\leqslant \mathfrak{C}\,\left(\left\|\*x^0-\overline{\*x}\right\|+\Big\|[w,b]\Big\|_{L^2(0,T;\R^{d_u})} + \|\*x-\overline{\*x}\|_{L^2(0,T;\R^{d_x})}\right)
	\end{equation*}
	also holds for all $t\in[0,T]$, where
	\begin{equation*}
	\mathfrak{C}:=C_1\max\left\{1, \Big\|[w,b]\Big\|_{L^2(0,T;\R^{d_u})}\right\}\exp\left(C_2\big\|w\big\|_{L^2(0,T;\R^{d\times d})}\right).
	\end{equation*}
	\end{enumerate}
	\end{lemma}
	
	\noindent 
	The proof of \Cref{lem: turnpike.1} may be found in the Appendix.
	Now suppose $\*x^{\tau_\circ}\in\R^{d_x}$ is given. Let $T>0$ and $\tau_\circ\in[0,T)$ be fixed, and consider the cost functional
	\begin{equation} \label{eq: j.tau.T}
	J_{\tau_\circ,T}(w,b):=\frac{1}{N}\int_{\tau_\circ}^T\|\*x(t)-\overline{\*x}\|^2\diff t+\lambda\Big\|[w,b]\Big\|_{L^2(\tau_\circ,T;\R^{d_u})}^2,
	\end{equation}
	with $\*x(\cdot)$ being the solution to 
	\begin{equation}\label{eq: system.tau0}
	\begin{dcases}
	\dot{\*x}(t) = \mathfrak{f}\left([w(t),b(t)], \*x(t)\right) &\text{ in } (\tau_{\circ}, T)\\
	\*x(\tau_\circ)=\*x^{\tau_\circ},
	\end{dcases}
	\end{equation}
	with $\mathfrak{f}$ as the dynamics in \eqref{eq: sigma.inside.i} or \eqref{eq: sigma.outside.i}.
	We then have the following result.
	
	\begin{lemma}[Uniform estimate of optimal trajectories]
	\label{lem: turnpike.2}
	Fix $\lambda>0$, let $P \in \Lip(\R^d;\R^m)$ be any given surjective map, and let $\overline{\*x} \in \R^{d_x}$ with $\overline{\*x}_i \in P^{-1}(\{\vec{y}_i\})$ for $i\in[N]$ be arbitrary but fixed. Let $r>0$ be fixed.
	Suppose that \eqref{eq: sigma.inside.i} (resp. \eqref{eq: sigma.outside.i} with $\sigma$ $1$--homogeneous) is controllable with linear cost in the sense of \Cref{def: ctrl}. 
	Then, there exists a constant 
	$\mathfrak{C}=\mathfrak{C}\left(\sigma, \|\overline{\*x}\|, \lambda, N, d, r\right)>0$ such that for all $T>0$, $\tau_\circ\in[0,T)$ and $\*x^{\tau_\circ}\in\R^{d_x}$ such that
	\begin{equation*}
	\left\|\*x^{\tau_\circ}-\overline{\*x}\right\|\leqslant r,
	\end{equation*}
	any pair of parameters $\left[w_T, b_T\right]\in L^2(\tau_\circ,T; \R^{d_u})$ minimizing $J_{\tau_\circ,T}$ and corresponding solution $\*x_T(\cdot)$ to \eqref{eq: system.tau0}, with $\mathfrak{f}$ as in \eqref{eq: sigma.inside.i} (resp \eqref{eq: sigma.outside.i} with $\sigma$ $1$--homogeneous), are such that
	\begin{align*} 
	\Big\|\left[w_T,b_T\right]\Big\|_{L^2(\tau_\circ,T;\R^{d_u})} + \left\|\*x_T-\overline{\*x}\right\|_{L^2(\tau_\circ,T; \R^{d_x})} &+ \|\*x_T(t)-\overline{\*x}\|\leqslant \mathfrak{C}\,\left\|\*x^{\tau_\circ}-\overline{\*x}\right\|
	\end{align*}
	holds for all $t\in[\tau_\circ,T]$.	
	\end{lemma}
	
	\noindent 
	The proof of \Cref{lem: turnpike.2} follows precisely the arguments given in the first step of the proof of \Cref{thm: turnpike.BV}, in which one solely needs to change the initial time ($0$ by $\tau_\circ$) and datum ($\*x^0$ by $\*x^{\tau_\circ}$), the $\BV$--norms by $L^2$, and note the precise form of the constant $\mathfrak{C}$ in \eqref{eq: constant.desired}. We give a brief sketch in the appendix for the sake of clarity.
	
	We will also make use of the following short observation.
	
	\begin{lemma} \label{lem: doesnt.work.in.BV}
	Let $T>0$ and let $[w_T,b_T]\in L^2(0,T;\R^{d_u})$ be a pair of minimizers to $J_T$ defined in \eqref{eq: time.dep.func}, and denote by $\*x_T$ the corresponding solution to \eqref{eq: system.tau0} on $[0,T]$ with $\*x_T(0)=\*x^0$. Let $\tau_\circ\in[0,T)$ be given. Then $[w_T,b_T]|_{[\tau_\circ,T]}$ minimize $J_{\tau_\circ,T}$ defined in \eqref{eq: j.tau.T} for System \eqref{eq: system.tau0} with initial data $\*x^{\tau_\circ}=\*x_T(\tau_\circ)$.
	\end{lemma}
	
	\noindent
	We refer to the appendix for a proof.

	\begin{remark}
	We note that our proof is specific to $L^2$--regularization and does not transfer to $\BV$--regularization a priori, due to the fact that the $\BV$ norm may see the singularities in discontinuous parameters. Therein lies the main impediment to ensuring the validity of \Cref{thm: turnpike.P} for $\BV$--regularized problems by means of our strategy.
	\end{remark}
	
	\noindent	
	Before concluding this section with a proof of \Cref{thm: turnpike.P}, we also state the following key lemma.
		
	\begin{lemma} \label{lem: banach.lemma}
	Let $X$ be a Banach space, $T>a\geqslant0$ and $f\in C^0([a,T];X)$. For any $\tau\leqslant T-a$, there exists $t_1\in[a,a+\tau)$ such that
	\begin{equation*}
	\left\|f\left(t_1\right)\right\|_X\leqslant\frac{\|f\|_{L^2(a,T;X)}}{\sqrt{\tau}}.
	\end{equation*}
	\end{lemma}
	
	\noindent
	The proof may be found in the appendix.
	We may now provide the proof of \Cref{thm: turnpike.P}.

	\begin{proof}[Proof of \Cref{thm: turnpike.P}]
	
	We shall concentrate on the neural ODE \eqref{eq: sigma.outside.i}. 
	The proof of the full result for \eqref{eq: sigma.inside.i} is identical to that presented in \citep{esteve2020turnpike}. We split the proof in two parts.
	\smallskip
	
	\noindent
	\textbf{Part 1:} \emph{Stability estimates for $\mathscr{E}(\*x_T(t))+\|\*x_T(t)-\overline{\*x}\|$}. 
	For 
	\begin{equation*}
	r:=\left\|\*x^0-\overline{\*x}\right\|,
	\end{equation*}
	denote by $\mathfrak{C}_1=\mathfrak{C}_1(\sigma,\|\overline{\*x}\|,\lambda,N,d,r)>0$ the universal constant given by \Cref{lem: turnpike.2}. 
	Let 
	\begin{equation*}
	\tau>\max\left\{\mathfrak{C}_1^4, \mathfrak{C}_1^2\right\}
	\end{equation*}
	be fixed and let 
	\begin{equation*}
	T\geqslant \tau+1.
	\end{equation*}
	First, note that by \Cref{lem: turnpike.2} (with $\tau_\circ=0$, $r:=\left\|\*x^0-\overline{\*x}\right\|$ and $\*x^{\tau_\circ}=\*x^0$), we have
	\begin{equation} \label{eq: 7.20}
	\|\*x_T-\overline{\*x}\|_{L^2(0,T;\R^{d_x})}+\|\*x_T(t)-\overline{\*x}\|\leqslant\mathfrak{C}_1\,\left\|\*x^0-\overline{\*x}\right\| \hspace{1cm} \text{ for } t\in[0,T].
	\end{equation}
	Now note that for $t\in[0,\tau+1]$, the desired exponential stability estimate can easily be obtained since the length of the time interval is independent of $T$. Indeed, from \eqref{eq: 7.20} we get
	\begin{align*}
	\|\*x_T(t)-\overline{\*x}\|&\leqslant\mathfrak{C}_1\,\left\|\*x^0-\overline{\*x}\right\| \exp(t)\exp(-t)\\
	&\leqslant \mathfrak{C}_1\,\left\|\*x^0-\overline{\*x}\right\| \exp(\tau+1)\exp(-t),
	\end{align*}
	and, since $\overline{\*x}_i\in P^{-1}(\{\vec{y}_i\})$ for $i\in[N]$, using this estimate we also find
	\begin{align*}
	\mathscr{E}(\*x_T(t))\leqslant \frac{c}{N}\sum_{i=1}^N \big\|P\*x_{T,i}(t)-\vec{y}_i\big\|^\alpha&\leqslant \frac{c}{N}\|P\|^\alpha\sum_{i=1}^N\|\*x_{T,i}(t)-\overline{\*x}_i\|^\alpha\\ 
	&\lesssim \mathfrak{C}_1^\alpha \left\|\*x^0-\overline{\*x}\right\|^\alpha \exp(\alpha(\tau+1)) \exp(-\alpha t).
	\end{align*}
	Thus, only the case $t\in[\tau+1,T]$ remains. To do so, we proceed in three steps.
	\begin{enumerate}
	\item[Step 1).]\textbf{Preparation.} Since $\tau\leqslant T$, using \Cref{lem: banach.lemma} and then \Cref{lem: turnpike.2} (with $\tau_\circ=0$, $r:=\left\|\*x^0-\overline{\*x}\right\|$ and $\*x^{\tau_\circ}=\*x^0$) we see that there exists $\tau_\circ\in[0,\tau)$ such that
	\begin{equation} \label{eq: 7.8}
	\left\|\*x_T(\tau_\circ)-\overline{\*x}\right\|\leqslant \frac{\|\*x_T-\overline{\*x}\|_{L^2(0,T;\R^{d_x})}}{\sqrt{\tau}} \leqslant \frac{\mathfrak{C}_1}{\sqrt{\tau}}\left\|\*x^{0}-\overline{\*x}\right\|.
	\end{equation}
	By \Cref{lem: doesnt.work.in.BV}, the parameters $[w_T,b_T]|_{[\tau_\circ,T]}$ minimize the functional $J_{\tau_\circ,T}$ for System \eqref{eq: system.tau0} with initial data $\*x^{\tau_\circ}=\*x_T(\tau_\circ)$, to which the solution is precisely $\*x_T|_{[\tau_\circ,T]}$. 
	Now, applying \Cref{lem: turnpike.2} (this time with with $\tau_\circ$ as in \eqref{eq: 7.8}, $\*x^{\tau_\circ}=\*x_T(\tau_\circ)$ and with $r=\left\|\*x^0-\overline{\*x}\right\|$, as $\sfrac{\mathfrak{C}_1}{\sqrt{\tau}}<1$) in combination with \eqref{eq: 7.8}, yields
	\begin{equation} \label{eq: bootstrap}
	\|\*x_T(t)-\overline{\*x}\|\leqslant\mathfrak{C}_1\,\left\|\*x_T(\tau_\circ)-\overline{\*x}\right\| \leqslant \frac{\mathfrak{C}_1^2}{\sqrt{\tau}}\left\|\*x^{0}-\overline{\*x}\right\|,
	\end{equation}
	which holds for all $t\in[\tau_\circ,T]$. Since $\tau_\circ<\tau$, \eqref{eq: bootstrap} also holds for $t\in[\tau,T]$.
	\smallskip
	\item[Step 2).]\textbf{Bootstrap.}
	We iterate \eqref{eq: bootstrap} and show that for any $n\in\N$ satisfying $n\leqslant\sfrac{T}{\tau}$, the following estimate holds: 
	\begin{equation} \label{eq: induction}
	\|\*x_T(t)-\overline{\*x}\|\leqslant \left(\frac{\mathfrak{C}_1^2}{\sqrt{\tau}}\right)^n\left\|\*x^{0}-\overline{\*x}\right\| \hspace{1cm} \text{ for } t\in[n\tau,T].
	\end{equation} 
	We proceed by induction -- the case $n=1$ holds by \eqref{eq: bootstrap}. Thus suppose that \eqref{eq: induction} holds at some stage $n\geqslant2$ and suppose that $n+1\leqslant\sfrac{T}{\tau}$.
	Now the parameters $[w_T,b_T]|_{[n\tau,T]}$ minimize $J_{n\tau,T}$ by \Cref{lem: doesnt.work.in.BV}, and so we can apply \Cref{lem: turnpike.2} (with $\tau_\circ=n\tau$, $\*x^{\tau_\circ}=\*x_T(n\tau)$, and with $r=\left\|\*x^0-\overline{\*x}\right\|$ as $\sfrac{\mathfrak{C}_1^2}{\sqrt{\tau}}<1$ and \eqref{eq: induction} is assumed to hold) in combination with \Cref{lem: banach.lemma} (as $n+1\leqslant\sfrac{T}{\tau}$ clearly implies that $\tau\leqslant T-n\tau$) to deduce that there exists some time $t_1\in[n\tau,(n+1)\tau)$ such that
	\begin{equation*}
	\|\*x_T(t_1)-\overline{\*x}\|\leqslant\frac{\|\*x_T-\overline{\*x}\|_{L^2(n\tau,T;\R^{d_x})}}{\sqrt{\tau}}\leqslant \frac{\mathfrak{C}_1}{\sqrt{\tau}}\|\*x_T(n\tau)-\overline{\*x}\|.
	\end{equation*}
	We once again apply \eqref{eq: induction} in the inequality just above to deduce
	\begin{equation} \label{eq: 7.12}
	\|\*x_T(t_1)-\overline{\*x}\|\leqslant \frac{\mathfrak{C}_1}{\sqrt{\tau}} \left(\frac{\mathfrak{C}_1^2}{\sqrt{\tau}}\right)^n\left\|\*x^0-\overline{\*x}\right\|.
	\end{equation}
	We then apply \Cref{lem: turnpike.2} for a second time (with $\tau_\circ=t_1$, $\*x^{\tau_\circ}=\*x_T(t_1)$ and with $r=\left\|\*x^0-\overline{\*x}\right\|$, as we have \eqref{eq: 7.12} with $\sfrac{\max\{\mathfrak{C}_1, \mathfrak{C}_1^2\}}{\sqrt{\tau}}<1$) and use \eqref{eq: 7.12} to find
	\begin{equation*}
	\|\*x_T(t)-\overline{\*x}\|\leqslant \mathfrak{C}_1 \|\*x_T(t_1)-\overline{\*x}\|\leqslant \left(\frac{\mathfrak{C}_1^2}{\sqrt{\tau}}\right)^{n+1}\left\|\*x^0-\overline{\*x}\right\|,
	\end{equation*}
	which holds for all $t\in[t_1,T]$. Clearly, as $t_1<(n+1)\tau$, the above estimate also holds for all $t\in[(n+1)\tau,T]$. This completes the proof of \eqref{eq: induction}.
	\smallskip
	\item[Step 3).]\textbf{Conclusion.} Suppose $t\in[\tau+1,T]$ is arbitrary and fixed. Set $n(t):=\left\lfloor\frac{t}{\tau+1}\right\rfloor$; note that $n(t)\geqslant1$, $t\geqslant n(t)\tau$ and $n(t)\leqslant\sfrac{T}{\tau}$. We may then apply \eqref{eq: induction} to find that
	\begin{equation*}
	\|\*x_T(t)-\overline{\*x}\|\leqslant \left(\frac{\mathfrak{C}_1^2}{\sqrt{\tau}}\right)^{n(t)}\left\|\*x^{0}-\overline{\*x}\right\|.
	\end{equation*}
	Since $\tau>\mathfrak{C}_1^4$ and $n(t)\geqslant\frac{t}{\tau+1}-1$, from the above inequality we infer
	\begin{align*}
	\|\*x_T(t)-\overline{\*x}\|&\leqslant\exp\left(-n(t)\log\left(\frac{\sqrt{\tau}}{\mathfrak{C}_1^2}\right)\right)\left\|\*x^{0}-\overline{\*x}\right\|\\
	&\leqslant\frac{\sqrt{\tau}}{\mathfrak{C}_1^2}
	\exp\left(-\frac{\log\left(\frac{\sqrt{\tau}}{\mathfrak{C}_1^2}\right)}{\tau+1}t\right)\left\|\*x^{0}-\overline{\*x}\right\|.
	\end{align*}
	The desired exponential stability estimate for $\|\*x_T(t)-\overline{\*x}\|$ thus also holds for $t\in[\tau+1,T]$, with $\mu:=\frac{\log\left(\frac{\sqrt{\tau}}{\mathfrak{C}_1^2}\right)}{\tau+1}>0$ and $\mathfrak{C}:=\frac{\sqrt{\tau}}{\mathfrak{C}_1^2}\left\|\*x^{0}-\overline{\*x}\right\|$. Note that, arguing as in the previous case, 
	\begin{equation*}
	\mathscr{E}(\*x_T(t))\lesssim \frac{1}{N}\sum_{i=1}^N \big\|P\*x_{T,i}(t)-\vec{y}_i\big\|^\alpha\lesssim \mathfrak{C}^\alpha \exp(-\alpha\mu t).
	\end{equation*}
	\end{enumerate}
	This concludes the proof of the first part.
	\smallskip
	
	\noindent
	\textbf{Part 2:} \emph{Stability estimate for the parameters.}
	The stability estimate for the parameters in the setting of \eqref{eq: sigma.outside.i} closely follows the proof presented in \citep{esteve2020turnpike} -- we give a sketch of the main ideas. We henceforth interchange between the notation $u:=[w,b]$ and $[w,b]$ to ease the reading.
	
	Fix an arbitrary $t\in[0,T)$ and $0<h\ll1$, so that $t+2h^2+2h\in[0,T]$, and set
	\begin{equation*}
	u^\aux(s):=\begin{dcases}
	u_T(s) &\text{ for } s\in[0,t]\\
	\frac12u_T\left(t+\frac{s-t}{2}\right) &\text{ for } s\in(t,t+2h^2]\\
	\frac{h+2}{2}u_T\left(\left(\frac{h+2}{2}\right)s-\frac{h+2}{2}(t+2h^2)+t+h^2\right) &\text{ for } s\in(t+2h^2,\\
	&\hspace{1.5cm}t+2h^2+2h]\\
	u_T(s) &\text{ for } s\in(t+2h^2+2h,\\
	&\hspace{3.25cm}T].
	\end{dcases}
	\end{equation*}
	From \Cref{lem: scaling} that the solution $\*x^\aux$ to \eqref{eq: sigma.outside.i} associated to the above pair is precisely
	\begin{equation*}
	\*x^\aux(s):=\begin{dcases}
	\*x_T(s) &\text{ for } s\in[0,t]\\
	\*x_T\left(t+\frac{s-t}{2}\right) &\text{ for } s\in(t,t+2h^2]\\
	\*x_T\left(\left(\frac{h+2}{2}\right)s-\frac{h+2}{2}(t+2h^2)+t+h^2\right) &\text{ for } s\in(t+2h^2,\\
	&\hspace{1.5cm}t+2h^2+2h]\\
	\*x_T(s) &\text{ for } s\in(t+2h^2+2h,T].
	\end{dcases}
	\end{equation*}
	This specific construction is in particular done as to ensure that $\*x^\aux(T)=\*x_T(T)$ and so $\mathscr{E}(\*x^\aux(T))=\mathscr{E}(\*x_T(T))$. Now by several straightforward computations (which may be found in \citep{esteve2020turnpike}) we deduce that
	\begin{align*}
	J_T(u^\aux) &= \mathscr{E}(\*x_T(T)) + \frac{1}{N}\int_0^T\|\*x_T(s)-\overline{\*x}\|^2\diff s + \frac{1}{N}\int_t^{t+h^2}\|\*x_T(s)-\overline{\*x}\|^2\diff s \\
	&\qquad + \frac{1}{N}\left(\frac{2}{h+2}-1\right)\int_{t+h^2}^{t+2h^2+2h} \|\*x_T(s)-\overline{\*x}\|^2\diff s\\
	&\qquad+ \lambda \int_0^T \|u_T(s)\|^2\diff s - \frac{\lambda}{2}\int_t^{t+h} \|u_T(s)\|^2\diff s + \lambda \frac{h}{2} \int_{t+h^2}^{t+h^2+2h} \|u_T(s)\|^2\diff s\\
	&\leqslant \mathscr{E}(\*x_T(T)) + \frac{1}{N}\int_0^T \|\*x_T(s)-\overline{\*x}\|^2\diff s + \frac{1}{N}\int_t^{t+h} \|\*x_T(s)-\overline{\*x}\|^2 \diff s \\
	&\qquad+ \lambda \int_0^T \|u_T(s)\|^2\diff s - \frac{\lambda}{2}\int_t^{t+h} \|u_T(s)\|^2\diff s + \lambda \frac{h}{2} \int_{t+h^2}^{t+h^2+2h} \|u_T(s)\|^2\diff s.
	\end{align*}
	The above identity combined with the optimality inequality $J_T(u_T)\leqslant J_T(u^\aux)$ leads us to 
	\begin{equation*}
	\frac{\lambda}{2} \int_{t}^{t+h} \|u_T(s)\|^2\diff s \leqslant \frac{1}{N}\int_{t}^{t+h} \|\*x_T(s)-\overline{\*x}\|^2 \diff s + \lambda \frac{h}{2}  \int_{t+h^2}^{t+h^2+2h} \|u_T(s)\|^2\diff s.
	\end{equation*}
	Using the exponential stability estimate for $\|\*x_T(\cdot)-\overline{\*x}\|$, we may find 
	\begin{align*}
	\frac{\lambda}{2}\frac{1}{h} \int_{t}^{t+h} \|u_T(s)\|^2\diff s &\leqslant \frac{1}{N}\frac{1}{h}\int_{t}^{t+h} \|\*x_T(s)-\overline{\*x}\|^2 \diff s + \frac{\lambda}{2}  \int_{t+h^2}^{t+h^2+2h} \|u_T(s)\|^2\diff s \\
	&\leqslant \frac{\mathfrak{C}}{N}\frac{1}{h} \int_t^{t+h} e^{-2\mu s}\diff s + \frac{\lambda}{2}  \int_{t+h^2}^{t+h^2+2h} \|u_T(s)\|^2\diff s\\
	&\leqslant \frac{\mathfrak{C}}{N} e^{-2\mu t} + \frac{\lambda}{2}  \int_{t+h^2}^{t+h^2+2h} \|u_T(s)\|^2\diff s.
	\end{align*}
	By using the Lebesgue dominated convergence theorem (applied to the integrable function $s\mapsto\|u_T(s)\|^2\one_{(t+h^2,t+h^2+2h)}(s)$) the second integral in the estimate just above goes to $0$ when $h\searrow0$. Hence, by applying the Lebesgue differentiation theorem in the estimate just above, we deduce that
	\begin{equation*}
	\Big\|[w_T(t),b_T(t)]\Big\|^2=\|u_T(t)\|^2 = \lim_{h\searrow0}\,\frac{1}{h}\int_t^{t+h} \|u_T(s)\|^2\diff s \leqslant \frac{\mathfrak{C}}{2N\lambda} e^{-2\mu t}
	\end{equation*}
	for a.e. $t\in[0,T]$, as desired.
	\end{proof}
	
	\subsection{Proof of \Cref{thm: turnpike.BV}}
	
	\begin{proof}[Proof of \Cref{thm: turnpike.BV}]
	We will interchange the notations $[w,b]$ and $u:=[w,b]$ for simplicity, and we split the proof in two parts. We henceforth set $r:=\left\|\*x^0-\overline{\*x}\right\|$.
	\smallskip
	
	\noindent
	\textbf{Part 1:} \emph{Uniform estimates}. We shall first establish uniform-in-$T$ estimates -- we find some $\mathfrak{C}>0$ independent of $T$ such that
	\begin{equation} \label{eq: bv.est}
	\left\|\*x_T(t)-\overline{\*x}\right\| + \left\|\*x_T-\overline{\*x}\right\|_{L^2(0,T;\R^{d_x})} + \Big\|[w_T,b_T]\Big\|_{\BV(0,T;\R^{d_u})} \leqslant \mathfrak{C}
	\end{equation} 
	whenever $T\geqslant1$ and for $t\in[0,T]$. By virtue of the controllability assumption, there exist parameters $u^\dagger\in C^0([0,1];\R^{d_u})\cap\BV([0,1];\R^{d_u})$ such that the corresponding solution $\*x^\dagger(\cdot)$ to \eqref{eq: sigma.outside.i} on $(0,1)$ satisfies $\*x^\dagger(1)=\overline{\*x}$. 
	By integrating the equation satisfied by $\*x^\dagger(\cdot)$, we see that for $t\in[0,1]$, 
	\begin{align*}
	\left\|\*x^\dagger(t)-\overline{\*x}\right\|\leqslant\left\|\*x^{0}-\overline{\*x}\right\|+c(\sigma)\Big(&\int_{0}^t\left\|\*w^\dagger(s)\right\|\left\|\*x^\dagger(s)-\overline{\*x}\right\|\diff s\\
	&+\|\overline{\*x}\|\int_{0}^t \left\|\*w^\dagger(s)\right\|\diff s+ \int_{0}^t\left\|\*b^\dagger(s)\right\|\diff s\Big)
	\end{align*}
	where $c(\sigma)>0$ is the Lipschitz constant of $\sigma$. By virtue of Gr\"onwall's inequality and \eqref{eq: linear.control.cost.2}, we deduce that
	\begin{equation} \label{eq: B.3}
	N^{-\sfrac{1}{2}}\left\|\*x^\dagger(t)-\overline{\*x}\right\|\leqslant C_1\left\|\*x^{0}-\overline{\*x}\right\|\exp\Big(C_1\left\|\*x^{0}-\overline{\*x}\right\|\Big)
	\end{equation}
	for $t\in[0,1]$ and for some constant $C_1=C_1(\sigma,\|\overline{\*x}\|,r)>0$ independent of $T$ (as well as and $N$, and only depends on $\*x^0$ via $r$).	
	We now set (recall that $T\geqslant1$)
	\begin{equation*}
	u^\aux(t):=\begin{dcases}u^\dagger(t) &\text{ in } (0,1)\\
	0 &\text{ in } (1, T),
	\end{dcases}
	\end{equation*}
	and we denote by $\*x^\aux(\cdot)$ the corresponding solution to \eqref{eq: sigma.outside.i}. One notes that $u^\aux\in L^2(0,T;\R^{d_u})\cap\BV(0,T;\R^{d_u})$, as the jump at $t=1$ only accounts to a Dirac mass. 
	Moreover, $\*x^\aux(t)=\overline{\*x}$ for $t\in[1,T]$ and thus also $P\*x^\aux_i(T)=\vec{y}_i$ for $i\in[N]$, which implies $\mathscr{E}(\*x^\aux(T))=0$. 
		By using the optimality inequality $J_{T}\left(u_T\right)\leqslant J_{T}\left(u^\aux\right)$ and the fact that $\mathscr{E}(\*x^\aux(T))=0$ and $\mathscr{E}\geqslant0$, we thus deduce that
		\begin{align} \label{eq: B.4}
	&N^{-1}\Big\|\*x_T-\overline{\*x}\Big\|^2_{L^2(0,T;\R^{d_x})}+\lambda\Big\|u_T\Big\|_{\BV(0,T;\R^{d_u})}^2\nonumber\\
	\leqslant\, &N^{-1}\Big\|\*x^\aux-\overline{\*x}\Big\|^2_{L^2(0,T;\R^{d_x})}+\lambda\Big\|u^\aux\Big\|_{\BV(0,T;\R^{d_u})}^2\nonumber\\
	\leqslant\, &N^{-1}\Big\|\*x^\dagger-\overline{\*x}\Big\|^2_{L^2(0,1;\R^{d_x})}+2\lambda\left\|u^\dagger\right\|_{\BV(0,1;\R^{d_u})}^2+2\lambda(d^2+d)\sum_{j=1}^{d^2+d}\left|u^\dagger_j(1)\right|^2,
	\end{align}
	where we used the vectorized form of $u^\dagger$ in the second component of the $\BV$--norm whilst maintaining the same notation (recall that $d_u:=d\times(d+1)$).
	Using \eqref{eq: B.3} and \eqref{eq: linear.control.cost.2}, we find
	\begin{align} \label{eq: B.5}
	N^{-1}\Big\|\*x^\dagger-\overline{\*x}\Big\|^2_{L^2(0,1;\R^{d_x})}+&2\lambda\left\|u^\dagger\right\|_{\BV(0,1;\R^{d_u})}^2+2\lambda(d^2+d)\sum_{j=1}^{d^2+d}\left|u^\dagger_j(1)\right|^2\nonumber\\
	&\hspace{1.5cm}\leqslant \underbrace{C_2 \left\|\*x^{0}-\overline{\*x}\right\|^2\exp\Big(2C_1\left\|\*x^0-\overline{\*x}\right\|\Big)}_{:=\mathfrak{C}_0^2}
	\end{align}
	for some constant $C_2=C_2(\sigma,\|\overline{\*x}\|,r, d,\max\{1,\lambda\})>0$.
	 Combining \eqref{eq: B.4} and \eqref{eq: B.5}, and recalling \Cref{lem: turnpike.1}, we may conclude that \eqref{eq: bv.est} holds with
 	\begin{equation} \label{eq: constant.desired}
  	\mathfrak{C}:= 
	C_3\max\left\{1,\frac{\mathfrak{C}_0}{\lambda}\right\}\exp\left(C_4\frac{\mathfrak{C}_0}{\lambda}\right) \left\|\*x^0-\overline{\*x}\right\|
	\end{equation}
	for some constants $C_3=C_3(\sigma, \|\overline{\*x}\|,r,N,\lambda)>0$ and $C_4=C_4(\sigma,\|\overline{\*x}\|,N)>0$. 
	\smallskip
	
	\noindent \textbf{Part 2:} \emph{Conclusion.} We note that the desired stability estimates thus follow	 from \eqref{eq: bv.est}, and, since  $\overline{\*x}_i\in P^{-1}(\{\vec{y}_i\})$ for $i\in[N]$, we also have
	\begin{equation*}
	\mathscr{E}(\*x_T(t))\leqslant N^{-1}\sum_{i=1}^N \big\|P\*x_{T,i}(t)-\vec{y}_i\big\|^\alpha\leqslant N^{-1}\|\*x_T(t)-\overline{\*x}\|^\alpha.
	\end{equation*}
	We conclude the proof by noting that the convergence of averages follows by dividing both sides in \eqref{eq: bv.est} by $T$, using the estimate just above for the training error term, and letting $T\longrightarrow\infty$.
	\end{proof}
	
	\subsection{Proof of \Cref{prop_lower_bound}}
	
	The proof of \Cref{prop_lower_bound} is a straightforward Gr\"onwall argument. We sketch it for completeness.
	
	\begin{proof}[Proof of \Cref{prop_lower_bound}]
	    For simplicity of presentation but without any loss of generality, we will henceforth concentrate on system \eqref{eq: sigma.inside.i}.
	For any $t\in [0,T]$, $i\in [N]$ and $j\in[N]$, we have
	\begin{equation*}
	    \*x_i(t)-\*x_j(t) = \*x_{i}^0-\*x_{j}^0+\int_0^t w(\tau) \Big(\sigma(\*x_i(\tau))-\sigma(\*x_j(\tau))\Big) \diff \tau.
	\end{equation*}
	Using the Lipschitz character of $\sigma$, we get
	\begin{align*} 
        \left\|\*x_i(t)-\*x_j(t)\right\|
        &\leqslant\left\|\*x_{i}^0-\*x_{j}^0\right\|+\int_0^t \left\|w(\tau)\right\| \left\|\sigma\left(\*x_i(\tau)\right)-\sigma\left(\*x_j(\tau)\right)\right\| \diff \tau \nonumber\\
        &\leqslant\left\|\*x_{i}^0-\*x_{j}^0\right\|+C(\sigma) \int_0^t \left\|w(\tau)\right\| \left\|\*x_i(\tau)-\*x_j(\tau)\right\| \diff \tau.
    \end{align*}
    We apply the Gr\"onwall inequality with the effect of
    \begin{equation*}
        \left\|\*x_i(t)-\*x_j(t)\right\|\leqslant \exp\left(C(\sigma) \int_0^t \left\|w(\tau)\right\| \diff \tau \right)\left\|\*x_{i}^0-\*x_{j}^0\right\|.
    \end{equation*}
    We evaluate the above expression at final time $t=T$ to obtain
    \begin{equation*}
        \left\|\*x_{i}^1-\*x_{j}^1\right\|\leqslant \exp\left(C(\sigma) \int_0^T \left\|w(\tau)\right\|\diff \tau\right)\left\|\*x_{i}^0-\*x_{j}^0\right\|,
    \end{equation*}
    for some $\*x_{i}^1\in P^{-1}\left(\{\vec{y}_i\}\right)$ and $\*x_{j}^1\in P^{-1}\left(\{\vec{y}_j\}\right)$, whence
    \begin{equation*}
        \exp\left(C(\sigma) \int_0^T \left\|w(\tau)\right\| \diff \tau \right)\geqslant \frac{\left\|\*x_{i}^1-\*x_{j}^1\right\|}{\left\|\*x_{i}^0-\*x_{j}^0\right\|}.
    \end{equation*}
    By taking the $\log$ on both sides we obtain \eqref{controllability_L1_lower_bound}.
	\end{proof}	
	
	\subsection{Proof of \Cref{prop_Nleqd+1}}
	
	The following short functional analysis lemma will be of use in the proof of \Cref{prop_Nleqd+1}. We omit the proof, which follows by using the open mapping theorem (see e.g. \cite[Theorem 2.6, pp. 35]{brezis2010functional}).
	
	\begin{lemma} \label{lemma_invmod}
        Let $\mathscr{H}_1$ and $\mathscr{H}_2$ be two real Hilbert spaces. 
        Let
        \begin{equation*}
            \Lambda: \mathscr{H}_1\longrightarrow \mathscr{H}_2
        \end{equation*}
        be a linear, bounded and surjective operator. 
        Then
        \begin{align*}
            \Gamma: \mathscr{H}_2 &\longrightarrow \mathscr{H}_1 \\
            y &\longmapsto \argmin_{x\in \Lambda^{-1}(\{y\})}\left\|x\right\|_{\mathscr{H}_1}^2
        \end{align*}
        is linear and bounded.
    \end{lemma}

 \begin{proof}[Proof of \Cref{prop_Nleqd+1}]
      
        Inspired by the techniques in \citep{coron2004global, pighin2018controllability}, we define the continuous arc
        \begin{align*}
        \gamma: [0,1]&\longrightarrow \R^{d_x} \\
        s&\longmapsto (1-s)\*x^0+s\*x^1.
        \end{align*}
        By assumption,
        \begin{equation*}
	        \Big\{\sigma\left(\*x_{1}^1\right),\dots,\sigma\left(\*x_{i}^1\right),\dots,\sigma\left(\*x_{N}^1\right)\Big\}
	    \end{equation*}
	    is a linearly independent system of vectors in $\R^d$ for any $s \in [0, 1]$. 
	    Thus, by using the continuity of $\gamma$, there exists an $\eta>0$, such that whenever $\left\|\*x^{1}-\*x^{0}\right\|\leqslant \eta$,
        \begin{equation}\label{svectors_system_2}
	        \Big\{\sigma\left(\gamma_1(s)\right),\dots,\sigma\left(\gamma_i(s)\right),\dots,\sigma\left(\gamma_N(s)\right)\Big\}
	    \end{equation}
	    is also a system of linearly independent vectors in $\R^{d}$ for any $s\in [0,1]$. 
	    Following the framework of \Cref{lemma_invmod}, for any $s\in \left[0,1\right]$, define
	   \begin{align*}
	   \Lambda_s:\R^{d\times d}&\longrightarrow \R^{d_x}\\
	   w&\longmapsto w\sigma\left(\gamma\left(s\right)\right).
	   \end{align*}
	   By the linear independence of the system of vectors \eqref{svectors_system_2}, $\Lambda_s$ is surjective for any $s\in [0,1]$. Hence, using \Cref{lemma_invmod}, we see that
	    \begin{align*}
	    \Gamma_s: \R^{d_x}&\longrightarrow \R^{d\times d}\\
	    y&\longmapsto \argmin_{w\in \Lambda_s^{-1}(\{y\})}\left\|w\right\|,
	    \end{align*}
	    is a linear and bounded operator for any $s\in [0,1]$, and, since \eqref{svectors_system_2} is independent and the arc $\gamma$ is continuous, 
	    $\{\Gamma_s\}_{s\in[0,1]}$ is uniformly bounded in operator norm:
	    \begin{equation} \label{eq: uniform.bound.gammat}
	    \left\|\Gamma_s\right\|_{\mathscr{L}\left(\R^{d_x}; \R^{d\times d}\right)}\leqslant C
	    \end{equation}
	    for some $C>0$ independent of $T>0$.
       	Now, for $t\in \left[0,T\right]$, set
        \begin{equation}\label{def_wGammat}
	        w(t)\coloneqq \Gamma_{s_t}\left(\dfrac{\*x^{1}-\*x^{0}}{T}\right),
	    \end{equation}
	   with $s_t\coloneqq \frac{t}{T}$. Note that for any $t\in \left[0,T\right]$, the vector $w(t)\in \R^{d\times d}$ solves the linear system of equations
    	\begin{equation*}\label{eq: sigma.inside.i_linearsystem}
    	w(t)\sigma\left(\*x_i(t)\right) = \dot{\*x}_i(t) \hspace{0.25cm} \text{ for } i \in [N],
    	\end{equation*}
    	where
    	\begin{equation*}
	\*x(t)\coloneqq \gamma\left(\dfrac{t}{T}\right)=\left(1-\dfrac{t}{T}\right)\*x^0+\dfrac{t}{T}\*x^1.
	\end{equation*}
    	Hence, $\*x(t)$ solves
    	\begin{equation*} \label{eq: sigma.inside.i_Nleqd+1_2}
    	\begin{dcases}
    	\dot{\*x_i}(t) = \*w(t) \sigma(\*x_i(t)) &\text{ for } t \in \left(0,T\right) \\
    	\*x_i(0) = \gamma(0)=\*x_{i}^0 \\
	\*x_i\left(T\right) = \gamma(1)=\*x_{i}^1,
    	\end{dcases}
    	\end{equation*}
    	for any $i\in [N]$. This thus demonstrates the existence of a control $w$ steering the stacked dynamics from $\*x^0$ to $\*x^1$ in time $T$.
	\smallskip
	
	Let us conclude by showing that $w$ satisfies the stated estimate. By the definition of $w$ in \eqref{def_wGammat} as well as \eqref{eq: uniform.bound.gammat}, for any $t\in \left[0,T\right]$ we have
    	\begin{equation*}
    	    \left\|w(t)\right\|=\left\|\Gamma_t\left(\dfrac{\*x^{1}-\*x^{0}}{T}\right)\right\| \leqslant \dfrac{C}{T}\left\|\*x^{1}-\*x^{0}\right\|,
    	\end{equation*}
    	as desired.
    \end{proof}
    
    \begin{remark}[Removing the smallness assumption]
    One could perhaps adapt the argument in the proof of \Cref{prop_Nleqd+1} (given just below) to obtain a global result, assuming the existence of a continuous arc $\gamma$ linking $\*x^0$ and $\*x^1$ such that
            \begin{equation*}\label{svectors_system_3}
	        \Big\{\sigma\left(\gamma_1(s)\right),\dots,\sigma\left(\gamma_i(s)\right),\dots,\sigma\left(\gamma_N(s)\right)\Big\}
    	    \end{equation*}
    	    is a system of linearly independent vectors in $\R^d$ for any $s\in [0,1]$. 
	    Problems arise however whenever this condition is not satisfied. 
	    In any case, in view of the uniqueness results for ODEs and \Cref{prop_lower_bound}, we have to assume that $\*x_{i}^0\neq \*x_{j}^0$ and $\*x_{i}^1\neq \*x_{j}^1$, for $i\neq j$.
    \end{remark}

	\section{Concluding remarks}
	
\noindent 
In this work, we have addressed the impact of the final time horizon $T$ in general learning problems for neural ODEs. 
	
	\begin{itemize}
	\item In the empirical risk minimization problem with a Tikhonov ($L^2$ or $H^1$) parameter regularization, we concluded via \Cref{thm: no.running} -- \Cref{thm: thm.classification.lambda} that when $T$ is large enough, the obtained optimal/trained parameters for neural ODEs are such that the corresponding trajectories reach zero training error with a quantitative rate (thus, stipulate an approximation property of the trained model with respect to $T$), whilst doing so with the least oscillations possible.
	In the associated discrete-time, residual neural network setting, this result indicates that adding more layers before training would guarantee the optimal trajectories approach the zero training error regime, but do so without overfitting.
	In more practical terms, to ensure that the global minimizer is near zero training error, while training, one could systematically increase the time horizon $T$ whilst keeping the regularization parameter $\lambda>0$ fixed. Moreover, roughly the same conclusions hold when $T$ is fixed and $\lambda$ is rendered small, thereby linking our  insights with the literature on the regularization path limit for various machine learning models.
	\smallskip
	
	\item To obtain better quantitative estimates on the time horizon (and thus, number of layers when the time-step is fixed, e.g. in ResNets) required to be $\varepsilon$--close to the zero training error regime, for a given tolerance $\varepsilon>0$, we introduced a minimization problem wherein we added a tracking term which regularizes the state trajectories over the entire time horizon. 
	In \Cref{thm: turnpike.P}, we show that the training error and the optimal parameters are at most of the order $\mathcal{O}\left(e^{-\mu t}\right)$ for all $t\in[0,T]$.
	This result, along with numerical experiments, demonstrates a strong approximation rate of the trained neural ODE flow (which ought to be compared with universal approximation results, in which, a key caveat is that there is no scalable method to compute the theoretically guaranteed parameters), with parameters which are exponentially small, and could thus stipulate that the flow would tend to oscillate little.
	Moreover, the exponential decay estimate also ensures that $T$ need not be chosen too large to render the training error small.
	\end{itemize}
	
	\subsection{Outlook}
	
	We present a list of questions and topics which would be complementary to our work.
	\smallskip
	
	\noindent\textbf{Generalization bounds.} 
	To complement our analytical study on the long time horizon/large layer regime, it would be of interest to provide generalization error bounds for the limiting, least $L^2$--norm parameters in the interpolation regime obtained in \Cref{thm: no.running}, via, for instance, commonly used metrics such as the VC dimension \citep{vapnik2013} or Rademacher complexity \citep{bartlett2002}.
	Such studies are, to the best of our knowledge, not done in the context of models such as neural ODEs.
	\smallskip
	
	\noindent \textbf{Stability estimate for \eqref{eq: time.dep.func.wanted}.}
	We provided a proof of the exponential stability estimate of the training error and optimal parameters in the context of $\ell^2$--like losses, and without regularizing the output $P\*x_i(t)$ but rather the features $\*x_i(t)$ over all time/layer $t \in [0,T]$. 
	We could stipulate that, whenever $P$ is Lipschitz (and possibly real analytic) and such that the training error attains its minimum (e.g. when $P$ is a matrix, or a matrix composed with a sigmoid truncated by some cut-off function), the exponential stability result could hold by making use of the Łojasiewicz inequality: for a compact set $K\subset\R^d$ and two continuous, sub-analytic functions $\mathfrak{g}, \mathfrak{h}: K\longrightarrow\R$, if $\mathfrak{g}^{-1}(\{0\})\subset\mathfrak{h}^{-1}(\{0\})$ then
	\begin{equation*}
	\|\mathfrak{h}(x)\|^\alpha\leqslant c \|\mathfrak{g}(x)\| \hspace{1cm} \text{ for all } x\in K, 
	\end{equation*}
	holds for some $c>0$ and $\alpha\in\N$
	(see \citep{lojasiewicz1961probleme} and also \citep[Theorem 6.4]{bierstone1988semianalytic}). This however remains an open problem. 
	On the other hand, addressing analytically the (exponential) stability stipulated by the numerical experiments presented herein for non $\ell^2$--losses such as cross-entropy also remains open.
	\smallskip
	
	\noindent \textbf{A pre-training algorithm.}
	Note that the insight from \Cref{thm: turnpike.P} and the experiments performed for \eqref{eq: time.dep.func.wanted} could motivate the following algorithm, which trains a shallower network (i.e., a neural ODE with a shorter time horizon) at first, and then successively increases this time horizon in an iterative manner.
	\smallskip
	
		\begin{algorithm}[H]
\SetAlgoLined
\KwResult{$T$, $[w_T,b_T]$}
 fix $\varepsilon>0$, $T_\circ>0$\;
 $j\leftarrow1$\;
Find $[w_{jT_\circ}, b_{jT_\circ}]$ and $\*x_{jT_\circ}$ by solving
 \begin{equation} \label{eq: cond.1}
 \inf_{[w,b]\in H^k((j-1)T_\circ,jT_\circ;\R^{d_u})} \int_{(j-1)T_\circ}^{jT_\circ}\mathscr{E}(\*x(t))\diff t + \Big\|[w,b]\Big\|^2_{H^k((j-1)T_\circ, jT_\circ;\R^{d_u})}
 \end{equation}
subject to
\begin{equation} \label{eq: cond.2}
\begin{dcases}
\dot{\*x}(t) = \mathfrak{f}\Big([w(t),b(t)], \*x(t)\Big) &\text{ in } ((j-1)T_\circ,jT_\circ)\\
\*x((j-1)T_\circ)=\*x^{(j-1)T_\circ}.
\end{dcases}
\end{equation}
\BlankLine
$[w_T,b_T]_{|_{((j-1)T_\circ,jT_\circ)}}\leftarrow[w_{jT_\circ}, b_{jT_\circ}]$\;
$\*x^j \leftarrow \*x_{jT_\circ}(jT_\circ)$\;
$T\leftarrow jT_\circ$\;
 \While{$\mathscr{E}(\*x(jT_\circ))>\varepsilon$}{
  $j\leftarrow j+1$\;
  find $[w_{jT_\circ}, b_{jT_\circ}]$ and $\*x_{jT_\circ}$ by solving \eqref{eq: cond.1} -- \eqref{eq: cond.2}\;
$[w_T,b_T]_{|_{((j-1)T_\circ,jT_\circ)}}\leftarrow[w_{jT_\circ}, b_{jT_\circ}]$\;
$\*x^{jT_\circ} \leftarrow \*x_{jT_\circ}(jT_\circ)$\;
$T\leftarrow jT_\circ$
 }
 \caption{A pre-training algorithm.}
\end{algorithm}
\smallskip

\noindent	
One of the most distinguished characteristics of \eqref{eq: time.dep.func} is that the time-horizon $T$ needed to get $\varepsilon$--close to any given target is in fact implicitly defined in the cost functional. 
	At the level of ResNets, this means that the required number of layers needed to fit the data up to $\varepsilon$-error is given by the cost itself. 
	The goal of the above algorithm is to take advantage of this artifact, and represents a \emph{greedy algorithm} which uses only the number of layers strictly needed, thus avoiding unnecessary ones.
	We leave open the possible numerical analysis of the above algorithm (see \citep[Chapter 15]{goodfellow2016deep} for some related literature on pre-training algorithms).
	\smallskip
		
	\noindent \textbf{More general models.} It would be of interest to directly investigate the appearance of the phenomena presented herein to neural ODEs with state of the art configurations, including convolutional, batch-normalization and max-pooling layers -- we have solely focused our theoretical analysis to the basic settings. 
	Further extensions could include the study of ResNet variants such as MomentumNets \citep{sander2021momentum}, which are second order ODEs wherein $\dot{\*x}$ is seen as a damping term, and mean field variants \citep{weinan2019mean}. 
	
	\smallskip
	
	\noindent \textbf{Unsupervised learning.}
	As discussed in the introduction, the neural ODE representation of deep supervised learning, due to the invertibility of the neural ODE flow, has seen fruitful applications in generative modeling via continuous normalizing flows (see \citep{ruthotto2021introduction}). 
	In generative modeling and unsupervised learning, one aims to infer the probability distribution of the inputs $\{\vec{x}_i\}_{i\in[N]}$ rather than, for instance, draw a decision boundary as in classification tasks via supervised learning.
	It would be of interest, in view of the existing applications, to investigate the potential use of the results presented in this work to the context of unsupervised learning.
	\smallskip

    	\subsubsection*{Acknowledgments} B.G. acknowledges Daniel Tenbrinck and Lukas Pflug (FAU Erlangen-N\"urnberg) for discussions on the foundations of neural networks and non-local equations.
	
	\smallskip
	
	\noindent
	{\small{\textbf{Funding:}}}
		{\small{B.G. and E.Z. have received funding from the European Union's Horizon 2020 research and innovation programme under the Marie Sklodowska-Curie grant agreement No.765579-ConFlex.
		D.P., C.E. and E.Z. have received funding from the European Research Council (ERC) under the European Union’s Horizon 2020 research and innovation programme (grant agreement NO. 694126-DyCon).
The work of E. Z. has been supported by the Alexander von Humboldt-Professorship program, the Transregio 154 Project ‘‘Mathematical Modelling, Simulation and Optimization Using the Example of Gas Networks’’ of the German DFG, grant MTM2017-92996-C2-1-R COSNET of MINECO (Spain) and by the Air Force Office of Scientific Research (AFOSR) under Award NO. FA9550-18-1-0242.}}

	\appendix

	\section{Auxiliary proofs}
	
	\begin{proof}[Proof of \Cref{lem: compact.flow}]
	The proof is mostly identical in for all of the three items, the difference being an underlying compact embedding.
	\smallskip
	
	\noindent
	\textbf{Proof of \textit{(i)}}.
	Let $\left\{ [w_n, b_n]\right\}_{n=1}^{\infty} \subset L^2(0,T; \R^{d_u})$ be a bounded sequence in $L^2(0,T;\R^{d_u})$.
	By the Banach-Alaoglu theorem, there exists a pair $\left[w^\dagger,b^\dagger\right] \in L^2(0,T;\R^{d_u})$ such that, along some subsequence,
	\begin{align*}
	[w_n,b_n] \xrightharpoonup[n\longrightarrow\infty]{} \left[w^\dagger,b^\dagger\right] \hspace{1cm} &\text{ weakly in } L^2(0,T;\R^{d_u}).
	\end{align*}
	Of course, the same convergences thence hold for $\*w_n := \text{diag}_N(w_n)$ to $\*w^\dagger := \text{diag}_N(w^\dagger)$, as well as $\*b_n := [b_n, \ldots, b_n]$ to $\*b^\dagger:= [b^\dagger, \ldots, b^\dagger]$.
	Let $\*x^\dagger \in C^0([0,T];\R^{d_x})$ be the unique solution to \eqref{eq: sigma.inside.i} associated to $[w^\dagger,b^\dagger]$ and the initial datum $\*x^0$.
	Let us prove that 
	\begin{equation} \label{eq: A1}
	\*x_n \xrightarrow[n\longrightarrow\infty]{} \*x^\dagger \hspace{1cm} \text{ strongly in } C^0([0,T];\R^{d_x})
	\end{equation}
	along the aforementioned subsequence.
	Take an arbitrary $t\in [0,T]$. Note that
	\begin{align*}
	\*x_{n}(t) - \*x^\dagger(t) &= \int_0^t \Big[ \*w_n(\tau) \sigma(\*x_n(\tau))+\*b_n(\tau)\Big] \diff \tau - \int_0^t \left[\*w^\dagger(\tau) \sigma\left(\*x^\dagger(\tau)\right)+\*b^\dagger(\tau)\right]\diff \tau \\
	&= \int_0^t \Big[\*w_n(\tau)\sigma(\*x_n(\tau))-\*w_n(\tau)\sigma\left(\*x^\dagger(\tau)\right)\Big] \diff \tau \\
	&\quad+ \int_0^t \Big[\*w_n(\tau)\sigma\left(\*x^\dagger(\tau)\right)-\*w^\dagger(\tau)\sigma\left(\*x^\dagger(\tau)\right) \Big]\diff \tau\\
	&\quad + \int_0^t \Big[\*b_n(\tau)-\*b^\dagger(\tau) \Big]\diff \tau.
	\end{align*}
	Hence, using the fact that $\sigma$ is globally Lipschitz with constant $c(\sigma)>0$,
	\begin{align*}
		\left\|\*x_{n}(t)-\*x^\dagger(t)\right\| &\leqslant \int_0^t \left \|\*w_n(\tau)\right\| \left\|\sigma\Big(\*x_n(\tau)\Big)-\sigma\left(\*x^\dagger(\tau)\right)\right\| \diff \tau \\
		&\quad+ \left\|\int_0^t \sigma\left(\*x^\dagger(\tau)\right) \left[\*w_{n}(\tau) - \*w^\dagger(\tau)\right] \diff \tau\right\| \\
		&\quad + \left\| \int_0^t \Big[\*b_n(\tau)-\*b^\dagger(\tau) \Big]\diff \tau \right\|\\
		&\leqslant c(\sigma) \int_0^t \left \|\*w_n(\tau)\right\| \left\|\*x_n(\tau)-\*x^\dagger(\tau)\right\| \diff \tau +  c_n,
	\end{align*}
	with 
	\begin{equation*}
	c_n\coloneqq \left\|\int_0^t \sigma\left(\*x^\dagger(\tau)\right) \left[\*w_{n}(\tau) - \*w^\dagger(\tau)\right] \diff \tau\right\|+\left\| \int_0^t \Big[\*b_n(\tau)-\*b^\dagger(\tau) \Big]\diff \tau \right\|.
	\end{equation*}
	Using Gr\"onwall's inequality, Cauchy-Schwarz, and the boundedness of the $L^2$--norm of $\{\*w_n\}_{n=1}^{\infty}$ by some constant $M>0$ independent of $t$, we thence obtain
	\begin{align*}
		\left\|\*x_{n}(t)-\*x^\dagger(t)\right\| &\leqslant c_n\exp\left(c(\sigma)\int_0^t \, \left \|\*w_{n}(\tau) \right\| \diff \tau\right) \\
		&\leqslant c_n \exp\left(c(\sigma) \sqrt{T} \, \left \|\*w_{n} \right\|_{L^2(0,T; \R^{d\times d\times N})}\right) \\
		&\leqslant c_n \exp\left(c(\sigma) \sqrt{T} M\right).
	\end{align*}
	As $c_n\longrightarrow 0$ along any subsequence as $n\to\infty$ by virtue of the weak convergences of $\{\*w_n\}_{n=1}^{\infty}$ to $\*w^\dagger$ and $\{\*b_n\}_{n=1}^{\infty}$ to $\*b^\dagger$, we deduce \eqref{eq: A1}. Hence, $\Phi_T$ sends bounded sequences in $L^2(0,T; \R^{d_u})$ into strongly convergent (sub)sequences in $C^0([0,T]; \R^{d_x})$ and is thus compact.
	\smallskip
	
	\noindent
	\textbf{Proof of \textit{(ii)}.} Let $\left\{ [w_n, b_n]\right\}_{n=1}^{\infty} \subset L^2(0,T;\R^{d_u})\cap\BV(0,T;\R^{d_u})$ be a bounded sequence in $L^2(0,T;\R^{d_u})\cap\BV(0,T;\R^{d_u})$. By the compactness of the embedding 
	\begin{equation*}
	\BV(0,T;\R^{d_u})\hookrightarrow L^1(0,T;\R^{d_u})
	\end{equation*}
	(see \citep[Theorem 3.23]{ambrosio2000functions}), there exists a pair $\left[w^\dagger,b^\dagger\right]\in\BV(0,T;\R^{d_u})$ such that, along some subsequence,
	\begin{equation*}
	[w_n,b_n] \xrightarrow[n\longrightarrow\infty]{} \left[w^\dagger,b^\dagger\right] \hspace{1cm} \text{ strongly in } L^1(0,T;\R^{d_u}).
	\end{equation*}
	Of course, the same convergences thence hold for $\*w_n := \text{diag}_N(w_n)$ to $\*w^\dagger := \text{diag}_N(w^\dagger)$, as well as $\*b_n := [b_n, \ldots, b_n]$ to $\*b^\dagger:= [b^\dagger, \ldots, b^\dagger]$.
	Let $\*x^\dagger \in C^0([0,T];\R^{d_x})$ be the unique solution to \eqref{eq: sigma.outside.i} associated to $[w^\dagger,b^\dagger]$ and the initial datum $\*x^0$. Let us prove \eqref{eq: A1}. Arguing as above, we see that
	\begin{align*}
		\left\|\*x_{n}(t)-\*x^\dagger(t)\right\| &\leqslant c(\sigma)\int_0^t\left\|\*w_n(\tau)-\*w^\dagger(\tau)\right\|\left\|\*x_n(\tau)\right\| \diff \tau \\
		&+c(\sigma)\int_0^t \left\|\*w^\dagger(\tau)\right\|\left\|\*x_n(\tau)-\*x^\dagger(\tau)\right\|\diff \tau \\
		&+c(\sigma)\int_0^t\left\|\*b_n(\tau)-\*b^\dagger(\tau)\right\|\diff\tau,
	\end{align*}
	where $c(\sigma)>0$ is the Lipschitz constant of $\sigma$,
	and thus, using Gr\"onwall's inequality, we obtain
	\begin{align*}
		\left\|\*x_{n}(t)-\*x^\dagger(t)\right\| &\leqslant c(\sigma)c_n\exp\left(c(\sigma)\int_0^t\left\|\*w^\dagger(\tau)\right\|\diff\tau\right),
	\end{align*}
	where
	\begin{equation*}
	c_n\coloneqq \int_0^t\left\|\*w_n(\tau)-\*w^\dagger(\tau)\right\|\left\|\*x_n(\tau)\right\| \diff \tau + \int_0^t\left\|\*b_n(\tau)-\*b^\dagger(\tau)\right\|\diff\tau.
	\end{equation*}
	Using the fact that $\*x_n$ is bounded uniformly in $n$ in $C^0([0,T];\R^{d_x})$ (this follows by applying a Gr\"onwall argument) and the strong $L^1$--convergences of the parameters, we conclude that $c_n\longrightarrow0$, whence \eqref{eq: A1} follows.
	
	The proof of \textit{(iii)} follows by arguing as in \textit{(i)} and \textit{(ii)}, with an intermediate use of the Rellich-Kondrachov compactness theorem to ensure strong convergence of the parameters in $L^2$ -- we omit the proof.
	This concludes the proof.
	\end{proof}
	
	\begin{proof}[Proof of \Cref{lem: turnpike.1}]
	We focus on proving \emph{(i)} and split the proof in two steps. The proof of \emph{(ii)} will follow by simply applying a Cauchy-Schwarz inequality at the conclusion of each of the two steps.
	\smallskip
	
	\noindent \textbf{Step 1.} 
	Let us first suppose that $t\leqslant 1$. 
	By integrating \eqref{eq: sigma.outside.i} on the interval $[0,t]\subset[0,1]$ and using the fact that $\sigma\in\Lip(\R)$ and $\sigma(0)=0$, it may be seen that
	\begin{equation*}
	\|\*x(t)\| \leqslant \left\|\*x^0\right\| + c(\sigma)\int_0^t \|\*w(s)\|\|\*x(s)\|\diff s + c(\sigma)\int_0^t \|\*b(s)\|\diff s
	\end{equation*}
	for $t\in[0,1]$, where $c(\sigma)>0$ denotes the Lipschitz constant of $\sigma$. By the Gr\"onwall inequality, we then have
	\begin{equation} \label{eq: B.2}
	\|\*x(t)\| \leqslant \underbrace{\left(\left\|\*x^0\right\|+c(\sigma)N^{\sfrac{1}{2}}\|b\|_{L^1(0,1;\R^d)} \right)\exp\left(c(\sigma)N^{\sfrac{1}{2}}\|w\|_{L^1(0,1;\R^{d\times d})}\right)}_{:=C\left(\big\|[w,b]\big\|_{L^1(0,1;\R^{d_u})}\right)}
	\end{equation}
	for $t\in[0,1]$. Using \eqref{eq: B.2}, it may be seen that
	\begin{align*}
	\|\*x(t)-\overline{\*x}\|&\leqslant \left\|\*x^0-\overline{\*x}\right\|+c(\sigma)N^{\sfrac{1}{2}}\left(\|\*x(t)\|\|w\|_{L^1(0,1;\R^{d\times d})}+\|b\|_{L^1(0,1;\R^d)} \right) \\
	&\leqslant \left\|\*x^0-\overline{\*x}\right\|+c(\sigma)N^{\sfrac{1}{2}}\max\left\{1,C\left(\Big\|[w,b]\Big\|_{L^1(0,1;\R^{d_u})}\right)\right\}\Big\|[w,b]\Big\|_{L^1(0,1;\R^{d_u})}.
	\end{align*}
	
	\noindent 
	\textbf{Step 2.}
	Now suppose that $t\in(1,T]$. We first show that there exists a $t^*\in(t-1,t]$ such that
	\begin{equation} \label{eq: lemma.claim.1}
	\|\*x(t^*)-\overline{\*x}\|\leqslant \|\*x-\overline{\*x}\|_{L^2(0,T;\R^{d_x})}.
	\end{equation}
	This follows by a contradiction argument -- indeed, suppose that 
	\begin{equation*}
	\|\*x(t^*)-\overline{\*x}\|>\|\*x-\overline{\*x}\|_{L^2(0,T;\R^{d_x})}.
	\end{equation*}
	for all $t^*\in(t-1,t]$, then 
	\begin{equation*}
	\int_0^T\|\*x(s)-\overline{\*x}\|^2\diff s \geqslant \int_{t-1}^t \|\*x(s)-\overline{\*x}\|^2\diff s>\int_0^T\|\*x(s)-\overline{\*x}\|^2\diff s,
	\end{equation*}
	which is a contradiction. Consequently, we know that there exists a $t^*\in(t-1,t]$ such that \eqref{eq: lemma.claim.1} holds. By integrating \eqref{eq: sigma.outside.i} in $[t^*,t]$ and using the fact that $\sigma\in\Lip(\R)$ and $\sigma(0)=0$, we see that
	\begin{align*}
	\|\*x(t)-\overline{\*x}\|&\leqslant \|\*x(t^*)-\overline{\*x}\| + c(\sigma) \int_{t^*}^t\Big(\left\|\*w(s)\right\|\left\|\*x(s)-\overline{\*x}\right\|+\left\|\overline{\*x}\right\|\left\|\*w(s)\right\|+\left\|\*b(s)\right\|\Big) \diff s.
	\end{align*}
	By virtue of Gr\"onwall's inequality, it can then be seen that
	\begin{align*}
	N^{-\sfrac{1}{2}}\|\*x(t)-\overline{\*x}\| \leqslant C\left(\|\*x(t^*)-\overline{\*x}\|+\Big\|[w,b]\Big\|_{L^1(0,T;\R^{d_u})}\right)\exp\left(C\big\|w\big\|_{L^1(0,T;\R^{d\times d})}\right)
	\end{align*}
	for some $C=C\left(\sigma,\|\overline{\*x}\|\right)>0$.
	Using \eqref{eq: lemma.claim.1}, we conclude the proof.
	\end{proof}
	
	\begin{proof}[Proof of \Cref{lem: turnpike.2}]
	
	We will make use of the notation $u:=[w,b]$ for simplicity.
	
	Let us first suppose that $T\geqslant 1+\tau_\circ$. In this case, the arguments follow precisely those presented in the proof of \Cref{thm: turnpike.BV}. 
	We just give a brief sketch of the main differences.
	Due to the controllability assumption, there exist $u^\dagger\in L^2(\tau_\circ,1+\tau_\circ;\R^{d_u})$ such that the corresponding solution $\*x^\dagger(\cdot)$ to \eqref{eq: system.tau0} on $(\tau_\circ,1+\tau_\circ)$ satisfies $\*x^\dagger(1+\tau_\circ)=\overline{\*x}$. 
	By a Gr\"onwall argument we see that for $t\in[\tau_\circ,1+\tau_\circ]$, 
	\begin{equation} \label{eq: B.3.app}
	\left\|\*x^\dagger(t)-\overline{\*x}\right\|\leqslant C_1N^{\sfrac{1}{2}}\left\|\*x^{\tau_\circ}-\overline{\*x}\right\|\exp\Big(C_1N^{\sfrac{1}{2}}\left\|\*x^{\tau_\circ}-\overline{\*x}\right\|\Big)
	\end{equation}
	holds for some constant $C_1=C_1(\sigma,\|\overline{\*x}\|,r)>0$ independent of $\tau_\circ, T, \*x^{\tau_\circ}, N$.	
	We now set 
	\begin{equation*}
	u^\aux(t):=\begin{dcases}u^\dagger(t) &\text{ in } (\tau_\circ,1+\tau_\circ)\\
	0 &\text{ in } (1+\tau_\circ, T),
	\end{dcases}
	\end{equation*}
	and we denote by $\*x^\aux(\cdot)$ the associated solution to \eqref{eq: system.tau0}. We note that $\*x^\aux(t)=\overline{\*x}$ for $t\in[1+\tau_\circ,T]$ and thus $\mathscr{E}(\*x^\aux(T))=0$. Just as for the proof of \Cref{thm: turnpike.BV},
	by using $J_{\tau_\circ,T}\left(u_T\right)\leqslant J_{\tau_\circ,T}\left(u^\aux\right)$, the fact that $\mathscr{E}(\*x^\aux(T))=0$, $\mathscr{E}\geqslant0$, \eqref{eq: B.3.app}, \eqref{eq: linear.control.cost.1}, and recalling \Cref{lem: turnpike.1}, we may conclude the proof when $T\geqslant 1+\tau_\circ$.
	
	Now suppose that $T\leqslant 1+\tau_\circ$. By making use of $J_{\tau_\circ, T}(u_T)\leqslant J_{\tau_\circ,T}(u_{1+\tau_\circ})$, where $u_{1+\tau_\circ}\in L^2(\tau_\circ,1+\tau_\circ;\R^{d_u})$ is a minimizer of $J_{\tau_\circ,1+\tau_\circ}$, and since by what precedes we have
	\begin{equation*}
	J_{\tau_\circ,T}\left(u_{1+\tau_\circ}\right)\leqslant J_{\tau_\circ, 1+\tau_\circ}\left(u_{1+\tau_\circ}\right)\leqslant C_2N\left\|\*x^{\tau_\circ}-\overline{\*x}\right\|^2\exp\Big(2C_1N^{\sfrac{1}{2}}r\Big),
	\end{equation*}
	we may combine this inequality with \Cref{lem: turnpike.1} to conclude the proof.
	\end{proof}	

	\begin{proof}[Proof of \Cref{lem: doesnt.work.in.BV}]
	We argue by contradiction. Suppose that there exist parameters $\left[w^\dagger,b^\dagger\right]\in L^2(\tau_\circ,T;\R^{d_u})$ such that
	\begin{equation*}
	J_{\tau_\circ,T}\left(w^\dagger,b^\dagger\right) < J_{\tau_\circ,T}\left(w_T,b_T\right).
	\end{equation*} 
	Set 
	\begin{equation*}
	\left[w^\aux(t),b^\aux(t)\right]:=
	\begin{dcases}
	[w_T(t),b_T(t)] &\text{ for } t\in[0,\tau_\circ)\\
	\left[w^\dagger(t),b^\dagger(t)\right] &\text{ for } t\in[\tau_\circ,T].
	\end{dcases}
	\end{equation*}
	Denoting by $\*x^\aux$ the associated neural ODE trajectory on $[0,T]$ with $\*x^\aux(0)=\*x^0$, we see that $\*x^\aux(T)=\*x^\dagger(T)$ and so $\mathscr{E}(\*x^\aux(T))=\mathscr{E}(\*x^\dagger(T))$. Consequently, 
	\begin{align*}
	J_T\left(w^\aux,b^\aux\right) &= \int_{0}^{\tau_\circ} \|\*x_T(t)-\overline{\*x}\|^2\diff t + \lambda\int_0^T\Big\|[w_T(t),b_T(t)]\Big\|^2\diff t + J_{\tau_\circ,T}\left(w^\dagger,b^\dagger\right)\\
	&< J_T(w_T,b_T),
	\end{align*}
	which contradicts the optimality of $[w_T,b_T]$. This concludes the proof.
	\end{proof} 
	
	\begin{proof}[Proof of \Cref{lem: banach.lemma}]
	Set $\eta(\tau):=\frac{\|f\|_{L^2(a,T;X)}}{\sqrt{\tau}}$. We argue by contradiction. Suppose that 
	\begin{equation*}
	\|f(t)\|_X > \eta(\tau) \hspace{1cm} \text{ for all } t\in[a,a+\tau).
	\end{equation*} 
	Then 
	\begin{equation*}
	\int_a^T \|f(t)\|^2_X\diff t \geqslant \int_a^{a+\tau}\|f(t)\|_X^2\diff t > \tau \eta(\tau)^2.
	\end{equation*}
	Whence
	\begin{equation*}
	\eta(\tau)^2 < \frac{1}{\tau}\int_a^T\|f(t)\|^2_X\diff t = \eta(\tau)^2,
	\end{equation*}
	which is a contradiction.
	\end{proof}

\bibliographystyle{apalike}
	\bibliography{refs}{}

\end{document}